\documentclass[letterpaper,11pt]{article}

\usepackage{float}

\usepackage[margin=1in]{geometry}  
\usepackage{xspace,exscale,relsize}
\usepackage{fancybox,shadow}
\usepackage{graphicx}
\usepackage{color}
\usepackage{amsthm,amsfonts}
\usepackage{amssymb}
\usepackage{authblk}
\usepackage[title]{appendix}
\usepackage{caption}
\usepackage{extarrows}

\usepackage[noadjust]{cite}

\usepackage{amsmath,mathrsfs}
	\makeatletter
	\let\over=\@@over \let\overwithdelims=\@@overwithdelims
	\let\atop=\@@atop \let\atopwithdelims=\@@atopwithdelims
  	\let\above=\@@above \let\abovewithdelims=\@@abovewithdelims
  	\makeatother
\usepackage{rotating}
\usepackage{multirow}

\usepackage{ifpdf}

\usepackage{subfigure}
\usepackage{psfrag}

\usepackage{prettyref}

\usepackage{tikz}
\usetikzlibrary{arrows}
\tikzstyle{int}=[draw, fill=blue!20, minimum size=2em]
\tikzstyle{dot}=[circle, draw, fill=blue!20, minimum size=2em]
\tikzstyle{init} = [pin edge={to-,thin,black}]

\usepackage[
            CJKbookmarks=true,
            bookmarksnumbered=true,
            bookmarksopen=true,
            colorlinks=true,
            citecolor=red,
            linkcolor=blue,
            anchorcolor=red,
            urlcolor=blue
            ]{hyperref}

\usepackage[all]{xy}

\usepackage{paralist}
\usepackage{enumitem}

\usepackage{mathtools}

\usepackage{ifthen}
\newboolean{aos}
\setboolean{aos}{FALSE}

\ifx\eqref\undefined
	\newcommand{\eqref}[1]{~(\ref{#1})}
\fi
\ifx\mod\undefined
	\def\mod{\mathop{\rm mod}}
\fi

\usepackage{bm}

\def\exp{\mathop{\rm exp}}

\def\EE{\Expect}

\def\Var{\mathrm{Var}}

\def\PP{\mathbb{P}}

\def\eqdef{\triangleq}

\def\simiid{\stackrel{iid}{\sim}}

\newcommand{\scr}[1]{\mathscr{#1}}
\newcommand{\abs}[1]{\left| #1 \right|}

\newcommand{\mymat}{M}

\makeatletter
\def\bbordermatrix#1{\begingroup \m@th
	\@tempdima 4.75\p@
	\setbox\z@\vbox{%
		\def\cr{\crcr\noalign{\kern2\p@\global\let\cr\endline}}%
		\ialign{$##$\hfil\kern2\p@\kern\@tempdima&\thinspace\hfil$##$\hfil
			&&\quad\hfil$##$\hfil\crcr
			\omit\strut\hfil\crcr\noalign{\kern-\baselineskip}%
			#1\crcr\omit\strut\cr}}%
	\setbox\tw@\vbox{\unvcopy\z@\global\setbox\@ne\lastbox}%
	\setbox\tw@\hbox{\unhbox\@ne\unskip\global\setbox\@ne\lastbox}%
	\setbox\tw@\hbox{$\kern\wd\@ne\kern-\@tempdima\left[\kern-\wd\@ne
		\global\setbox\@ne\vbox{\box\@ne\kern2\p@}%
		\vcenter{\kern-\ht\@ne\unvbox\z@\kern-\baselineskip}\,\right]$}%
	\null\;\vbox{\kern\ht\@ne\box\tw@}\endgroup}
\makeatother

\newcommand{\M}{M}

\newcommand{\ones}{\mathbf{1}}
\newcommand{\stepa}[1]{\overset{\rm (a)}{#1}}
\newcommand{\stepb}[1]{\overset{\rm (b)}{#1}}
\newcommand{\stepc}[1]{\overset{\rm (c)}{#1}}
\newcommand{\stepd}[1]{\overset{\rm (d)}{#1}}
\newcommand{\stepe}[1]{\overset{\rm (e)}{#1}}

\newcommand{\Poi}{\mathrm{Poi}}

\newcommand{\Unif}{\mathrm{Uniform}}

\newcommand{\ceil}[1]{{\left\lceil {#1} \right \rceil}}

\newcommand{\reals}{\mathbb{R}}
\newcommand{\naturals}{\mathbb{N}}
\newcommand{\integers}{\mathbb{Z}}

\newcommand{\Expect}{\mathbb{E}}
\newcommand{\expect}[1]{\mathbb{E}\left[#1\right]}
\newcommand{\Prob}{\mathbb{P}}
\newcommand{\prob}[1]{\mathbb{P}\left[#1\right]}

\newcommand{\iid}{iid\xspace}

\newcommand{\pth}[1]{\left( #1 \right)}
\newcommand{\qth}[1]{\left[ #1 \right]}
\newcommand{\sth}[1]{\left\{ #1 \right\}}

\newcommand{\eqlaw}{{\stackrel{\rm law}{=}}}
\newcommand{\iiddistr}{{\stackrel{\text{\iid}}{\sim}}}

\newcommand{\var}{\Var}

\newcommand{\Binom}{\text{Binom}}

\newcommand{\indc}[1]{{\mathbf{1}_{\left\{{#1}\right\}}}}

\definecolor{myblue}{rgb}{.8, .8, 1}
\definecolor{mathblue}{rgb}{0.2472, 0.24, 0.6} 
\definecolor{mathred}{rgb}{0.6, 0.24, 0.442893}
\definecolor{mathyellow}{rgb}{0.6, 0.547014, 0.24}

\newcommand{\calM}{{\mathcal{M}}}

\newcommand{\calO}{{\mathcal{O}}}
\newcommand{\calP}{{\mathcal{P}}}

\newcommand{\calS}{{\mathcal{S}}}

\newcommand{\calV}{{\mathcal{V}}}

\newcommand{\calX}{{\mathcal{X}}}
\newcommand{\calY}{{\mathcal{Y}}}

\newcommand{\diverge}{\to \infty}

\newrefformat{eq}{(\ref{#1})}
\newrefformat{thm}{Theorem~\ref{#1}}
\newrefformat{th}{Theorem~\ref{#1}}
\newrefformat{chap}{Chapter~\ref{#1}}
\newrefformat{sec}{Section~\ref{#1}}
\newrefformat{seca}{Section~\ref{#1}}
\newrefformat{algo}{Algorithm~\ref{#1}}
\newrefformat{fig}{Fig.~\ref{#1}}
\newrefformat{tab}{Table~\ref{#1}}
\newrefformat{rmk}{Remark~\ref{#1}}
\newrefformat{clm}{Claim~\ref{#1}}
\newrefformat{def}{Definition~\ref{#1}}
\newrefformat{cor}{Corollary~\ref{#1}}
\newrefformat{lmm}{Lemma~\ref{#1}}
\newrefformat{prop}{Proposition~\ref{#1}}
\newrefformat{pr}{Proposition~\ref{#1}}
\newrefformat{property}{Property~\ref{#1}}
\newrefformat{app}{Appendix~\ref{#1}}
\newrefformat{apx}{Appendix~\ref{#1}}
\newrefformat{ex}{Example~\ref{#1}}
\newrefformat{exer}{Exercise~\ref{#1}}
\newrefformat{soln}{Solution~\ref{#1}}

\def\unifto{\mathop{{\mskip 3mu plus 2mu minus 1mu%
	\setbox0=\hbox{$\mathchar"3221$}%
	\raise.6ex\copy0\kern-\wd0%
	\lower0.5ex\hbox{$\mathchar"3221$}}\mskip 3mu plus 2mu minus 1mu}}

\ifx\lesssim\undefined
\def\simleq{{{\mskip 3mu plus 2mu minus 1mu%
	\setbox0=\hbox{$\mathchar"013C$}%
	\raise.2ex\copy0\kern-\wd0%
	\lower0.9ex\hbox{$\mathchar"0218$}}\mskip 3mu plus 2mu minus 1mu}}
\else
\def\simleq{\lesssim}
\fi

\ifx\gtrsim\undefined
\def\simgeq{{{\mskip 3mu plus 2mu minus 1mu%
	\setbox0=\hbox{$\mathchar"013E$}%
	\raise.2ex\copy0\kern-\wd0%
	\lower0.9ex\hbox{$\mathchar"0218$}}\mskip 3mu plus 2mu minus 1mu}}
\else
\def\simgeq{\gtrsim}
\fi

\newtheorem{theorem}{Theorem}
\newtheorem{lemma}[theorem]{Lemma}
\newtheorem{corollary}[theorem]{Corollary}

\theoremstyle{definition}

\newtheorem{remark}{Remark}

\newif\ifmapx
{\catcode`/=0 \catcode`\\=12/gdef/mkillslash\#1{#1}}
\edef\jobnametmp{\expandafter\string\csname ic_apx\endcsname}
\edef\jobnameapx{\expandafter\mkillslash\jobnametmp}
\edef\jobnameexpand{\jobname}
\ifx\jobnameexpand\jobnameapx
\mapxtrue
\else
\mapxfalse
\fi

\newcommand{\Red}{\mathsf{Red}}
\newcommand{\Risk}{\mathsf{Risk}}

\renewcommand{\hat}{\widehat}
\renewcommand{\tilde}{\widetilde}

\renewcommand{\calO}{O}

\begin{document}
\ifpdf
\DeclareGraphicsExtensions{.pgf,.jpg,.pdf}
\graphicspath{{figures/}{plots/}}
\fi

\title{Optimal prediction of Markov chains with and without spectral gap}

\author{Yanjun Han, Soham Jana, and Yihong Wu\thanks{
			Y.~Han is with the Simons Institute for the Theory of Computing, University of California, Berkeley, email: \url{yjhan@berkeley.edu}.
			S.~Jana and Y.~Wu are with the Department of Statistics and Data Science, Yale University, New Haven, CT, email: \url{soham.jana@yale.edu} and \url{yihong.wu@yale.edu}.
			Y.~Wu is
supported in part by NSF Grant CCF-1900507, an NSF CAREER award CCF-1651588,
and an Alfred Sloan fellowship.}
			}
	

\maketitle

\begin{abstract}
We study the following learning problem with dependent data: Observing a trajectory of length $n$ from a stationary Markov chain with $k$ states, the goal is to predict the next state.
For  $3 \leq k \leq O(\sqrt{n})$, using techniques from universal compression, the optimal prediction risk in Kullback-Leibler divergence is shown to be $\Theta(\frac{k^2}{n}\log \frac{n}{k^2})$, in contrast to the optimal rate of $\Theta(\frac{\log \log n}{n})$ for $k=2$ previously shown in \cite{FOPS2016}. These rates, slower than the parametric rate of $O(\frac{k^2}{n})$, can be attributed to the memory in the data, as the spectral gap of the Markov chain can be arbitrarily small. To quantify the memory effect, we study irreducible reversible chains with a prescribed spectral gap. In addition to characterizing the optimal prediction risk for two states, we show that, as long as the spectral gap is not excessively small, the prediction risk in the Markov model is $O(\frac{k^2}{n})$, which coincides with that of an iid model with the same number of parameters. 
Extensions to higher-order Markov chains are also obtained.
\end{abstract}



\tableofcontents

\section{Introduction}
\label{sec:intro}



Learning distributions from samples is a central question in statistics and machine learning. 
While significant progress has been achieved in property testing and estimation based on independent and identically distributed (iid) data, for many applications, most notably natural language processing, two new challenges arise:
(a) Modeling data as independent observations fails to capture their temporal dependency;
(b) Distributions are commonly supported on a large domain whose cardinality is comparable to or even exceeds the sample size.
Continuing the progress made in \cite{FOPS2016,HOP18}, in this paper we study the following prediction problem with dependent data modeled as Markov chains.

Suppose $X_1,X_2,\dots$ is a stationary first-order Markov chain on state space $[k]\eqdef\sth{1,\dots,k}$ with unknown statistics.
Observing a trajectory $X^n\triangleq(X_1,\ldots,X_n)$, the goal is to predict the next state $X_{n+1}$ by estimating its distribution conditioned on the present data. 
We use the Kullback-Leibler (KL) divergence as the loss function:
For distributions $P=\qth{p_1,\dots,p_k},Q=\qth{q_1,\dots,q_k}$, $D(P\|Q)=\sum_{i=1}^k p_i\log \frac{p_i}{q_i}$ if $p_i=0$ whenever $q_i=0$ and $D(P\|Q)=\infty$ otherwise.
The minimax prediction risk is given by
	\begin{align}
	\Risk_{k,n} 
	&\eqdef \inf_{\hat M} \sup_{\pi,M} \Expect[D(M(\cdot|X_n) \| \hat M(\cdot|X_n))]
	=\inf_{\hat M} \sup_{\pi,M} \sum_{i=1}^k \Expect[D(M(\cdot|i) \| \hat M(\cdot|i)) \indc{X_n=i}]
	\label{eq:riskkn}
\end{align}
where the supremum is taken over all stationary distributions $\pi$ and transition matrices $M$ (row-stochastic) such that $\pi M=\pi$, 
the infimum is taken over all estimators $\hat M=\hat M(X_1,\dots,X_n)$ that are proper Markov kernels (i.e. rows sum to 1), and $M(\cdot|i)$ denotes the $i$th row of $M$.
Our main objective is to characterize this minimax risk within universal constant factors as a function of $n$ and $k$.

The prediction problem \prettyref{eq:riskkn} is distinct from the parameter estimation problem such as estimating the transition matrix \cite{bartlett1951frequency,anderson1957statistical,B61,wolfer2019minimax} or its properties \cite{csiszar2000consistency,kamath2016estimation,HJLWWY18,hsu2019mixing} in that the quantity to be estimated (conditional distribution of the next state) depends on the sample path itself. This is precisely what renders the prediction problem closely relevant to natural applications such as autocomplete and text generation. In addition, this formulation allows more flexibility with far less assumptions compared to the estimation framework. 
For example, if certain state has very small probability under the stationary distribution, consistent estimation of the transition matrix with respect to usual loss function, e.g. squared risk, may not be possible, whereas the prediction problem is unencumbered by such rare states.

In the special case of iid data, 
the prediction problem reduces to estimating the distribution in KL divergence. In this setting the optimal risk is well understood, which is known to be ${k-1\over 2n}(1+o(1))$
when $k$ is fixed and $n\diverge$ \cite{BFSS02} and $\Theta(\frac{k}{n})$ for $k=O(n)$ \cite{paninski2004variational,KOPS15}.\footnote{Here and below $\asymp, \lesssim, \gtrsim$ or $\Theta(\cdot),O(\cdot),\Omega(\cdot)$ denote equality and inequalities up to universal multiplicative constants.}
Typical in parametric models, this rate $\frac{k}{n}$ is commonly referred to the ``parametric rate'', which leads to a sample complexity that scales proportionally to the number of parameters and inverse proportionally to the desired accuracy.

In the setting of Markov chains, however, the prediction problem is much less understood especially for large state space. Recently the seminal work \cite{FOPS2016} showed the surprising result that for stationary Markov chains on two states, the optimal prediction risk satisfies
\begin{align}
	\Risk_{2,n} = \Theta\pth{\log\log n\over n},
	\label{eq:risk-mc1}
\end{align}
which has a nonparametric rate even when the problem has only two parameters. 
The follow-up work \cite{HOP18} studied general $k$-state chains and showed a lower bound of $\Omega(\frac{k\log\log n}{n})$
for uniform (not necessarily stationary) initial distribution; however, the upper bound $O(\frac{k^2\log\log n}{n})$ in \cite{HOP18} relies on implicit assumptions on mixing time such as spectral gap conditions: the proof of the upper bound for prediction (Lemma 7  in the supplement) and for estimation (Lemma 17 of the supplement) is based on Berstein-type concentration results of the empirical transition counts, which depend on spectral gap.
The following theorem resolves the optimal risk for $k$-state Markov chains:

\begin{theorem}[Optimal rates without spectral gap]
\label{thm:optimal}
There exists a universal constant $C>0$ such that for all $3 \leq k \leq \sqrt{n}/C$,
\begin{equation}
	\frac{k^2}{C n}\log\left(\frac{n}{k^2}\right) \leq 	\Risk_{k,n} \leq \frac{C k^2}{n}\log\left(\frac{n}{k^2}\right).
	\label{eq:optimal}
	\end{equation}
	Furthermore, the lower bound continues to hold even if the Markov chain is restricted to be irreducible and reversible.
\end{theorem}


\begin{remark}
\label{rmk:ach}	
The optimal prediction risk of $O(\frac{k^2}{n} \log \frac{n}{k^2})$ can be achieved by an average version of the \emph{add-one estimator} (i.e.~Laplace's rule of succession).
Given a trajectory $x^n=(x_1,\ldots,x_n)$ of length $n$, denote the transition counts (with the convention $N_i\equiv N_{ij}\equiv 0$ if $n=0,1$)
\begin{align}\label{eq:transition.count}
	N_{i}=\sum_{\ell=1}^{n-1}\indc{x_\ell=i},\quad N_{ij}=\sum_{\ell=1}^{n-1}\indc{x_\ell=i,x_{\ell+1}=j}.
\end{align}
The add-one estimator for the transition probability $M(j|i)$ is given by 
\begin{equation}
\hat M^{+1}_{x^n}(j|i) \triangleq {N_{ij}+1\over N_i+k},
\label{eq:addone}
\end{equation}
which is an additively smoothed version of the empirical frequency.
Finally, the optimal rate in 
\prettyref{eq:optimal} can be achieved by	the following estimator $\hat M$ defined as an average of add-one estimators over different sample sizes:
	\begin{equation}
		\hat M_{x^n}(x_{n+1}|x_n) \triangleq \frac{1}{n} \sum_{t=1}^{n} \hat M^{+1}_{x_{n-t+1}^n}(x_{n+1}|x_n).
		\label{eq:markov-add-1}
		\end{equation}
		In other words, we apply the add-one estimator to the most recent $t$ observations $(X_{n-t+1},\ldots,X_n)$ to predict the next $X_{n+1}$, then average over $t=1,\ldots,n$.
Such Ces\`aro-mean-type estimators have been introduced before in the density estimation literature (see, e.g., \cite{YB99}).
It remains open whether the usual add-one estimator (namely, the last term in \prettyref{eq:markov-add-1} which uses all the data) or any add-$c$ estimator for constant $c$ achieves the optimal rate.
In contrast, for two-state chains the optimal risk \prettyref{eq:risk-mc1} is attained by a hybrid strategy \cite{FOPS2016}, applying add-$c$ estimator for $c = \frac{1}{\log n}$ for trajectories with at most one transition and $c=1$ otherwise. Also note that the estimator in \eqref{eq:markov-add-1} can be computed in $O(nk)$ time. To derive this first note that given any $j\in [k]$ calculating $\hat M^{+1}_{x^{n-1}_{1}}(j|x_{n-1})$ takes $O(n)$ time and given any $M^{+1}_{x^{n-1}_{n-t+1}}(j|x_{n-1})$ we need $O(1)$ time to calculate $\hat M^{+1}_{x^{n-1}_{n-t+2}}(j|x_{n-1})$. Summing over all $j$ we get the algorithmic complexity upper bound.
\end{remark}

	%
\prettyref{thm:optimal} shows that the departure from the parametric rate of $\frac{k^2}{n}$, first discovered in \cite{FOPS2016,HOP18} for binary chains, is even more pronounced for larger state space.
As will become clear in the proof, there is some fundamental difference between two-state and three-state chains, resulting in $\Risk_{3,n} = \Theta(\frac{\log n}{n}) \gg \Risk_{2,n} = \Theta(\frac{\log \log n}{n})$. 
It is instructive to compare the sample complexity for prediction in the iid and Markov model. 
Denote by $d$ the number of parameters, which is $k-1$ for the iid case and $k(k-1)$ for Markov chains.
Define the sample complexity $n^*(d,\epsilon)$ as the smallest sample size $n$ in order to achieve a prescribed prediction risk $\epsilon$.
For $\epsilon = O(1)$, we have
\begin{equation}
n^*(d,\epsilon) \asymp
\begin{cases}
\frac{d}{\epsilon}  &  \text{iid}\\
\frac{d}{\epsilon} \log \log \frac{1}{\epsilon}  & \text{Markov with $2$ states}\\
\frac{d}{\epsilon} \log \frac{1}{\epsilon} & \text{Markov with $k\geq 3$ states}.
\end{cases}
\label{eq:samplecomplexity}
\end{equation}

At a high level, the nonparametric rates in the Markov model can be attributed to the memory in the data.
On the one hand, \prettyref{thm:optimal} as well as \prettyref{eq:risk-mc1} affirm that one can obtain meaningful prediction without imposing any mixing conditions;\footnote{To see this, it is helpful to consider the extreme case where the chain does not move at all or is periodic, in which case predicting the next state is in fact easy.}
such decoupling between learning and mixing has also been observed in other problems such as learning linear dynamics \cite{simchowitz2018learning,dean2019sample}.
On the other hand, the dependency in the data does lead to a strictly higher sample complexity than that of the iid case; in fact, the lower bound in \prettyref{thm:optimal} is proved by constructing chains with spectral gap  as small as $O(\frac{1}{n})$ (see \prettyref{sec:optimal}).
Thus, it is conceivable that with sufficiently favorable mixing conditions, the prediction risk improves over that of the worst case and, at some point, reaches the parametric rate.
To make this precise, we focus on Markov chains with a prescribed spectral gap.

It is well-known that for an irreducible and reversible chain, the transition matrix $M$ has $k$ real eigenvalues satisfying $1=\lambda_1\geq \lambda_2\geq\dots\lambda_k\geq -1$. 
The \emph{absolute spectral gap} of $M$, defined as
\begin{equation}
\gamma_*\triangleq 1-\max\sth{\abs{\lambda_i}:i\neq 1},
\label{eq:gamma-def}
\end{equation}
quantifies the memory of the Markov chain. For example, the mixing time is determined by $1/\gamma^*$ (relaxation time) up to logarithmic factors.
As extreme cases, the chain which does not move ($M$ is identity) and which is iid ($M$ is rank-one) have spectral gap equal to 0 and 1, respectively.
 We refer the reader to \cite{LevinPeres17} for more background. Note that the definition of absolute spectral gap requires irreducibility and reversibility, thus we restrict ourselves to this class of Markov chains (it is possible to use more general notions such as pseudo spectral gap to quantify the memory of the process, which is beyond the scope of the current paper).
Given $\gamma_0\in (0,1)$, define $\calM_k(\gamma_0)$ as the set of transition matrices corresponding to irreducible and reversible chains whose absolute spectral gap exceeds $\gamma_0$.
Restricting \prettyref{eq:riskkn} to this subcollection and noticing the stationary distribution here is uniquely determined by $M$, we define the corresponding minimax risk: 
\begin{align} 
\Risk_{k,n}(\gamma_0)&\triangleq\inf_{\hat \mymat}\sup_{\mymat\in \calM_k(\gamma_0)}\EE\qth{D(\mymat(\cdot|X_n)\|\hat \mymat(\cdot|X_n))}
\label{eq:risk-gamma}
\end{align}

Extending the result \prettyref{eq:risk-mc1} of \cite{FOPS2016}, the following theorem characterizes the optimal prediction risk for two-state chains with prescribed spectral gaps (the case $\gamma_0=0$ correspond to the minimax rate in \cite{FOPS2016} over all binary Markov chains):
\begin{theorem}[Spectral gap dependent rates for binary chain]
\label{thm:gamma2}
	For any $\gamma_0\in (0,1)$
	\[
	\Risk_{2,n}(\gamma_0)
	\asymp \frac{1}{n} \max\sth{1,\log\log\pth{\min\sth{n,\frac 1{\gamma_0}}}}.
	\]
\end{theorem}

\prettyref{thm:gamma2} shows that for binary chains, parametric rate $O(\frac{1}{n})$ is achievable if and only if the spectral gap is nonvanishing. While this holds for bounded state space 
(see \prettyref{cor:parametricrate} below), for large state space, it turns out that much weaker conditions on the absolute spectral gap suffice to guarantee the parametric rate $O({k^2\over n})$, achieved by the add-one estimator applied to the entire trajectory. In other words, as long as the spectral gap is not excessively small, 
the prediction risk in the Markov model behaves in the same way as that of an iid model with equal number of parameters. Similar conclusion has been established previously for the sample complexity of estimating the entropy rate of Markov chains in \cite[Theorem 1]{HJLWWY18}.
 \begin{theorem}
 	The add-one estimator in \eqref{eq:addone} achieves the following risk bound.
\begin{enumerate}[label=(\roman*)]
		\item  For any $k\geq 2$,
		$
		\Risk_{k,n}(\gamma_0)\lesssim {k^2\over n}
		$
		provided that $\gamma_0\gtrsim (\frac{\log k}{k})^{1/4}$.
	\item In addition, for $k\gtrsim (\log n)^6$, $\Risk_{k,n}(\gamma_0)\lesssim {k^2\over n}$ provided that $\gamma_0\gtrsim{(\log(n+k))^2\over k}$.
		\end{enumerate}
\label{thm:gammak}
\end{theorem}

\begin{corollary}
\label{cor:parametricrate}
	For any fixed $k\geq 2$, $\Risk_{k,n}(\gamma_0)=O(\frac{1}{n})$ if and only if $\gamma_0=\Omega(1)$.
\end{corollary}


Finally, we address the optimal prediction risk for higher-order Markov chains:
\begin{theorem}\label{thm:m-order-rate}
	There is a constant $C_m$ depending on $m$ such that for any $2\leq k\leq n^{\frac{1}{m+1}}/C_m$ and constant $m\geq 2$ the minimax prediction rate for $m^{\text{th}}$-order Markov chains with stationary initialization is 
	$
	\Theta_m\pth{{k^{m+1}\over n}\log{n\over k^{m+1}}}.
	$
\end{theorem}
Notably, for binary states, it turns out that the optimal rate $\Theta\pth{\log\log n\over n}$ for first-order Markov chains determined by \cite{FOPS2016} is something very special, as we show that for second-order chains the optimal rate is $\Theta\pth{\log n\over n}$. 

\subsection{Proof techniques}
\label{sec:technique}

The proof of \prettyref{thm:optimal} deviates from existing approaches based on concentration inequalities for Markov chains.
For instance, the standard program for analyzing the add-one estimator \prettyref{eq:addone} involves 
proving concentration of the empirical counts on their population version, namely, $N_i \approx n\pi_i$ and $N_{ij} \approx n\pi_iM(j|i)$, and bounding the risk in the atypical case by concentration inequalities, such as the Chernoff-type bounds in \cite{L98,P15}, which have been widely used 
in recent work on statistical inference with Markov chains \cite{kamath2016estimation,HJLWWY18,HOP18,hsu2019mixing,wolfer2019minimax}.
However, these concentration inequalities inevitably depends on the spectral gap of the Markov chain, 
leading to results which deteriorate as the spectral gap becomes smaller. 
For two-state chains, results free of the spectral gap are obtained in \cite{FOPS2016} using explicit joint distribution of the transition counts;
this refined analysis, however, is difficult to extend to larger state space as the probability mass function of $(N_{ij})$ is given by Whittle's formula \cite{W55} which takes an unwieldy determinantal form.

Eschewing concentration-based arguments, the crux of our proof of \prettyref{thm:optimal}, for both the upper and lower bound, revolves around the following quantity known as \emph{redundancy}:
\begin{equation}
	\Red_{k,n} \triangleq \inf_{Q_{X^n}} \sup_{P_{X^n}} D(P_{X^n} \| Q_{X^n}) = \inf_{Q_{X^n}} \sup_{P_{X^n}} \sum_{x^n} P_{X^n}(x^n) \log \frac{P_{X^n}(x^n) }{Q_{X^n}(x^n) }.
	\label{eq:red-intro}
	\end{equation}
	Here the supremum is taken over all joint distributions of stationary Markov chains	$X^n$ on $k$ states, and the infimum is over all joint distributions $Q_{X^n}$.
A central quantity which measures the minimax regret in universal compression, the redundancy \prettyref{eq:red-intro} corresponds to minimax cumulative risk (namely, the total prediction risk when the sample size ranges from 1 to $n$),
while \prettyref{eq:riskkn} is the individual minimax risk at sample size $n$ -- see \prettyref{sec:riskred} for a detailed discussion. 
We prove the following reduction between prediction risk and  redundancy:
\begin{equation}
\frac{1}{n} \Red_{k-1,n}^{\mathsf{sym}} - \frac{\log k}{n}  \lesssim \Risk_{k,n} \leq \frac{1}{n-1} \Red_{k,n}
\label{eq:riskred-approx}
\end{equation}
where $\Red^{\mathsf{sym}}$ denotes the redundancy for \emph{symmetric} Markov chains. The upper bound is standard: thanks to the convexity of the loss function and stationarity of the Markov chain, the risk of the Ces\`{a}ro-mean estimator \prettyref{eq:markov-add-1} can be upper bounded using the cumulative risk and, in turn, the redundancy. The proof of the lower bound is more involved. Given a $(k-1)$-state chain, we embed it into a larger state space by introducing a new state, such that with constant probability, the chain starts from and gets stuck at this state for a period time that is approximately uniform in $[n]$, then enters the original chain. Effectively, this scenario is equivalent to a prediction problem on $k-1$ states with a \emph{random} (approximately uniform) sample size, whose prediction risk can then be related to the cumulative risk and redundancy.
This intuition can be made precise by considering a Bayesian setting, in which the $(k-1)$-state chain is randomized according to the least favorable prior for \prettyref{eq:red-intro}, and representing the Bayes risk as conditional mutual information and applying the chain rule.

Given the above reduction in \prettyref{eq:riskred-approx}, it suffices to show both redundancies therein are on the order of $\frac{k^2}{n}\log \frac{n}{k^2}$. 
The redundancy is upper bounded by \emph{pointwise redundancy}, which replaces the average in \prettyref{eq:red-intro} by the maximum over all trajectories. 
Following \cite{davisson1981efficient,csiszar2004information}, we consider an explicit probability assignment defined by add-one smoothing and using combinatorial arguments to bound the pointwise redundancy, shown optimal by information-theoretic arguments.

The optimal spectral gap-dependent rate in \prettyref{thm:gamma2} relies on the key observation in \cite{FOPS2016} that, for binary chains, the dominating contribution to the prediction risk comes from trajectories with a single transition, for which we may apply an add-$c$ estimator with $c$ depending appropriately on the spectral gap. The lower bound is shown using a Bayesian argument similar to that of \cite[Theorem 1]{HOP18}. The proof of \prettyref{thm:gammak} relies on more delicate concentration arguments as the spectral gap is allowed to be vanishingly small. Notably, for small $k$, direct application of existing Bernstein inequalities for Markov chains in \cite{L98,P15} falls short of establishing 
the parametric rate of $O(\frac{k^2}{n})$ (see \prettyref{rmk:4thmoment} in \prettyref{sec:gammak} for details); instead, we use a fourth moment bound which turns out to be well suited for analyzing concentration of empirical counts conditional on the terminal state. 

For large $k$, we further improve the spectral gap condition using a simulation argument for Markov chains using independent samples \cite{B61,HJLWWY18}. A key step is a new concentration inequality for $D(P\|\widehat{P}_{n,k}^{+1})$, where $\widehat{P}_{n,k}^{+1}$ is the add-one estimator based on $n$ iid observations of $P$ supported on $[k]$: 
\begin{align}\label{eq:KL_concentration}
	\PP\left(D(P\|\widehat{P}_{n,k}^{+1}) \ge c\cdot \frac{k}{n} + \frac{\mathsf{polylog}(n)\cdot \sqrt{k}}{n} \right) \le \frac{1}{\mathsf{poly}(n)},
\end{align}
for some absolute constant $c>0$. Note that an application of the classical concentration inequality of McDiarmid would result in the second term being $\mathsf{polylog}(n)/\sqrt{n}$, and  \eqref{eq:KL_concentration} crucially improves this to $\mathsf{polylog}(n)\cdot \sqrt{k}/n$.
Such an improvement has been recently observed by \cite{mardia2020concentration,agrawal2020finite,guo2020chernoff} in studying the similar quantity $D(\widehat{P}_{n}\|P)$ for the (unsmoothed) empirical distribution  $\widehat{P}_{n}$; however, these results, based on either the method of types or an explicit upper bound of the moment generating function, are not directly applicable to \prettyref{eq:KL_concentration} in which the true distribution $P$ appears as the first argument in the KL divergence. 

The nonasymptotic analysis of the prediction rate for higher-order chains with large alphabets is based on a similar redundancy-based reduction as the first-order chain. However, optimal nonasymptotic redundancy bounds for higher-order chains is more challenging. 
Notably, in lower bounding redundancy, we need to bound the mutual information from below by upper bounding the squared error of certain estimators. 
As noted in \cite{tatwawadi2018minimax}, existing analysis in \cite[Sec~III]{D83} based on simple mixing conditions from \cite{parzen1962stochastic} leads to suboptimal results on large alphabets.
To bypass this issue, we show the pseudo spectral gap \cite{P15} of the transition matrix of the first-order chain $\{(X_{t+1},\dots,X_{t+m-1})\}_{t=0}^{n-m+1}$ is at least a constant. This is accomplished by a careful construction of a prior on $m^{\rm th}$-order transition matrices with $\Theta\pth{k^{m+1}}$ degrees of freedom.

\subsection{Related work}
\label{sec:related}

While the exact prediction problem studied in this paper has recently been in focus since \cite{FOPS2016,HOP18}, there exists a large body of literature on relate works. As mentioned before some of our proof strategies draws inspiration and results from the study of redundancy in universal compression, its connection to mutual information,
as well as the perspective of sequential probability assignment as prediction, dating back to \cite{D73,davisson1981efficient,rissanen1984universal,shtarkov1987univresal,R88}. Asymptotic characterization of the minimax redundancy for Markov sources, both average and pointwise, were obtained in \cite{D83,A99,jacquet2002combinatorial}, in the regime of fixed alphabet size $k$ and large sample size $n$. Non-asymptotic characterization was obtained in \cite{D83} for $n\gg k^2\log k$ and recently extended to $n\asymp k^2$ in \cite{tatwawadi2018minimax}, which further showed that the behavior of the redundancy remains unchanged even if the Markov chain is very close to being iid in terms of spectral gap $\gamma^*=1-o(1)$.

The current paper adds to a growing body of literature devoted to statistical learning with dependent data, in particular those dealing with Markov chains. 
Estimation of the transition matrix \cite{bartlett1951frequency,anderson1957statistical,B61,sinkhorn1964relationship} and testing the order of Markov chains \cite{csiszar2000consistency} have been well studied in the large-sample regime.
More recently attention has been shifted towards large state space and nonasymptotics. For example, 
\cite{wolfer2019minimax} studied the estimation of transition matrix in $\ell_\infty\to\ell_\infty$ induced norm for Markov chains with prescribed pseudo spectral gap and minimum probability mass of the stationary distribution, and determined sample complexity bounds up to logarithmic factors. Similar results have been obtained for estimating properties of Markov chains, including 
mixing time and spectral gap \cite{hsu2019mixing}, entropy rate \cite{kamath2016estimation,HJLWWY18,obremski2020complexity}, graph statistics based on random walk \cite{ben2018estimating}, as well as identity testing \cite{daskalakis2018testing,cherapanamjeri2019testing,wolfer2020minimax,fried2021identity}. 
Most of these results rely on assumptions on the Markov chains such as lower bounds on the spectral gap and the stationary distribution, which afford concentration for sample statistics of Markov chains. In contrast, one of the main contributions in this paper, in particular \prettyref{thm:optimal}, is that optimal prediction can be achieved without these assumptions, thereby providing a novel way of tackling these seemingly unavoidable issues. This is ultimately accomplished by information-theoretic and combinatorial techniques from universal compression.

\subsection{Notations and preliminaries}

For $n\in\naturals$, let $[n]\triangleq \{1,\ldots,n\}$. Denote $x^n=(x_1,\ldots,x_n)$ and $x_t^n=(x_t,\ldots,x_n)$.
The distribution of a random variable $X$ is denoted by $P_X$. In a Bayesian setting, the distribution of a parameter $\theta$ is referred to as a prior, denoted by $P_\theta$.
We recall the following definitions from information theory \cite{ckbook,cover}.
The conditional KL divergence is defined as as an average of KL divergence between conditional distributions:
	\begin{equation}
	D(P_{A|B}\|Q_{A|B}|P_B) \triangleq \Expect_{B\sim P_B} [D(P_{A|B}\|Q_{A|B})] = \int P_B(db) D(P_{A|B=b}\|Q_{A|B=b}).
	\label{eq:KLcond}
	\end{equation}
The mutual information between random variables $A$ and $B$ with joint distribution $P_{AB}$ is
$I(A;B) \triangleq
D(P_{B|A}\|P_B|P_A)$;
similarly, the conditional mutual information is defined as
\[
I(A;B|C) 
\triangleq D(P_{B|A,C}\|P_{B|C}|P_{A,C}).
\]
The following variational representation of (conditional) mutual information is well-known
\begin{equation}
I(A;B) = \min_{Q_B} D(P_{B|A}\|Q_B|P_A), \quad I(A;B|C) = \min_{Q_{B|C}} D(P_{B|A,C}\|Q_{B|C}|P_{AC}).
\label{eq:MI-var}
\end{equation}
The entropy of a discrete random variables $X$ is $H(X) \triangleq \sum_x P_X(x) \log\frac{1}{P_X(x)}$.


\subsection{Organization}
	\label{sec:org}

The rest of the paper is organized as follows.
In \prettyref{sec:riskred} we describe the general paradigm of minimax redundancy and prediction risk and their dual representation in terms of mutual information.
We give a general redundancy-based bound on the prediction risk, which, combined with redundancy bounds for Markov chains, leads to the upper bound in \prettyref{thm:optimal}.
\prettyref{sec:optimal} presents the lower bound construction, starting from three states and then extending to $k$ states.
Spectral-gap dependent risk bounds in Theorems \ref{thm:gamma2} and \ref{thm:gammak} are given in \prettyref{sec:pf-gamma}. \prettyref{sec:order-m} presents the results and proofs for $m^{\text{th}}$-order Markov chains.
\prettyref{sec:discussion} discusses the assumptions and implications of our results and related open problems.

\section{Two general paradigms}
\label{sec:riskred}


\subsection{Redundancy, prediction risk, and mutual information representation}
	\label{sec:riskred-bound}

For $n\in\naturals$, let $\calP=\{P_{X^{n+1}|\theta}: \theta\in\Theta\}$ be a collection of joint distributions parameterized by $\theta$.

\paragraph{``Compression''.} 
Consider a sample $X^n\triangleq(X_1,\ldots,X_n)$ of size $n$ drawn from $P_{X^n|\theta}$ for some unknown $\theta\in\Theta$. The \emph{redundancy} of a probability assignment (joint distribution) $Q_{X^n}$ is defined as the worst-case KL risk of fitting the joint distribution of $X^n$, namely
	\begin{equation}
	\Red(Q_{X^n}) \triangleq \sup_{\theta\in\Theta} D(P_{X^n|\theta} \| Q_{X^n}).
	\label{eq:red-Q}
	\end{equation}		
	Optimizing over $Q_{X^n}$, the minimax redundancy is defined as
	\begin{equation}
	\Red_n \triangleq \inf_{Q_{X^n}} \Red_n(Q_{X^n}),
	\label{eq:red}
	\end{equation}		
	where the infimum is over all joint distribution $Q_{X^n}$.
	This quantity can be operationalized as the redundancy (i.e.~regret) in the setting of universal data compression, that is, the excess number of bits compared to the optimal compressor of $X^n$ that knows $\theta$ \cite[Chapter 13]{cover}.
	
	The capacity-redundancy theorem (see \cite{kemperman1974shannon} for a very general result) provides the following mutual information characterization of \prettyref{eq:red}:
\begin{equation}
\Red_n = \sup_{P_\theta} I(\theta;X^n),
\label{eq:capred}
\end{equation}
where the supremum is over all distributions (priors) $P_\theta$ on $\Theta$.
In view of the variational representation \prettyref{eq:MI-var}, this result can be interpreted as a minimax theorem:
\[
\Red_n = \inf_{Q_{X^n}} \sup_{P_\theta} D(P_{X^n|\theta} \| Q_{X^n}|P_\theta) = \sup_{P_\theta} \inf_{Q_{X^n}} D(P_{X^n|\theta} \| Q_{X^n}|P_\theta).
\]
	
	Typically, for fixed model size and $n\diverge$, one expects that $\Red_n = \frac{d}{2} \log n (1+o(1)$, where 
	$d$ is the number of parameters; see \cite{rissanen1984universal} for a general theory of this type. Indeed, on a fixed alphabet of size $k$,
	we have $\Red_n = \frac{k-1}{2}\log n (1+o(1))$ for iid model \cite{D73} and 
		$\Red_n = \frac{k^m(k-1)}{2}\log n (1+o(1))$ for $m$-order Markov models \cite{trofimov1974redundancy}, 
		with more refined asymptotics shown in \cite{xie1997minimax,szpankowski2012minimax}.
	For large alphabets, nonasymptotic results have also been obtained. 
	For example, for first-order Markov model, $\Red_{n} \asymp k^2 \log\frac{n}{k^2}$ provided that $n \gtrsim k^2$ \cite{tatwawadi2018minimax}.

\paragraph{``Prediction''.}  
Consider the problem of predicting the next unseen data point $X_{n+1}$ based on the observations $X_1,\ldots,X_n$, where $(X_1,\ldots,X_{n+1})$ are jointly distributed as $P_{X^{n+1}|\theta}$ for some unknown $\theta\in\Theta$. Here, an estimator is a distribution (for $X_{n+1}$) as a function of $X^n$, which, in turn, can be written as a conditional distribution $Q_{X_{n+1}|X^{n}}$.
As such, its worst-case average risk is
\begin{equation}
	\Risk(Q_{X_{n+1}|X^{n}}) \triangleq \sup_{\theta\in\Theta} D(P_{X_{n+1}|X^{n},\theta} \| Q_{X_{n+1}|X^{n}} | P_{X^{n}|\theta}),
	\label{eq:risk-Q}
	\end{equation}	
	 where the conditional KL divergence is defined in \prettyref{eq:KLcond}.
	The minimax prediction risk is then defined as
	\begin{equation}
	\Risk_n \triangleq \inf_{Q_{X_{n+1}|X^{n}}} 	\Risk_n(Q_{X_{n+1}|X^{n}}),
	\label{eq:risk}
	\end{equation}	
	While \prettyref{eq:red} does not directly correspond to a statistical estimation problem, \prettyref{eq:risk} is exactly the familiar setting of ``density estimation'', where $Q_{X_{n+1}|X^n}$ is understood as an estimator for the distribution of the unseen $X_{n+1}$ based on the available data $X_1,\ldots,X_{n}$.
	
	In the Bayesian setting where $\theta$ is drawn from a prior $P_\theta$, the Bayes prediction risk coincides with the conditional mutual information
	as a consequence of the variational representation \prettyref{eq:MI-var}:
	\begin{equation}
	\inf_{Q_{X_{n+1}|X^{n}}} \Expect_\theta[D(P_{X_{n+1}|X^{n},\theta} \| Q_{X_{n+1}|X^{n}} | P_{X^{n}|\theta})] = I(\theta;X_{n+1}|X^n).
	\label{eq:MC_Bayes-MI}
	\end{equation}
	Furthermore, the Bayes estimator that achieves this infimum takes the following form:
	\begin{equation}
	Q_{X_{n+1}|X^{n}}^{\sf Bayes} = P_{X^{n+1}|X^n} = \frac{\int_\Theta P_{X^{n+1}|\theta}\,dP_\theta }{\int_\Theta P_{X^{n}|\theta}\,dP_\theta },
	\label{eq:MC_bayes}
	\end{equation}	
		known as the Bayes predictive density \cite{D73,liang2004exact}.
	These representations play a crucial role in the lower bound proof of \prettyref{thm:optimal}.
	Under appropriate conditions which hold for Markov models (see \prettyref{lmm:caprisk} in \prettyref{app:caprisk}), the minimax prediction risk \prettyref{eq:risk} also admits a dual representation analogous to \prettyref{eq:capred}:
	\begin{equation}
	\Risk_n = \sup_{\theta\sim\pi} I(\theta;X_{n+1}|X^n),
	\label{eq:caprisk}
	\end{equation}	
	which, in view of \prettyref{eq:MC_Bayes-MI}, show that the principle of ``minimax=worst-case Bayes'' continues to 
	hold for prediction problem in Markov models.

	The following result
	relates the redundancy and the prediction risk.

	\begin{lemma}
	\label{lmm:riskred}	
	For any model $\calP$, 
	\begin{equation}
	\Red_n\leq \sum_{t=0}^{n-1} \Risk_t.
	\label{eq:riskred1}
	\end{equation}
	In addition, suppose that each $P_{X^n|\theta}\in\calP$ is stationary and $m^{\rm th}$-order Markov. 
	Then for all $n\geq m+1$,
		\begin{equation}
	\Risk_n \leq \Risk_{n-1} \leq \frac{\Red_n}{n-m}.
	\label{eq:riskred2}
	\end{equation}
	Furthermore, for any joint distribution $Q_{X^n}$ factorizing as $Q_{X^n}=\prod_{t=1}^n Q_{X_t|X^{t-1}}$, the prediction risk of the estimator 
			\begin{equation}
		\tilde Q_{X_n|X^{n-1}}(x_n|x^{n-1}) \triangleq \frac{1}{n-m} \sum_{t=m+1}^n Q_{X_t|X^{t-1}}(x_n|x_{n-t+1}^{n-1})
		\label{eq:red-ach}
		\end{equation}
		is bounded by the redundancy of $Q_{X^n}$ as
		\begin{equation}
		\Risk(\tilde Q_{X_n|X^{n-1}}) \leq \frac{1}{n-m} \Red(Q_{X^n}).
		\label{eq:riskred-ach}
		\end{equation}
	\end{lemma}

\begin{remark}
Note that the upper bound \prettyref{eq:riskred1} on redundancy, known as the ``estimation-compression inequality'' \cite{KOPS15,FOPS2016}, holds without conditions, while the lower bound \prettyref{eq:riskred2} relies on stationarity and Markovity.
For iid data, the estimation-compression inequality is almost an equality; however, this is not the case for Markov chains, as both sides of \prettyref{eq:riskred1} differ by an unbounded factor of $\Theta(\log\log n)$ for $k=2$ and $\Theta(\log n)$ for fixed $k\geq 3$ -- see \prettyref{eq:risk-mc1} and \prettyref{thm:optimal}.
On the other hand, Markov chains with at least three states offers a rare instance where \prettyref{eq:riskred2} is tight, namely, $\Risk_{n} \asymp \frac{\Red_{n}}{n}$ (cf.~\prettyref{lmm:red-addone}).	
	
\end{remark}



	
	\begin{proof}
		The upper bound on the redundancy follows from the chain rule of KL divergence:
		\begin{equation}
		D(P_{X^n|\theta} \| Q_{X^n}) = \sum_{t=1}^n D(P_{X_{t}|X^{t-1},\theta} \| Q_{X_{t}|X^{t-1}} | P_{X^{t-1}}).
		\label{eq:KLchain}
		\end{equation}
		Thus
		\[
		\sup_{\theta\in\Theta} D(P_{X^n|\theta} \| Q_{X^n})
		\leq \sum_{t=1}^n \sup_{\theta\in\Theta} D(P_{X_{t}|X^{t-1},\theta} \| Q_{X_{t}|X^{t-1}} | P_{X^{t-1}}).
		\]
		Minimizing both sides over $Q_{X^n}$ (or equivalently, $Q_{X_{t}|X^{t-1}}$ for $t=1,\ldots,n$) yields \prettyref{eq:riskred1}.

		To upper bound the prediction risk using redundancy, fix any $Q_{X^n}$, which gives rise to $Q_{X_{t}|X^{t-1}}$ for $t=1,\ldots,n$.
		For clarity, let use denote the $t^{\rm th}$ estimator as $\hat P_t(\cdot|x^{t-1}) = Q_{X_{t}|X^{t-1}=x^{t-1}}$.
		Consider the estimator $\tilde Q_{X_n|X^{n-1}}$ defined in \prettyref{eq:red-ach}, namely,
		\begin{equation}
		\tilde Q_{X_n|X^{n-1}=x^{n-1}} \triangleq \frac{1}{n-m} \sum_{t=m+1}^n \hat P_t(\cdot|x_{n-t+1},\ldots,x_{n-1}).
		\label{eq:tildeQ}
		\end{equation}
		That is, we apply $\hat P_t$ to the most recent $t-1$ symbols prior to $X_n$ for predicting its distribution, then average over $t$.	
		We may bound the prediction risk of this estimator by redundancy as follows: Fix $\theta\in\Theta$.
		To simplify notation, we suppress the dependency of $\theta$ and write $P_{X^n|\theta} \equiv P_{X^n}$. Then
		\begin{align*}
		D(P_{X_{n}|X^{n-1}} \| \tilde Q_{X_{n}|X^{n-1}} | P_{X^{n-1}})
		\stepa{=} & ~ \expect{D\pth{P_{X_{n}|X^{n-1}_{n-m}} \Big\| \frac{1}{n} \sum_{t=1}^n \hat P_t(\cdot|X_{n-t+1}^{n-1})}} \\
		\stepb{\leq} & ~ \frac{1}{n-m} \sum_{t=m+1}^n \expect{D(P_{X_{n}|X^{n-1}_{n-m}} \| \hat P_t(\cdot|X_{n-t+1}^{n-1}))} \\
		\stepc{=} & ~ \frac{1}{n-m} \sum_{t=m+1}^n \Expect\qth{D(P_{X_{t}|X^{t-1}_{t-m}} \| \hat P_t(\cdot|X^{t-1})) } \\
		\stepd{=} & ~ \frac{1}{n-m} \sum_{t=m+1}^n D(P_{X_{t}|X^{t-1}} \| Q_{X^{t}|X^{t-1}}|P_{X^{t-1}}) \\
		\leq & ~ \frac{1}{n-m} \sum_{t=1}^n D(P_{X_{t}|X^{t-1}} \| Q_{X^{t}|X^{t-1}}|P_{X^{t-1}}) \\
		\stepe{=} & ~ \frac{1}{n-m} D(P_{X^n} \| Q_{X^n}),
		\end{align*}
		where 
		(a) uses the $m^{\rm th}$-order Markovian assumption; 
		(b) is due to the convexity of the KL divergence;
		(c) uses the crucial fact that for all $t=1,\ldots,n-1$, $(X_{n-t},\ldots,X_{n-1}) \eqlaw (X_1,\ldots,X_{t})$, thanks to stationarity;
		(d) follows from substituting $\hat P_t(\cdot|x^{t-1}) = Q_{X_{t}|X^{t-1}=x^{t-1}}$, the Markovian assumption $P_{X_{t}|X^{t-1}_{t-m}}=P_{X_{t}|X^{t-1}}$, and rewriting the expectation as conditional KL divergence;
		(e) is by the chain rule \prettyref{eq:KLchain} of KL divergence.
		Since the above holds for any $\theta\in \Theta$, the desired \prettyref{eq:riskred-ach} follows which implies that $\Risk_{n-1} \leq \frac{\Red_n}{n-m}$.
				Finally, $\Risk_{n-1} \leq \Risk_{n}$ follows from
				$
\Expect[		D(P_{X_{n+1|X_n}} \| \hat P_n(X_2^{n}))]
		= 
\Expect[		D(P_{X_{n|X_{n-1}}} \| \hat P_n(X_1^{n-1}))]
				$, since $(X_2,\ldots,X_n)$ and $(X_1,\ldots,X_{n-1})$ are equal in law.
	\end{proof}

	\begin{remark}
	\label{rmk:MI-pf}	
	Alternatively, \prettyref{lmm:riskred} also follows from the mutual information representation \prettyref{eq:capred} and \prettyref{eq:caprisk}. Indeed, by the chain rule for mutual information,
		\begin{equation}
		I(\theta;X^n) = \sum_{t=1}^n I(\theta;X_t|X^{t-1}),
		\label{eq:MIchain}
		\end{equation}
		taking the supremum over $\pi$ (the distribution of $\theta$) on both sides yields \prettyref{eq:capred}.
		For \prettyref{eq:caprisk}, it suffices to show that $I(\theta;X_t|X^{t-1})$ is decreasing in $t$:
		for any $\theta\sim \pi$,
		\begin{align*}
		I(\theta; X_{n+1}|X^n)
		= & ~ \Expect \log \frac{P_{X_{n+1}|X^n,\theta}}{P_{X_{n+1}|X^n}} = \Expect \log \frac{P_{X_{n+1}|X^n,\theta}}{P_{X_{n+1}|X^n_2}} + \underbrace{\Expect \log \frac{P_{X_{n+1}|X^n_2}}{P_{X_{n+1}|X^n}}}_{-I(X_1;X_{n+1}|X_2^n)},
		\end{align*}
		and the first term is 
		\[
		\Expect \log \frac{P_{X_{n+1}|X^n,\theta}}{P_{X_{n+1}|X^n_2}} = 
		\Expect \log \frac{P_{X_{n+1}|X^n_{n-m+1},\theta}}{P_{X_{n+1}|X^n_2}}
		= \Expect \log \frac{P_{X_{n}|X^{n-1}_{n-m},\theta}}{P_{X_{n}|X^{n-1}}}
		= I(\theta;X_{n}|X^{n-1})
		\]
		where the first and second equalities follow from the $m^{\rm th}$-order Markovity and stationarity, respectively.
		Taking supremum over $\pi$ yields $\Risk_n \leq \Risk_{n-1}$. 
		Finally, by the chain rule \prettyref{eq:MIchain}, we have 
		$I(\theta;X^n) \geq (n-m) I(\theta;X_n|X^{n-1})$, yielding
		$\Risk_{n-1} \leq \frac{\Red_n}{n-m}$.		
	\end{remark}

\subsection{Proof of the upper bound part of \prettyref{thm:optimal}}
	\label{sec:main-ub}

Specializing to first-order stationary Markov chains with $k$ states, 
we denote the redundancy and prediction risk in \prettyref{eq:red} and \prettyref{eq:risk} by $\Red_{k,n}$ and $\Risk_{k,n}$, the latter of which is precisely the quantity previously defined in \prettyref{eq:riskkn}. 
Applying \prettyref{lmm:riskred} yields $\Risk_{k,n} \leq \frac{1}{n-1}\Red_{k,n}$.
To upper bound $\Red_{k,n}$, consider the following probability assignment:
\begin{align}
Q(x_1,\cdots,x_n) = \frac{1}{k}\prod_{t=1}^{n-1} \widehat{M}_{x^t}^{+1}(x_{t+1}|x_t)
\label{eq:addone-joint}
\end{align}
where $\widehat{M}^{+1}$ is the add-one estimator defined in \prettyref{eq:addone}.
This $Q$ factorizes as $Q(x_1)=\frac{1}{k}$ and $Q(x_{t+1}|x^t) = \widehat{M}_{x^t}^{+1}(x_{t+1}|x_t)$.
The following lemma bounds the redundancy of $Q$:
\begin{lemma}
\label{lmm:red-addone}	
	$\Red(Q) \leq k(k-1)\left[\log \left(1+\frac{n-1}{k(k-1)}\right)+1\right] + \log k.
	$
\end{lemma}
Combined with \prettyref{lmm:riskred}, 
\prettyref{lmm:red-addone}	shows that 
$\Risk_{k,n} \leq C \frac{k^2}{n} \log \frac{n}{k^2}$ for all $k \leq \sqrt{n/C}$ and some universal constant $C$, 
achieved by the estimator \prettyref{eq:markov-add-1}, which is obtained by applying the rule 
\prettyref{eq:red-ach} to \prettyref{eq:addone-joint}.



It remains to show \prettyref{lmm:red-addone}. To do so, we in fact bound 
the pointwise redundancy of the add-one probability assignment \prettyref{eq:addone-joint} over all (not necessarily stationary) Markov chains on $k$ states.
The proof is similar to those of \cite[Theorems 6.3 and 6.5]{csiszar2004information}, which, in turn, follow the arguments of \cite[Sec.~III-B]{davisson1981efficient}. 
\begin{proof}
We show that for every Markov chain with transition matrix $M$ and initial distribution $\pi$, and every trajectory $(x_1,\cdots,x_n)$, it holds that
\begin{align}\label{eq:pointwise-red}
\log\frac{\pi(x_1)\prod_{t=1}^{n-1} M(x_{t+1}|x_t) }{Q(x_1,\cdots,x_n)} \le k(k-1)\left[\log \left(1+\frac{n}{k(k-1)}\right)+1\right] + \log k,
\end{align}
where we abbreviate the add-one estimator $M_{x^t}(x_{t+1}|x_t)$ defined in \prettyref{eq:addone} as $M(x_{t+1}|x_t)$.

To establish \eqref{eq:pointwise-red}, note that $Q(x_1,\cdots,x_n)$ could be equivalently expressed using the empirical counts $N_i$ and $N_{ij}$ in \prettyref{eq:transition.count} as
\begin{align*}
	Q(x_1,\cdots,x_n) = \frac{1}{k}\prod_{i=1}^k \frac{\prod_{j=1}^k N_{ij}!}{k\cdot (k+1)\cdot \cdots\cdot (N_i+k-1)}.
\end{align*}
Note that 
\[
\prod_{t=1}^{n-1} M(x_{t+1}|x_t) = \prod_{i=1}^k \prod_{j=1}^k M(j|i)^{N_{ij}} \le \prod_{i=1}^k \prod_{j=1}^k (N_{ij}/N_i)^{N_{ij}},
\]
where the inequality follows from $\sum_j \frac{N_{ij}}{N_i}\log \frac{N_{ij}/N_i }{M(j|i)}\geq 0$ for each $i$, by the nonnegativity of the KL divergence.
Therefore, we have
\begin{align}\label{eq:likelihood-ratio}
\frac{\pi(x_1)\prod_{t=1}^{n-1} M(x_{t+1}|x_t) }{Q(x_1,\cdots,x_n)} \le k\cdot \prod_{i=1}^k \frac{k\cdot (k+1)\cdot \cdots\cdot (N_i+k-1)}{N_i^{N_i}} \prod_{j=1}^k\frac{ N_{ij}^{N_{ij}} }{N_{ij}!}.
\end{align}
We claim that: for $n_1,\cdots,n_k\in \integers_+$ and $n=\sum_{i=1}^k n_i \in\naturals$, it holds that
\begin{align}\label{eq:MC_claim}
\prod_{i=1}^k \left(\frac{n_i}{n}\right)^{n_i} \le \frac{\prod_{i=1}^k n_i!}{n!},
\end{align}
with the understanding that $(\frac{0}{n})^{0} = 0!=1$.
Applying this claim to \eqref{eq:likelihood-ratio} gives
\begin{align*}
	\log \frac{\pi(x_1)\prod_{t=1}^{n-1} M(x_{t+1}|x_t) }{Q(x_1,\cdots,x_n)} &\le \log k+\sum_{i=1}^k \log \frac{k\cdot (k+1)\cdot \cdots\cdot (N_i+k-1)}{N_i!} \\
	&= \log k + \sum_{i=1}^k \sum_{\ell = 1}^{N_i} \log\left(1+\frac{k-1}{\ell}\right)\\
	&\le \log k + \sum_{i=1}^k \int_0^{N_i} \log\left(1+\frac{k-1}{x}\right)dx \\
	&= \log k + \sum_{i=1}^k \left((k-1)\log\left(1+\frac{N_i}{k-1}\right) + N_i\log\left(1+\frac{k-1}{N_i}\right) \right) \\
	&\stepa{\le} k(k-1)\log \left(1+\frac{n-1}{k(k-1)}\right) + k(k-1) + \log k,
\end{align*}
where (a) follows from the concavity of $x\mapsto \log x$, $\sum_{i=1}^k N_i=n-1$, and $\log(1+x)\le x$. 

It remains to justify \prettyref{eq:MC_claim}, which has a simple information-theoretic proof:
Let $T$ denote the collection of sequences $x^n$ in $[k]^n$ whose \emph{type} is given by $(n_1,\ldots,n_k)$. Namely, for each $x^n \in T$, 
$i$ appears exactly $n_i$ times for each $i \in[k]$.
Let $(X_1,\ldots,X_n)$ be drawn uniformly at random from the set $T$.
Then
\[
\log \frac{n!}{\prod_{i=1}^k n_i!} = H(X_1,\ldots,X_n) \stepa{\leq} \sum_{j=1}^n H(X_j) \stepb{=}  n \sum_{i=1}^k \frac{n_i}{n}\log\frac{n}{n_i},
\]
where (a) follows from the fact that the joint entropy is at most the sum of marginal entropies; (b) is because each $X_j$ is distributed as $(\frac{n_1}{n},\ldots,\frac{n_k}{n})$.
\end{proof}

\section{Optimal rates without spectral gap}
\label{sec:optimal}
In this section, we prove the lower bound part of \prettyref{thm:optimal}, which shows the optimality of the average version of the add-one estimator \prettyref{eq:red-ach}.
We first describe the lower bound construction for three-state chains, which is subsequently extended to $k$ states.

\subsection{Warmup: an $\Omega(\frac{\log n}{n})$ lower bound for three-state chains}
\label{sec:threestates}

\begin{theorem}
\label{thm:threestates}
	$
	\Risk_{3,n} = \Omega\pth{\frac{\log n}{n}}.
	$
	
\end{theorem}

To show \prettyref{thm:threestates}, consider the following one-parameter family of transition matrices:
\begin{equation}
\calM = \sth{ \M_p=\qth{\begin{matrix} 
1-\frac{2}{n} & \frac{1}{n} & \frac{1}{n}\\  
\frac{1}{n} & 1-\frac{1}{n}-p & p\\  
\frac{1}{n} & p & 1-\frac{1}{n}-p 
 \end{matrix}}\colon   0 \leq p \leq 1-\frac{1}{n}}.
\label{eq:M3}
\end{equation}
Note that each transition matrix in $\calM$ is symmetric (hence doubly stochastic), whose corresponding chain is reversible with a uniform stationary distribution and spectral gap $\Theta(\frac{1}{n})$; see \prettyref{fig:M3}.
\begin{figure}[ht]%
\centering
\begin{center}
	\begin{tikzpicture}[scale=0.6,
		roundnode/.style={circle, draw=black, thick, minimum size=5mm},
		squarednode/.style={rectangle, draw=black,semithick, minimum size=5mm},minimum size= 1mm, every edge/.style={
			draw,->,>=stealth',auto,semithick}]
		\node[roundnode] (s1) at (0,2*1.732) {1};
		\node[roundnode] (s2) at (-2,0) {2};
		\node[roundnode] (s3) at (2,0) {3};
		\path[->] (s1) edge [bend left=15] node [right] {$\frac 1n$}(s2)
		edge [bend left=15] node [right] {$\frac 1n$}(s3)
		edge [loop left,every loop/.style={looseness=12, in=60,out=120}] node [above]{$1-\frac 2n$}();
		
		\path[->] (s2) edge [bend left=15] node [left] {$\frac 1n$}(s1)
		edge [bend left=15] node [above] {$p$}(s3)
		edge [loop left,every loop/.style={looseness=12, in=150,out=210}] node {$1-\frac 1n-p$}();
		
		\path[->] (s3) edge [bend left=15] node [below] {$p$}(s2)
		edge [bend left=15] node [left] {$\frac 1n$}(s1)
		edge [loop right, every loop/.style={looseness=12, in=330,out=30}] node {$1-\frac 1n-p$}();
	\end{tikzpicture}
\end{center}
\caption{Lower bound construction for three-state chains.}%
\label{fig:M3}%
\end{figure}
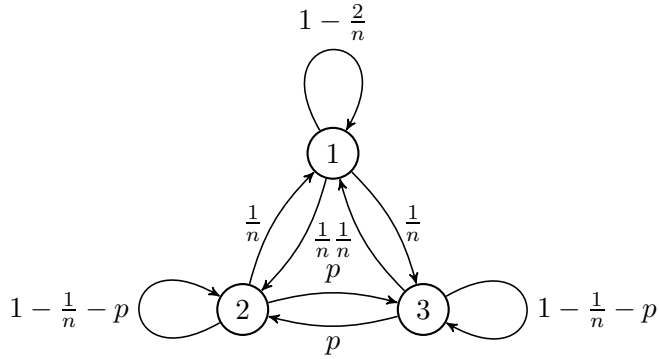

The main idea is as follows. 
Notice that by design, with constant probability, the trajectory is of the following form: The chain starts and stays at state 1 for $t$ steps, and then transitions into state 2 or 3 and never returns to state 1, where $t=1,\ldots,n-1$. Since $p$ is the single unknown parameter, the only useful observations are visits to state $2$ and $3$ and each visit entails one observation about $p$ by flipping a coin with bias roughly $p$. Thus the effective sample size for estimating $p$ is $n-t-1$ and we expect the best estimation error is of the order of $\frac{1}{n-t}$. However, $t$ is not fixed. In fact, conditioned on the trajectory is of this form, $t$ is roughly uniformly distributed between $1$ and $n-1$. As such, we anticipate the estimation error of $p$ is approximately
\[
\frac{1}{n-1}\sum_{i=1}^{n-1} \frac{1}{n-t} = \Theta\pth{\frac{\log n}{n}}.
\]
Intuitively speaking, the construction in \prettyref{fig:M3} ``embeds'' a symmetric two-state chain (with states 2 and 3) with unknown parameter $p$ into a space of three states, by adding a ``nuisance'' state 1, which effectively slows down the exploration of the useful part of the state space, so that in a trajectory of length $n$, the effective number of observations we get to make about $p$ is roughly uniformly distributed between $1$ and $n$. This explains the extra log factor in \prettyref{thm:threestates}, which actually stems from the harmonic sum in $\Expect[\frac{1}{\Unif([n])}]$.
We will fully explore this embedding idea in \prettyref{sec:kstates} to deal with larger state space.

	
	

Next we make the above intuition rigorous using a Bayesian argument.
	Let us start by recalling the following well-known lemma.
\begin{lemma}
\label{lmm:addone}	
	Let $q \sim \Unif(0,1)$. Conditioned on $q$, let $N \sim \Binom(m,q)$.
	Then the Bayes estimator of $q$ given $N$ is the ``add-one'' estimator:
	\[
	\Expect[q|N] = \frac{N+1}{m+2}
	\]
	and the Bayes risk is given by
	\[
	\Expect[(q-\Expect[q|N])^2] = \frac{1}{6(m+2)}.
	\]
	
\end{lemma}

\begin{proof}[Proof of \prettyref{thm:threestates}]
Consider the following Bayesian setting: First, we draw $p$ uniformly at random from $[0,1-\frac{1}{n}]$. Then, we generate the sample path $X^n=(X_1,\ldots,X_n)$ of a stationary (uniform) Markov chain with transition matrix $\M_p$ as defined in \prettyref{eq:M3}.
Define 
	\begin{align}\label{eq:MC_Xt-sets}
		\begin{gathered}
	\calX_t = \{x^n: x_1=\ldots=x_t=1, x_i \neq 1, i=t+1,\ldots,n\}, \quad t=1,\dots,n-1,\\
	\quad \calX = \cup_{t=1}^{n-1} \calX_t.
	\end{gathered}
	\end{align}
	Let $\mu(x^n|p) = \prob{X=x^n}$. Then 
	\begin{equation}
	\mu(x^n|p) = \frac{1}{3} \pth{1-\frac{2}{n}}^{t-1} \frac{2}{n} p^{N(x^n)}  \pth{1-\frac{1}{n}-p}^{n-t-1-N(x^n)}, \quad x^n \in \calX_t,
	\label{eq:muxp}
	\end{equation}
	where $N(x^n)$ denotes the number of transitions from state 2 to 3 or from 3 to 2.
	Then 
	\begin{align}
	\prob{X^n \in \calX_t}
	= & ~ \frac{1}{3} \pth{1-\frac{2}{n}}^{t-1}  \frac{2}{n}  \sum_{k=0}^{n-t-1}\binom{n-t-1}{k}  p^{k}  \pth{1-\frac{1}{n}-p}^{n-t-1-k} \nonumber \\
	= & ~ 	\frac{1}{3} \pth{1-\frac{2}{n}}^{t-1}  \frac{2}{n} \pth{1-\frac{1}{n}}^{n-t-1} = 	\frac{2}{3n} \pth{1-\frac{1}{n}}^{n-2}  \pth{1-\frac{1}{n-1}}^{t-1}
	\label{eq:calXt}
	\end{align}
	and hence
	\begin{align}
	\prob{X^n \in \calX} = & ~ \sum_{t=1}^{n-1} \prob{X^n \in \calX_t} = \frac{2(n-1)}{3n}  \pth{1-\frac{1}{n}}^{n-2}\pth{1-\pth{1-\frac{1}{n-1}}^{n-1}}	\label{eq:calX} 	\\
	= & ~ \frac{2(1-1/e)}{3e}  +o_n(1) \nonumber.
	\end{align}
	Consider the Bayes estimator (for estimating $p$ under the mean-squared error)
	\[
	\hat p(x^n) = \Expect[p|x^n] = \frac{\Expect[p \cdot \mu(x^n|p)]}{\Expect[\mu(x^n|p)]}.
	\]
	For $x^n\in \calX_t$, using \prettyref{eq:muxp} we have  
	\begin{align*}
	\hat p(x^n)
	= & ~ \frac{\expect{ p^{N(x^n)+1}  \pth{1-\frac{1}{n}-p}^{n-t-1-N(x^n)}}}{ \expect{ p^{N(x^n)}  \pth{1-\frac{1}{n}-p}^{n-t-1-N(x^n)}}  }, \quad p \sim \Unif\pth{0,\frac{n-1}{n}} \\
	= & ~ \frac{n-1}{n}	\frac{\expect{ U^{N(x^n)+1}  \pth{1-U}^{n-t-1-N(x^n)}}}{ \expect{ U^{N(x^n)}  \pth{1-U}^{n-t-1-N(x^n)}}  }, \quad U \sim \Unif(0,1)	\\
	= & ~ \frac{n-1}{n} \frac{N(x^n)+1}{n-t+1},
	\end{align*}
	where the last step follows from \prettyref{lmm:addone}.
	From \prettyref{eq:muxp}, we conclude that conditioned on $X^n \in \calX_t$ and on $p$, $N(X^n) \sim \Binom(n-t-1,q)$, where $q = \frac{p}{1-\frac{1}{n}} \sim \Unif(0,1)$.
	Applying \prettyref{lmm:addone} (with $m=n-t-1$ and $N=N(X^n)$), we get
	\begin{align*}
	\Expect[(p-\hat p(X^n))^2|X^n \in \calX_t]
	= & ~ \pth{\frac{n-1}{n}}^2 \Expect\qth{\pth{q- \frac{N(x^n)+1}{n-t+1}  }^2} \\
	= & ~  \pth{\frac{n-1}{n}}^2 \frac{1}{6(n-t+1)}.
	\end{align*}
	Finally, 
	note that conditioned on $X^n \in \calX$, the probability of $X^n \in \calX_t$ is close to uniform. Indeed, from \prettyref{eq:calXt} and \prettyref{eq:calX} we get
	\[
	\prob{X^n \in \calX_t| \calX}
	= \frac{1}{n-1} 
	\frac{\pth{1-\frac{1}{n-1}}^{t-1}}{1- \pth{1-\frac{1}{n-1}}^{n-1} } \geq \frac{1}{n-1} \pth{ \frac{1}{e-1} + o_n(1)}, \quad t=1,\ldots,n-1.
	\]
	Thus
	\begin{align}
	\Expect[(p-\hat p(X^n))^2 \indc{X^n \in \calX}]   
	= & ~ \prob{X^n \in \calX}  \sum_{t=1}^{n-1} \Expect[(p-\hat p(X^n))^2 | X^n \in \calX_t] \prob{X^n \in \calX_t| \calX}   \nonumber \\
	\gtrsim & ~ \frac{1}{n-1} \sum_{t=1}^{n-1} \frac{1}{n-t+1}  = \Theta\pth{\frac{\log n}{n}} \label{eq:pbayes}.
	\end{align}
	
	Finally, we relate \prettyref{eq:pbayes} formally to the minimax prediction risk under the KL divergence.
	Consider any predictor $\hat \M(\cdot|i)$ (as a function of the sample path $X$) for the $i$th row of $\M$, $i=1,2,3$.
	By Pinsker inequality, we conclude that
	\begin{align}
	D(\M(\cdot|2) \| \hat \M(\cdot|2))
 \geq  \frac{1}{2} \|\M(\cdot|2)-\hat \M(\cdot|2)\|_{\ell_1}^2 \geq \frac{1}{2} (p-\hat \M(3|2))^2
	\end{align}
	and similarly, $D(\M(\cdot|3) \| \hat \M(\cdot|3)) \geq \frac{1}{2}(p-\hat \M(2|3))^2$.
	Abbreviate $\hat \M(3|2) \equiv \hat p_2$ and $\hat \M(2|3) \equiv \hat p_3$, both functions of $X$.
	Taking expectations over both $p$ and $X$, 
	the Bayes prediction risk can be bounded as follows
	\begin{align}
	& \sum_{i=1}^3 \Expect[D( \M(\cdot|i)\|\hat \M(\cdot|i)) \indc{X_n=i} ]	\nonumber \\
	\geq & ~ \frac{1}{2} \Expect[(p-\hat p_2)^2 \indc{X_n=2} + (p-\hat p_3)^2 \indc{X_n=3}]	\nonumber \\
	\geq  & ~ \frac{1}{2} \sum_{x \in \calX} \mu(x^n)\pth{ \Expect[(p-\hat p_2)^2|X=x^n]  \indc{x_n=2} + \Expect[(p-\hat p_3)^2|X=x^n] \indc{x_n=3} }\nonumber \\
	\geq  & ~ \frac{1}{2} \sum_{x^n \in \calX} \mu(x^n) \Expect[(p-\hat p(x^n))^2|X=x^n]  (\indc{x_n=2} + \indc{x_n=3}) \nonumber \\
	= & ~ \frac{1}{2} \sum_{x^n \in \calX} \mu(x^n) \Expect[(p-\hat p(x^n))^2|X=x^n] \nonumber \\
	= & ~ \frac{1}{2} \Expect[(p-\hat p(X))^2 \indc{X\in\calX}] \overset{\prettyref{eq:pbayes}}{=} \Theta\pth{\frac{\log n}{n}}.  \nonumber
	\end{align}	
\end{proof}

\subsection{$k$-state chains}
\label{sec:kstates}

The lower bound construction for $3$-state chains in \prettyref{sec:threestates}
 can be generalized to $k$-state chains.
The high-level argument is again to augment a $(k-1)$-state chain into a $k$-state chain. Specifically, we partition the state space $[k]$ into two sets $\calS_1=\{1\}$ and $\calS_2=\{2,3,\cdots,k\}$. 
Consider a $k$-state Markov chain such that the transition probabilities from $\calS_1$ to $\calS_2$, and from $\calS_2$ to $\calS_1$, are both very small (on the order of $\Theta(1/n)$). At state $1$, the chain either stays at $1$ with probability $1-1/n$ or moves to one of the states in $\calS_2$ with equal probability $\frac{1}{n(k-1)}$; 
at each state in $\calS_2$, the chain moves to $1$ with probability $\frac{1}{n}$; otherwise, within the state subspace $\calS_2$, the chain evolves according to some symmetric transition matrix $T$. (See \prettyref{fig:Mk} in Section~\ref{subsec:chain_construction} for the precise transition diagram.)


The key feature of such a chain is as follows. Let $\calX_t$ be the event that $X_1,X_2,\cdots,X_t\in \calS_1$ and $X_{t+1},\cdots,X_n\in \calS_2$. For each $t\in [n-1]$, 
one can show that $\PP(\calX_t)\ge c/n$ for some absolute constant $c>0$. Moreover, conditioned on the event $\calX_t$, $(X_{t+1},\ldots,X_n)$ is equal in law to a stationary Markov chain $(Y_1,\cdots,Y_{n-t})$ on state space $\calS_2$ with symmetric transition matrix $T$. 
It is not hard to show that estimating $M$ and $T$ are nearly equivalent. Consider the Bayesian setting where $T$ is drawn from some prior. We have
\begin{align*}
	\inf_{\widehat{M}}\EE_{T}\left[\EE[D(M(\cdot | X_n) \| \widehat{M}(\cdot|X_n)) | \calX_t]\right] \approx \inf_{\widehat{T}}\EE_{T}\left[\EE[D(T(\cdot | Y_{n-t}) \| \widehat{T}(\cdot|Y_{n-t}))]\right] 
	=  I(T; Y_{n-t+1} | Y^{n-t}),
\end{align*}
where the last equality follows from the representation \prettyref{eq:MC_Bayes-MI} of Bayes prediction risk as conditional mutual information.
Lower bounding the minimax risk by the Bayes risk, we have
\begin{align}\label{eq:riskred3}
	\Risk_{k,n} &\ge \inf_{\widehat{M}}\EE_{T}\left[\EE[D(M(\cdot | X_n) \| \widehat{M}(\cdot|X_n))]\right] \nonumber \\
	&\ge \inf_{\widehat{M}}\sum_{t=1}^{n-1} \EE_{M}\left[\EE[D(M(\cdot | X_n) \| \widehat{M}(\cdot|X_n))|\calX_t] \cdot \PP(\calX_t)\right] \nonumber \\
	&\ge \frac{c}{n}\cdot \sum_{t=1}^{n-1} \inf_{\widehat{M}}\EE_{M}\left[\EE[D(M(\cdot | X_n) \| \widehat{M}(\cdot|X_n))|\calX_t] \right] \nonumber \\
	&\approx \frac{c}{n}\cdot \sum_{t=1}^{n-1} I(T;Y_{n-t+1}|Y^{n-t}) = \frac{c}{n}\cdot (I(T;Y^n) - I(T;Y_1)). 
\end{align}
Note that $I(T;Y_1)\le H(Y_1)\le \log(k-1)$ since $Y_1$ takes values in $\calS_2$. 
Maximizing the right hand side over the prior $P_T$ and recalling the dual representation for redundancy in \prettyref{eq:capred}, the above inequality \eqref{eq:riskred3} leads to a risk lower bound of $\Risk_{k,n} \gtrsim \frac{1}{n} (\Red_{k-1,n}^{\sf sym} - \log k)$, 
where $\Red_{k-1,n}^{\sf sym}=\sup I(T;Y_1)$ is the redundancy for \emph{symmetric} Markov chains with $k-1$ states and sample size $n$. 
Since symmetric transition matrices still have $\Theta(k^2)$ degrees of freedom, it is expected that $\Red_{k,n}^{\sf sym} \asymp k^2\log \frac{n}{k^2}$ for $n\gtrsim k^2$, so that \eqref{eq:riskred3} yields the desired lower bound $\Risk_{k,n} = \Omega(\frac{k^2}{n} \log \frac{n}{k^2})$ in Theorem \ref{thm:optimal}.

	
Next we rigorously carry out the lower bound proof sketched above: In Section \ref{subsec:chain_construction}, we explicitly construct the $k$-state chain which satisfies the desired properties in \prettyref{sec:kstates}. In Section \ref{subsec:bayes_lower_bound}, we make the steps in \eqref{eq:riskred3} precise and bound the Bayes risk from below by an appropriate mutual information. In Section \ref{subsec:prior_construction}, we choose a prior distribution on the transition probabilities and prove a lower bound on the resulting mutual information, thereby completing the proof of Theorem \ref{thm:optimal}, with the added bonus that the construction is restricted to irreducible and reversible chains.

\subsubsection{Construction of the $k$-state chain}\label{subsec:chain_construction}
We construct a $k$-state chain with the following transition probability matrix:
\begin{align}\label{eq:M_construction}
	M = \left[\begin{matrix}
		1 - \frac{1}{n} & \begin{matrix} \frac{1}{n(k-1)} & \frac{1}{n(k-1)}  & \cdots & \frac{1}{n(k-1)}  \end{matrix} \\
		\begin{matrix} 1/n \\ 1/n \\ \vdots \\ 1/n \end{matrix} & \mbox{\LARGE $\left(1-\frac{1}{n}\right)T$}
	\end{matrix}\right],
\end{align}
where $T\in \reals^{\calS_2\times \calS_2}$ is a symmetric stochastic matrix to be chosen later. The transition diagram of $M$ is shown in Figure \ref{fig:Mk}. 
One can also verify that the spectral gap of $M$ is $\Theta(\frac{1}{n})$.

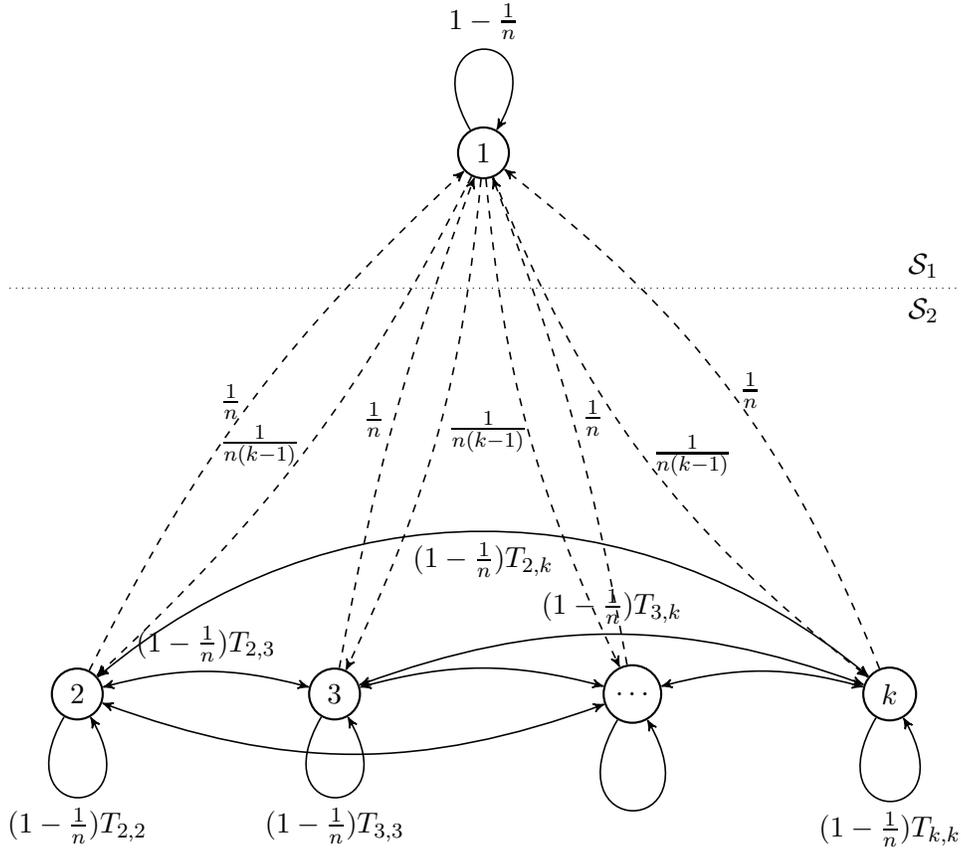
\begin{figure}[ht]%
	\centering
	\begin{center}
		\begin{tikzpicture}[scale=0.9,
			roundnode/.style={circle, draw=black, thick, minimum size=5mm},
			squarednode/.style={rectangle, draw=black,semithick, minimum size=5mm},minimum size= 1mm, every edge/.style={draw,->,>=stealth',auto,semithick}]
			\node [roundnode] (s1) at (4,6) {1}; 
			\node [roundnode] (s2) at (-2,-2) {2};
			\node [roundnode] (s3) at (1.8,-2) {3};
			\node [roundnode] (s4) at (6.2,-2) {$\ldots$};
			\node [roundnode] (s5) at (10,-2) {$k$};
			\draw [dotted] (-3,4) -- (11,4); 
			\node [above] at (10.5,4) {$\calS_1$}; \node [below] at (10.5,4) {$\calS_2$};
			
			\path (s1) edge [dashed, bend left=10] node [left] {$\frac{1}{n(k-1)}$} (s2)
			edge [dashed, bend left=10] node [right] {$\frac{1}{n(k-1)}$} (s3)
			edge [dashed, bend right=10] (s4)
			edge [dashed, bend right=15] node [right] {$\frac{1}{n(k-1)}$} (s5)
			edge [loop left,every loop/.style={looseness=12, in=60,out=120}] node [above]{$1-\frac{1}{n}$}();
			
			\path [->] (s2) edge [dashed, bend left=10] node [left] {$\frac{1}{n}$} (s1)
			edge [bend left=15, <->] node [above] {$(1-\frac{1}{n})T_{2,3}$} (s3)
			edge [bend right=20, <->] (s4)
			edge [bend left=40, <->] node [below] {$(1-\frac{1}{n})T_{2,k}$} (s5)
			edge [loop left,every loop/.style={looseness=12, in=300,out=240}] node [below]{$(1-\frac{1}{n})T_{2,2}$}();
			
			\path [->] (s3) edge [dashed, bend left=5] node [left] {$\frac{1}{n}$} (s1)
			edge [bend left=15, <->] (s4)
			edge [loop left,every loop/.style={looseness=12, in=300,out=240}] node [below]{$(1-\frac{1}{n})T_{3,3}$}()
			edge [bend left=20, <->] node [above] {$(1-\frac{1}{n})T_{3,k}$} (s5);
			
			\path [->] (s4) edge [bend right=5, dashed] node [right] {$\frac{1}{n}$} (s1)
			edge [loop left,every loop/.style={looseness=12, in=300,out=240}] ();
			
			\path [->] (s5) edge [bend right=15, dashed] node [right] {$\frac{1}{n}$} (s1)
			edge [loop left,every loop/.style={looseness=12, in=300,out=240}] node [below]{$(1-\frac{1}{n})T_{k,k}$}()
			edge [bend right=15, <->] (s4);
		\end{tikzpicture}
	\end{center}
	\caption{Lower bound construction for $k$-state chains. Solid arrows represent transitions within $\calS_1$ and $\calS_2$, and dashed arrows represent transitions between $\calS_1$ and $\calS_2$. The double-headed arrows denote transitions in both directions with equal probabilities.}%
	\label{fig:Mk}%
\end{figure}

Let $(X_1,\ldots,X_n)$ be the trajectory of a stationary Markov chain with transition matrix $M$. 
We observe the following properties:
\begin{enumerate}[label=(P\arabic*)]
	\item \label{pt:1} This Markov chain is irreducible and reversible, with stationary distribution $(\frac{1}{2},\frac{1}{2(k-1)},\cdots,\frac{1}{2(k-1)})$; 
	\item \label{pt:2} For $t\in [n-1]$, let $\calX_t$ denote the collections of trajectories $x^n$ such that $x_1,x_2,\cdots,x_t\in \calS_1$ and $x_{t+1},\cdots,x_n\in \calS_2$. Then
	\begin{align}\label{eq:Xt_prob}
		\PP(X^n\in\calX_t)&= \PP(X_1=\cdots=X_t=1)\cdot \PP(X_{t+1}\neq 1 | X_t = 1)\cdot \prod_{s=t+1}^{n-1} \PP(X_{s+1}\neq 1|X_s\neq 1) \nonumber \\
		&= \frac{1}{2}\cdot \left(1-\frac{1}{n}\right)^{t-1}\cdot \frac{1}{n}\cdot \left(1-\frac{1}{n}\right)^{n-1-t} \ge \frac{1}{2en}.
	\end{align}
	Moreover, this probability does not depend of the choice of $T$; 
	
	\item \label{pt:3} Conditioned on the event that $X^n\in\calX_t$, the trajectory $(X_{t+1},\cdots,X_n)$ has the same distribution as a length-$(n-t)$ trajectory of a stationary Markov chain with state space $\calS_2=\{2,3,\cdots,k\}$ and transition probability $T$, and the uniform initial distribution.
	Indeed,
	\begin{align*}
	\prob{X_{t+1}=x_{t+1},\ldots,X_n=x_n|X^n\in\calX_t}
	= & ~ \frac{\frac{1}{2}\cdot \left(1-\frac{1}{n}\right)^{t-1}\cdot \frac{1}{n(k-1)} \prod_{s=t+1}^{n-1} M(x_{s+1}|x_s)  }{\frac{1}{2}\cdot \left(1-\frac{1}{n}\right)^{t-1}\cdot \frac{1}{n}\cdot \left(1-\frac{1}{n}\right)^{n-1-t}} \\
	= & ~ \frac{1}{k-1} 
 \prod_{s=t+1}^{n-1} T(x_{s+1}|x_s).
	\end{align*}
\end{enumerate}

\subsubsection{Reducing the Bayes prediction risk to redundancy}\label{subsec:bayes_lower_bound}
Let $\calM_{k-1}^{\mathsf{sym}}$ be the collection of all symmetric transition matrices on state space $\calS_2=\{2,\ldots,k\}$.
Consider a Bayesian setting where the transition matrix $M$ is constructed in \eqref{eq:M_construction} and the submatrix $T$ is drawn from an arbitrary prior on $\calM_{k-1}^{\mathsf{sym}}$.
The following lemma lower bounds the Bayes prediction risk.

\begin{lemma}\label{lmm:riskred_markov}
Conditioned on $T$, let $Y^n=(Y_1,\ldots,Y_n)$ denote a stationary Markov chain on state space $\calS_2$ with transition matrix $T$ and uniform initial distribution. Then
\begin{align*}
	\inf_{\widehat{M}}\EE_{T}\left[\EE[D(M(\cdot|X_n)\| \widehat{M}(\cdot|X_n)) ]\right] \ge \frac{n-1}{2en^2}\left(I(T;Y^n) - \log (k-1)\right). 
\end{align*}
\end{lemma}

Lemma \ref{lmm:riskred_markov} is the formal statement of the inequality \eqref{eq:riskred3} presented in the proof sketch. Maximizing the lower bound over the prior on $T$ and in view of the mutual information representation \prettyref{eq:capred}, we obtain the following corollary. 

\begin{corollary}
Let $\Risk_{k,n}^{\mathsf{sym}}$ denote the minimax prediction risk for stationary irreducible and reversible Markov chains on $k$ states and $\Red_{k,n}^{\mathsf{sym}}$ the redundancy 
for stationary symmetric Markov chains on $k$ states. Then
\begin{align*}
	\Risk_{k,n}^{\sf rev} \ge \frac{n-1}{2en^2}(\Red_{k-1,n}^{\mathsf{sym}} - \log(k-1)).
\end{align*}

\end{corollary}

We make use of the properties \ref{pt:1}--\ref{pt:3} in Section \ref{subsec:chain_construction} to prove Lemma \ref{lmm:riskred_markov}.
\begin{proof}[Proof of Lemma \ref{lmm:riskred_markov}]
Recall that in the Bayesian setting, we first draw $T$ from some prior on $\calM_{k-1}^{\mathsf{sym}}$, then generate the stationary Markov chain $X^n=(X_1,\ldots,X_n)$ with state space $[k]$ and transition matrix $M$ in \prettyref{eq:M_construction}, and $(Y_1,\ldots,Y_n)$ with state space $\calS_2=\{2,\ldots,k\}$ and transition matrix $T$.

We first relate the Bayes estimator of $M$ and $T$ (given the $X$ and $Y$ chain respectively).
For clarity, we spell out the explicit dependence of the estimators on the input trajectory.
For each $t\in[n]$, denote by $\hat M_t=\hat M_t(\cdot|x^t)$ 
the Bayes estimator of $M(\cdot|x_t)$ give $X^t=x^t$, and 
$\hat T_t(\cdot|y^t)$ 
the Bayes estimator of $T(\cdot|y_t)$ give $Y^t=y^t$. 
For each $t=1,\ldots,n-1$ and for each trajectory $x^n=(1,\ldots,1,x_{t+1},\ldots,x_n) \in \calX_t$, recalling the form \prettyref{eq:MC_bayes} of the Bayes estimator, 
we have, for each $j\in\calS_2$, 
\begin{align*}
\hat M_n(j|x^n) 
= & ~ \frac{\prob{X^{n+1}=(x^n,j)}}{\prob{X^n=x^n}} \\
= & ~ \frac{\Expect[\frac{1}{2}  M(1|1)^{t-1} M(x_{t+1}|1) M(x_{t+2}|x_{t+1}) \ldots M(x_n|x_{n-1}) M(j|x_n) ]}{\Expect[\frac{1}{2} M(1|1)^{t-1} M(x_{t+1}|1) M(x_{t+2}|x_{t+1}) \ldots M(x_n|x_{n-1}) ]} \\
= & ~ \pth{1-\frac{1}{n}} \frac{\Expect[T(x_{t+2}|x_{t+1}) \ldots T(x_n|x_{n-1}) T(j|x_n) ]}{\Expect[T(x_{t+2}|x_{t+1}) \ldots T(x_n|x_{n-1}) ]} \\
= & ~ \pth{1-\frac{1}{n}} \hat T_{n-t}(j|x_{t+1}^n) ,
\end{align*}
where we used the stationary distribution of $X$ in \ref{pt:1} and the uniformity of the stationary distribution of $Y$, neither of which depends on $T$.
Furthermore, by construction in \prettyref{eq:M_construction}, $\hat M_n(1|x^n) = \frac{1}{n}$ is deterministic.
In all, we have
\begin{equation}
\hat M_n(\cdot|x^n) = \frac{1}{n} \delta_1 + \pth{1-\frac{1}{n}}\hat T_{n-t}(\cdot|x_{t+1}^n) , \quad x^n\in\calX_t.
\label{eq:MC_bayesMT}
\end{equation}
with $\delta_1$ denoting the point mass at state 1, which parallels the fact that 
\begin{equation}
M(\cdot|x) = \frac{1}{n} \delta_1 + \pth{1-\frac{1}{n}} T(\cdot|x), \quad x \in \calS_2.
\label{eq:MT}
\end{equation}

By \ref{pt:2}, each event $\{X^n\in\calX_t\}$ occurs with probability at least $1/(2en)$, and is independent of $T$. Therefore,
\begin{align}\label{eq:decomposition}
	\EE_{T}\left[\EE[D(M(\cdot|X_n)\| \widehat{M}(\cdot|X^n)) ]\right]  \ge \frac{1}{2en}\sum_{t=1}^{n-1}\EE_{T}\left[\EE[D(M(\cdot|X_n)\| \widehat{M}(\cdot|X^n)) | X^n\in\calX_t ]\right].  
\end{align}
By \ref{pt:3},  the conditional joint law of $(T,X_{t+1},\ldots,X_n)$ on the event $\{X^n\in\calX_t\}$ is the same as the joint law of $(T,Y_{1},\ldots,Y_{n-t})$. 
Thus, we may express the Bayes prediction risk in the $X$ chain as 
\begin{align}\label{eq:reduction_M_N}
	\EE_{T}\left[\EE[D(M(\cdot|X_n)\| \widehat{M}(\cdot|X^n)) | X^n\in\calX_t ]\right] 
&	\stepa{=} \left(1-\frac{1}{n}\right)\cdot \EE_{T}\left[\EE[D(T(\cdot|Y_{n-t}) \| \widehat{T}(\cdot|Y^{n-t}))]\right] \nonumber \\
	& \stepb{=} \left(1-\frac{1}{n}\right)\cdot I(T;Y_{n-t+1}|Y^{n-t}),
\end{align}
where (a) follows from \prettyref{eq:MC_bayesMT}, \prettyref{eq:MT}, and the fact that for distributions $P,Q$ supported on $\calS_2$,
$D(\epsilon \delta_1 + (1-\epsilon) P \|\epsilon \delta_1 + (1-\epsilon) Q)=(1-\epsilon) D(P\|Q)$;
(b) is the mutual information representation \prettyref{eq:MC_Bayes-MI} of the Bayes prediction risk.
Finally, the lemma follows from \eqref{eq:decomposition}, \eqref{eq:reduction_M_N}, and the chain rule
\begin{align*}
	\sum_{t=1}^{n-1} I(T;Y_{n-t+1}|Y^{n-t}) = I(T;Y^{n}) - I(T;Y_1) \ge  I(T;Y^{n}) - \log(k-1),
\end{align*} 
as $I(T;Y_1)\le H(Y_1)\le \log(k-1)$. 
\end{proof}

\subsubsection{Prior construction and lower bounding the mutual information}\label{subsec:prior_construction}
In view of Lemma \ref{lmm:riskred_markov}, it remains to find a prior on $\calM_{k-1}^{\mathsf{sym}}$ for $T$, such that the mutual information $I(T;Y^n)$ is large. 
We make use of the connection identified in \cite{davisson1981efficient,D83,rissanen1984universal} between estimation error and mutual information (see also
\cite[Theorem 7.1]{csiszar2004information} for a self-contained exposition). 
To lower the mutual information, a key step is to find a good estimator $\widehat{T}(Y^n)$ of $T$. This is carried out in the following lemma.

\begin{lemma}\label{lmm:L2_upper}
In the setting of \prettyref{lmm:riskred_markov}, suppose that $T\in \calM_{k}^{\mathsf{sym}}$ with $T_{ij}\in [\frac{1}{2k},\frac{3}{2k}]$ for all $i,j\in [k]$. Then there is an estimator $\widehat{T}$ based on $Y^n$ such that
\begin{align*}
	\EE[\|\widehat{T} - T \|_{\mathsf{F}}^2] \le \frac{16k^2}{n-1},
\end{align*}
where $\|\widehat{T} - T \|_{\mathsf{F}} = \sqrt{\sum_{ij} (\widehat{T}_{ij} - T_{ij})^2}$ denotes the Frobenius norm.
\end{lemma}

We show how Lemma \ref{lmm:L2_upper} leads to the desired lower bound on the mutual information $I(T;Y^n)$. Since $k\ge 3$, we may assume that $k-1=2k_0$ is an even integer. Consider the following prior distribution $\pi$ on $T$: let $u=(u_{i,j})_{i,j\in [k_0], i\le j}$ be iid and uniformly distributed in $[1/(4k_0),3/(4k_0)]$, and $u_{i,j} = u_{j,i}$ for $i>j$. Let the transition matrix $T$ be given by
\begin{align}
T_{2i-1,2j-1} = T_{2i,2j} = u_{i,j}, \quad T_{2i-1,2j} = T_{2i,2j-1} = \frac{1}{k_0} - u_{i,j}, \quad \forall i,j\in [k]. 
\label{eq:Tu}
\end{align}
It is easy to verify that $T$ is symmetric and a stochastic matrix, and each entry of $T$ is supported in the interval $[1/(4k_0), 3/(4k_0)]$. Since $2k_0 = k-1$, the condition of Lemma \ref{lmm:L2_upper} is fulfilled, so there exist estimators $\widehat{T}(Y^n)$ and $\widehat{u}(Y^n)$ such that
\begin{align}
\EE[\|\widehat{u}(Y^n) - u\|_2^2] \le \EE[\|\widehat{T}(Y^n) - T \|_{\mathsf{F}}^2] \le \frac{64k_0^2}{n-1}. 
\label{eq:MSEu}
\end{align}
Here and below, we identify $u$ and $\hat u$ as $\frac{k_0(k_0+1)}{2}$-dimensional vectors.

Let $h(X) = \int -f_X(x)\log f_X(x)dx$ denote the differential entropy of a continuous random vector $X$ with density $f_X$ w.r.t the Lebesgue measure 
and $h(X|Y)=\int -f_{XY}(xy)\log f_{X|Y}(x|y)dxdy$ the conditional differential entropy (cf.~e.g.~\cite{cover}). Then
\begin{align}\label{eq:diff_entropy}
	h(u) = \sum_{i,j\in [k_0], i\le j}h(u_{i,j}) = -\frac{k_0(k_0+1)}{2}\log(2k_0). 
\end{align}
Then
\begin{align*}
	I(T;Y^n)
	& \stepa{=} I(u;Y^n)\\
	& \stepb{\geq} I(u;\hat u(Y^n)) = h(u)-h(u|\hat u(Y^n))\\
	& \stepc{\geq} h(u)-h(u-\hat u(Y^n))\\
& \stepd{\geq} \frac{k_0(k_0+1)}{4}\log\left(\frac{n-1}{1024\pi ek_0^2}\right) \ge \frac{k^2}{16}\log\left(\frac{n-1}{256\pi ek^2}\right).
\end{align*}
where 
(a) is because $u$ and $T$ are in one-to-one correspondence by \prettyref{eq:Tu};
(b) follows from the data processing inequality;
(c) is because $h(\cdot)$ is translation invariant and concave;
(d) follows from the maximum entropy principle \cite{cover}: 
$h(u-\hat u(Y^n)) \leq \frac{k_0(k_0+1)}{4}\log\left(\frac{2\pi e}{k_0(k_0+1)/2}\cdot \EE[\|\widehat{u}(Y^n) - u\|_2^2] \right)$, which in turn is bounded by \prettyref{eq:MSEu}.
Plugging this lower bound into Lemma \ref{lmm:riskred_markov} completes the lower bound proof of Theorem \ref{thm:optimal}. 

\begin{proof}[Proof of Lemma \ref{lmm:L2_upper}]
Since $T$ is symmetric, the stationary distribution is uniform, and there is a one-to-one correspondence between the joint distribution of $(Y_1,Y_2)$ and the transition probabilities. Motivated by this observation, consider the following estimator $\widehat{T}$: for $i,j\in [k]$, let 
\begin{align*}
	\widehat{T}_{ij} = k\cdot \frac{\sum_{t=1}^n \indc{Y_t=i, Y_{t+1}=j} }{n-1}. 
\end{align*}
Clearly $\EE[\widehat{T}_{ij}]=k\cdot \PP(Y_1=i,Y_2=j)= T_{ij}$. The following variance bound is shown in \cite[Lemma 7, Lemma 8]{tatwawadi2018minimax} using the concentration inequality of \cite{P15}:
\begin{align*}
	\var(\widehat{T}_{ij}) \le k^2\cdot \frac{8T_{ij}k^{-1}}{\gamma_*(T)(n-1)}, 
\end{align*}
where $\gamma_*(T)$ is the absolute spectral gap of $T$ defined in \prettyref{eq:gamma-def}. Note that $T = k^{-1} \mathbf{J} + \Delta$,
where $\mathbf{J}$ is the all-one matrix and each entry of $\Delta$ lying in $[-1/(2k),1/(2k)]$. Thus the spectral radius of $\Delta$ is at most $1/2$ and thus $\gamma_*(T)\ge 1/2$. Consequently, we have
\begin{align*}
	\EE[\|\widehat{T} - T \|_{\mathsf{ F}}^2] = \sum_{i,j\in [k]} 	\var(\widehat{T}_{ij}) \le  \sum_{i,j\in [k]} \frac{16kT_{ij}}{n-1} = \frac{16k^2}{n-1}, 
\end{align*}
completing the proof. 
\end{proof}

%
%

\section{Spectral gap-dependent risk bounds}
\label{sec:pf-gamma}
	\subsection{Two states}\label{sec:lower bound}
		To show \prettyref{thm:gamma2}, 
		let us prove a refined version. In addition to the absolute spectral gap defined in \prettyref{eq:gamma-def}, define the spectral gap
		\begin{equation}
\gamma \triangleq 1-\lambda_2
\label{eq:gamma-def1}
\end{equation}
and $\calM_k'(\gamma_0)$ the collection of transition matrices whose spectral gap exceeds $\gamma_0$.
Paralleling $\Risk_{k,n}(\gamma_0)$ defined in 
\prettyref{eq:risk-gamma}, define $\Risk_{k,n}'(\gamma_0)$ as the minimax prediction risk restricted to $M\in \calM_k'(\gamma_0)$
Since $\gamma\geq \gamma^*$, we have 
$\calM_k(\gamma_0)\subseteq \calM_k'(\gamma_0)$ and hence 
$\Risk_{k,n}'(\gamma_0) \geq \Risk_{k,n}(\gamma_0)$.
Nevertheless, the next result shows that for $k=2$ they have the same rate:

		\begin{theorem}[Spectral gap dependent rates for binary chain]
\label{thm:gamma2-refined}
	For any $\gamma_0\in (0,1)$
	\[
	\Risk_{2,n}(\gamma_0)
	\asymp
	\Risk_{2,n}'(\gamma_0)
	\asymp \frac{1}{n} \max\sth{1,\log\log\pth{\min\sth{n,\frac 1{\gamma_0}}}}.
	\]
\end{theorem}

We first prove the upper bound on $\Risk_{2,n}'$. Note that it is enough to show
			\begin{align}
			\Risk_{2,n}'(\gamma_0)
			\lesssim {\log\log\pth{1/ \gamma_0}\over n},\quad 
		\text{if } n^{-0.9}\leq \gamma_0\leq e^{-e^{5}}.
		\label{eq:gamma2a}
			\end{align}
		Indeed,	for any $\gamma_0\leq n^{-0.9}$, the upper bound $\calO\pth{\log\log n/n}$ proven in \cite{FOPS2016}, which does not depend on the spectral gap, suffices; for any $\gamma_0> e^{-e^{5}}$,  by monotonicity we can use the upper bound $\Risk_{2,n}'(e^{-e^{5}})$. 
		
		We now define an estimator that achieves \prettyref{eq:gamma2a}. Following \cite{FOPS2016}, consider trajectories with a single transition, namely, $\sth{2^{n-\ell}1^\ell,1^{n-\ell}2^\ell:1\leq \ell\leq n-1}$, where $2^{n-\ell}1^{\ell}$ denotes the trajectory $(x_1,\cdots,x_n)$ with $x_1=\cdots=x_{n-\ell}=2$ and $x_{n-\ell+1}=\cdots=x_n=1$. 
	We refer to this type of $x^n$ as \emph{step sequences}. For all non-step sequences $x^n$, we apply the add-$\frac 12$ estimator similar to \prettyref{eq:addone}, namely
			\begin{align*}
				\widehat{M}_{x^n}(j|i) = \frac{N_{ij}+\frac 12}{N_i+1}, \qquad i,j\in \{1,2\},
			\end{align*}
			where the empirical counts $N_i$ and $N_{ij}$ are defined in \prettyref{eq:transition.count};
			for step sequences of the form $2^{n-\ell}1^{\ell}$, we estimate by
			\begin{align}\label{eq:z.ell.est}
				{\hat \mymat_\ell(2|1)}={1/(\ell \log(1/\gamma_0))}, \quad
				{\hat \mymat_\ell(1|1)}=1-{\hat \mymat_\ell(2|1)}.
			\end{align}
			The other type of step sequences $1^{n-\ell}2^{\ell}$ are dealt with by symmetry.
			
			Due to symmetry it suffices to analyze the risk for sequences ending in 1. The risk of add-$\frac 12$ estimator for the non-step sequence $1^n$ is bounded as
			\begin{align*}
				\EE\qth{\indc{X^n=1^n}D({\mymat(\cdot|1)}\|{\hat \mymat_{1^n}(\cdot|1)})}
				&=P_{X^n}(1^n)\sth{M(2|1)\log\pth{ \frac{M(2|1)}{1/(2n)} }+M(1|1)\log\pth{M(1|1)\over (n-\frac 12)/n}}
				\nonumber\\
				&\leq (1-M(2|1))^{n-1}\sth{2M(2|1)^2n+\log\pth{{n\over n-\frac 12}}}
				\lesssim \frac 1n.
			\end{align*}
			where the last step followed by using $(1-x)^{n-1}x^2\leq n^{-2}$ with $x=M(2|1)$ and $\log x\leq x-1$. From \cite[Lemma 7,8]{FOPS2016} we have that the total risk of other non-step sequences is bounded from above by $\calO\pth{\frac 1n}$ and hence it is enough to analyze the risk for step sequences, and further by symmetry, those in $\sth{2^{n-\ell}1^\ell:1\leq \ell\leq n-1}$. The desired upper bound \prettyref{eq:gamma2a} then follows from \prettyref{lmm:stepbound} next. 
			
		\begin{lemma}\label{lmm:stepbound} 
			For any $n^{-0.9}\leq \gamma_0\leq e^{-e^{5}}$, $\hat \mymat_\ell(\cdot|1)$ in \eqref{eq:z.ell.est} satisfies 
		$$
		\sup_{M\in \calM'_2(\gamma_0)}\sum_{\ell=1}^{n-1}\EE\qth{\indc{X^n=2^{n-\ell}1^\ell}D({\mymat(\cdot|1)}\|{\hat \mymat_\ell(\cdot|1)})}
		\lesssim  {\log\log(1/\gamma_0)\over n}.
		$$
	\end{lemma}
	
	\begin{proof}\label{app:stepbound}
		For each $\ell$ using $\log\pth{ 1\over {1-x}}\leq  2x,x\leq \frac 12$ with $x=\frac 1{\ell\log(1/\gamma_0)}$, 
		\begin{align}
			D({\mymat(\cdot|1)}\|{\hat \mymat_\ell(\cdot|1)})
			&= {M(1|1)\log \pth{M(1|1)\over 1-\frac 1{\ell{\log(1/\gamma_0)}}}+{\mymat(2|1)}\log \pth{{\mymat(2|1)}\ell{\log(1/\gamma_0)}}}
			\nonumber\\
			&\lesssim {1\over \ell {\log(1/\gamma_0)}}+{\mymat(2|1)}\log(M(2|1)\ell)+{\mymat(2|1)}\log {{\log(1/\gamma_0)}}
			\nonumber\\
			&\le {1\over \ell {\log(1/\gamma_0)}}+\mymat(2|1)\log_+(\mymat(2|1)\ell) + \mymat(2|1) {\log\log(1/\gamma_0)}\label{eq:appgub5},
		\end{align}
		where we define $\log_+(x) = \max\{1,\log x\}$. 
		Recall the following Chebyshev's sum inequality: for $a_1\le a_2\le \cdots\le a_n$ and $b_1\ge b_2\ge \cdots\ge b_n$, it holds that 
		\begin{align*}
			\sum_{i=1}^n a_ib_i \le \frac{1}{n}\pth{\sum_{i=1}^n a_i}\pth{\sum_{i=1}^n b_i}. 
		\end{align*}
		The following inequalities are thus direct corollaries: for $x,y \in [0,1]$, 
		\begin{align}
			\sum_{\ell=1}^{n-1} x(1-x)^{n-\ell-1}y(1-y)^{\ell-1} &\le \frac{1}{n-1}\pth{\sum_{\ell=1}^{n-1} x(1-x)^{n-\ell-1}}\pth{\sum_{\ell=1}^{n-1} y(1-y)^{\ell-1}}\nonumber\\
			& \le \frac{1}{n-1}, \label{eq:chebyshev-1} \\
			 \sum_{\ell=1}^{n-1} x(1-x)^{n-\ell-1}y(1-y)^{\ell-1}\log_+(\ell y) &\le \frac{1}{n-1}\pth{\sum_{\ell=1}^{n-1} x(1-x)^{n-\ell-1}}\pth{\sum_{\ell=1}^{n-1} y(1-y)^{\ell-1}\log_+(\ell y)} \nonumber \\
			 &\le \frac{1}{n-1}\sum_{\ell=1}^{n-1}y(1-y)^{\ell-1}(1+\ell y) \le \frac{2}{n-1}, \label{eq:chebyshev-2}
		\end{align}
		where in \eqref{eq:chebyshev-2} we need to verify that $\ell\mapsto y(1-y)^{\ell-1}\log_+(\ell y)$ is non-increasing. To verify it, w.l.o.g. we may assume that $(\ell+1)y\ge e$, and therefore
		\begin{align*}
			\frac{y(1-y)^{\ell}\log_+((\ell+1)y )}{y(1-y)^{\ell-1}\log_+(\ell y)} &= \frac{(1-y)\log((\ell+1)y)}{\log_+(\ell y)} \le \pth{1-\frac{e}{\ell+1}}\pth{1 + \frac{\log(1+1/\ell)}{\log_+(\ell y)} } \\
			&\le \pth{1-\frac{e}{\ell+1}}\pth{1+\frac{1}{\ell}} < 1 + \frac{1}{\ell} - \frac{e}{\ell+1} < 1. 
		\end{align*}

		Therefore, 
	   \begin{align}
	   	&\sum_{\ell=1}^{n-1}\EE\qth{\indc{X^n=2^{n-\ell}1^\ell}D(\mymat(\cdot|1)\|\hat \mymat_\ell(\cdot|1))} \nonumber\\
	   	&\le \sum_{\ell=1}^{n-1}M(2|2)^{n-\ell-1}M(1|2)
	   	M(1|1)^{\ell-1}D(\mymat(\cdot|1)\|\hat \mymat_\ell(\cdot|1))\nonumber\\
	   	&\stepa{\lesssim} \sum_{\ell=1}^{n-1}M(2|2)^{n-\ell-1}M(1|2)M(1|1)^{\ell-1}
	   	\pth{\frac {1}{\ell \log(1/\gamma_0)}+M(2|1)\log_+(M(2|1)\ell) + M(2|1)\log\log(1/\gamma_0)} \nonumber\\
	   	&\stepb{\le} \sum_{\ell=1}^{n-1} \frac{M(2|2)^{n-\ell-1}M(1|2)M(1|1)^{\ell-1}}{\ell \log(1/\gamma_0)} + \frac{2+\log\log(1/\gamma_0)}{n-1}, \label{eq:appgub6}
	   \end{align}
   		where (a) is due to \eqref{eq:appgub5}, (b) follows from \eqref{eq:chebyshev-1} and \eqref{eq:chebyshev-2} applied to $x=M(1|2), y=M(2|1)$. To deal with the remaining sum, we distinguish into two cases. Sticking to the above definitions of $x$ and $y$, if $y > \gamma_0/2$, then
   		\begin{align*}
   			\sum_{\ell=1}^{n-1}\frac{x(1-x)^{n-\ell-1}(1-y)^{\ell-1}}{\ell} \le \frac{1}{n-1}\pth{\sum_{\ell=1}^{n-1} x(1-x)^{n-\ell-1} }\pth{\sum_{\ell=1}^{n-1} \frac{(1-y)^{\ell-1}}{\ell}} \le \frac{\log(2/\gamma_0)}{n-1},
   		\end{align*}
   		where the last step has used that $\sum_{\ell=1}^\infty t^{\ell-1}/\ell = \log(1/(1-t))$ for $|t|<1$. If $y\le \gamma_0/2$, notice that for two-state chain the 
spectral gap is given explicitly by 
$\gamma=M(1|2)+M(2|1)=x+y$, so that the assumption $\gamma\ge \gamma_0$ implies that $x\ge \gamma_0/2$. In this case, 
  			\begin{align*}
  	\sum_{\ell=1}^{n-1}\frac{x(1-x)^{n-\ell-1}(1-y)^{\ell-1}}{\ell} &\le \sum_{\ell < n/2} (1-x)^{n/2-1} + \sum_{\ell\ge n/2} \frac{x(1-x)^{n-\ell-1}}{n/2} \\
  	&\le \frac{n}{2}e^{-(n/2-1)\gamma_0} + \frac{2}{n} \lesssim \frac{1}{n},
  		\end{align*}
  	thanks to the assumption $\gamma_0\ge n^{-0.9}$. Therefore, in both cases, the first term in \eqref{eq:appgub6} is $O(1/n)$, as desired. 	
	\end{proof}
			
			Next we prove the lower bound on $\Risk_{2,n}$. It is enough to show that
			$
			\Risk_{2,n}(\gamma_0)
			\gtrsim \frac{1}{n}\log\log\pth{1/ \gamma_0}$
			for $n^{-1}\leq \gamma_0\leq e^{-e^{5}}$.
			Indeed, for $\gamma_0\geq e^{-e^{5}}$, we can apply the result in the \iid setting (see, e.g., \cite{BFSS02}), in which the absolute spectral gap is 1, to obtain the usual parametric-rate lower bound $\Omega\pth{\frac 1n}$;
			for $\gamma_0 <n^{-1}$, we simply bound $\Risk_{2,n}(\gamma_0)$ from below by $\Risk_{2,n}(n^{-1})$.
			Define \begin{align}\label{eq:alpha.beta}
			\alpha=\log(1/\gamma_0),\quad\beta=\left\lceil {\alpha\over 5\log\alpha} \right\rceil,
			\end{align} 
			and consider the prior distribution
			\begin{align}\label{eq:twostate.prior}
			\scr M=\Unif(\calM),
			\quad \calM&=\sth{\mymat:{\mymat(1|2)}=\frac 1n,{\mymat(2|1)}={1\over \alpha^m}: m\in \naturals \cap\pth{\beta,5\beta}}.
			\end{align}
			Then the lower bound part of \prettyref{thm:gamma2} follows from the next lemma.

		\begin{lemma}\label{lmm:bayes.risk}
			Assume that $n^{-0.9}\leq \gamma_0\leq e^{-e^{5}}$. Then
			\begin{enumerate}[label=(\roman*)]
				\item $\gamma_*> \gamma_0$ for each $M\in\calM$;
				\item the Bayes risk with respect to the prior $\scr M$ is at least  $\Omega\pth{\log\log(1/\gamma_0)\over n}$.
			\end{enumerate}
		\end{lemma}
		\begin{proof}
			Part (i) follows by noting that absolute spectral gap for any two states matrix $M$ is $1-\abs{1-{\mymat(2|1)}-{\mymat(1|2)}}$ and for any $M\in \cal M$, $M(2|1)\in  \pth{\alpha^{-5\beta},\alpha^{-\beta}}\subseteq(\gamma_0,\gamma_0^{1/5})\subseteq (\gamma_0,1/2) $ which guarantees
			$
				\gamma_*=M(1|2)+M(2|1)>\gamma_0.
			$
			
			To show part (ii) we lower bound the Bayes risk when the observed trajectory $X^n$ is a step sequence in $\sth{2^{n-\ell}1^\ell:1\leq \ell\leq n-1}$. Our argument closely follows that of \cite[Theorem 1]{HOP18}. Since $\gamma_0\ge n^{-1}$, for each $M\in \calM$, the corresponding stationary distribution $\pi$ satisfies
			\begin{align*}
				\pi_2 = \frac{M(2|1)}{M(2|1)+M(1|2)} \ge \frac{1}{2}. 
			\end{align*}
		
		Denote by $\Risk(\scr M)$ the Bayes risk with respect to the prior $\scr M$ and by ${\hat \mymat^{\mathsf{B}}_\ell(\cdot|1)}$ 
		the Bayes estimator for prior $\scr M$ given $X^n=2^{n-\ell}1^\ell$. Note that 
		\begin{equation}
		\prob{X^n = 2^{n-\ell}1^\ell} = \pi_2 \pth{1-\frac{1}{n}}^{n-\ell-1} \frac{1}{n} M(1|1)^{\ell-1} \geq \frac{1}{2en} M(1|1)^{\ell-1}.
		\label{eq:probstepseq}
		\end{equation}
		Then
		\begin{align}
			\Risk(\scr M)
			&\geq \EE_{M\sim \scr M}\qth{\sum_{\ell=1}^{n-1}\EE\qth{\indc{X^n=2^{n-\ell}1^\ell}
					D({\mymat(\cdot|1)}\|{\hat \mymat^{\mathsf{B}}_\ell(\cdot|1)})}}
			\nonumber\\
			&\geq \EE_{M\sim \scr M}\qth{\sum_{\ell=1}^{n-1}{M(1|1)^{\ell-1}\over 2en}
			D({\mymat(\cdot|1)}\|{\hat \mymat^{\mathsf{B}}_\ell(\cdot|1)})} \nonumber \\
		&= \frac{1}{2en}\sum_{\ell=1}^{n-1}\EE_{M\sim \scr M} \qth{M(1|1)^{\ell-1}
			D({\mymat(\cdot|1)}\|{\hat \mymat^{\mathsf{B}}_\ell(\cdot|1)})}. \label{eq:MC_bayes_risk_gamma}
		\end{align}
		Recalling the general form of the Bayes estimator in \prettyref{eq:MC_bayes} and in view of \prettyref{eq:probstepseq}, we get
		\begin{align}\label{eq:MC_bayes_est}
			\widehat{M}_{\ell}^{\mathsf{B}}(2|1) = \frac{\EE_{M\sim \scr M}[M(1|1)^{\ell-1}M(2|1)]}{\EE_{M\sim \scr M}[M(1|1)^{\ell-1}]}, \quad \widehat{M}_{\ell}^{\mathsf{B}}(1|1) = 1 - \widehat{M}_{\ell}^{\mathsf{B}}(2|1). 
		\end{align}
		Plugging \eqref{eq:MC_bayes_est} into \eqref{eq:MC_bayes_risk_gamma}, and using 
			\begin{align*}
			D((x,1-x)\|(y,1-y))
			= x\log{x\over y}+(1-x) \log{1-x\over 1-y}
			\geq x\max\sth{0,\log{x\over y}-1}, 
		\end{align*}
		we arrive at the following lower bound for the Bayes risk: 
		\begin{align}
			&\Risk(\scr M) \nonumber\\
			\ge & \frac{1}{2en}\sum_{\ell=1}^{n-1}\EE_{M\sim \scr M} \qth{M(1|1)^{\ell-1} M(2|1)\max\left\{0, \log\pth{\frac{M(2|1)\cdot \EE_{M\sim \scr M}[M(1|1)^{\ell-1}]}{\EE_{M\sim \scr M}[M(1|1)^{\ell-1}M(2|1)]}} -1\right\}}. \label{eq:KL.lowerbound}
		\end{align}
		Under the prior $\scr M$, $M(2|1)=1-M(1|1)=\alpha^{-m}$ with $\beta \leq m \leq 5\beta$.
			
			
			We further lower bound \eqref{eq:KL.lowerbound} by summing over an appropriate range of $\ell$. For any $m\in [\beta,3\beta]$, define
			\begin{align*}
				\ell_1(m) = \left\lceil \frac{\alpha^m}{\log \alpha} \right\rceil, \qquad \ell_2(m) = \left\lfloor \alpha^m\log \alpha \right\rfloor.
			\end{align*}
			Since $\gamma_0 \le e^{-e^5}$, our choice of $\alpha$ ensures that the intervals $\{[\ell_1(m),\ell_2(m)]\}_{\beta\le m\le 3\beta}$ are disjoint. We will establish the following claim: for all $m\in [\beta,3\beta]$ and $\ell\in [\ell_1(m), \ell_2(m)]$, it holds that
			\begin{align}\label{eq:MC_bayes-claim}
				\frac{\alpha^{-m}\cdot \EE_{M\sim \scr M}[M(1|1)^{\ell-1}]}{\EE_{M\sim \scr M}[M(1|1)^{\ell-1}M(2|1)]} \gtrsim \frac{\log(1/\gamma_0)}{\log\log(1/\gamma_0)}. 
			\end{align}
		
			We first complete the proof of the Bayes risk bound assuming \eqref{eq:MC_bayes-claim}. Using \eqref{eq:KL.lowerbound} and \eqref{eq:MC_bayes-claim}, we have
		\begin{align*}
			\Risk(\scr M)
			&\gtrsim {1\over n}\cdot 
			\frac{1}{4\beta}\sum_{m=\beta}^{3\beta}\sum_{\ell=\ell_1(m)}^{\ell_2(m)}\alpha^{-m}(1-\alpha^{-m})^{\ell-1}\cdot \log\log(1/\gamma_0)
			\nonumber\\
			&= {\log\log(1/\gamma_0)\over 4n\beta}
			\sum_{m=\beta}^{3\beta}
			\sth{(1-\alpha^{-m})^{\ell_1(m) - 1 }-(1-\alpha^{-m})^{\ell_2(m)}}\nonumber \\
			&\stepa{\ge} {\log\log(1/\gamma_0)\over 4n\beta}\sum_{m=\beta}^{3\beta}\pth{\pth{\frac 14}^{1\over \log\alpha}-\pth{\frac 1e}^{-1+\log\alpha }}
			\gtrsim {\log\log(1/\gamma_0)\over n}, 
		\end{align*}
	with (a) following from $\frac 14\leq (1-x)^{1\over x}\leq \frac 1e$ if $x\leq \frac 12$, and $\alpha^{-m}\leq \alpha^{-\beta}\le\gamma_0^{1/5}\leq \frac 12$. 
			
		
	Next we prove the claim \eqref{eq:MC_bayes-claim}. Expanding the expectation in \eqref{eq:twostate.prior}, we write the LHS of \eqref{eq:MC_bayes-claim} as
	\begin{align*}
			\frac{\alpha^{-m}\cdot \EE_{M\sim \scr M}[M(1|1)^{\ell-1}]}{\EE_{M\sim \scr M}[M(1|1)^{\ell-1}M(2|1)]} = {X_\ell + A_\ell+ B_\ell\over X_\ell + C_\ell+D_\ell}, 
	\end{align*}
	where
	\begin{align*}
		X_\ell&=\pth{1-\alpha^{-m}}^\ell, \quad 
		A_\ell= \sum_{j={\beta}}^{m-1}\pth{1-\alpha^{-j}}^\ell, \quad 
		B_\ell= \sum_{j=m+1}^{5{\beta}}\pth{1-\alpha^{-j}}^\ell,\\
		C_\ell&= \sum_{j={\beta}}^{m-1}\pth{1-\alpha^{-j}}^\ell{\alpha}^{m-j}, \quad 
		D_\ell= \sum_{j=m+1}^{5{\beta}}\pth{1-\alpha^{-j}}^\ell{\alpha}^{m-j}.
	\end{align*}
			
			We bound each of the terms individually.
			Clearly, $X_\ell\in (0,1)$ and $A_\ell\geq 0$. 
			Thus it suffices to show that $B_\ell \gtrsim \beta$ and $C_\ell,D_\ell\lesssim 1$, 
			for $m\in [\beta,3\beta]$ and $
			\ell_1(m)\leq \ell\leq  \ell_2(m)$.
			Indeed,
			\begin{itemize}
			\item 
			For $j\geq m+1$, we have
			\begin{align*}
			\pth{1-{\alpha}^{-j}}^\ell 
			\geq\pth{1-{\alpha}^{-j}}^{\ell_2(m)}
			\stepa{\geq} \pth{1/4}^{\ell_2(m)\over {\alpha}^j}
			\geq \pth{1/4}^{\log {\alpha} \over {\alpha}}
			\geq 1/4, 
			\end{align*}
			where in (a) we use the inequality $(1-x)^{1/x}\geq 1/4$ for $x\leq 1/2$. Consequently, $B_\ell \ge \beta/2$; 
			\item 
			For $j\le m-1$, we have
			\begin{align*}
				\pth{1-{\alpha}^{-j}}^\ell
				\leq \pth{1-{\alpha}^{-j}}^{\ell_1(m)} \stepb{\le} e^{-\frac{{\alpha}^{m-j}}{\log \alpha}}
				= \gamma_0^{{\alpha}^{m-j-1}\over \log {\alpha}},
			\end{align*}
			where (b) follows from $(1-x)^{1/x}\le 1/e$ and the definition of $\ell_1(m)$. Consequently, 
			\begin{align*}
				C_\ell \le \gamma_0^{\alpha\over \log\alpha}\sum_{j={\beta}}^{m-2}{\alpha}^{m-j}
				+{\alpha}\gamma_0^{1\over \log {\alpha}}
				\leq e^{-{\alpha^2\over \log \alpha} + (2\beta+1)\log\alpha }
				+e^{\log\alpha -\frac{\alpha}{\log\alpha}}
				\leq 2, 
			\end{align*}
			where the last step uses the definition of $\beta$ in \eqref{eq:alpha.beta};
			\item 
			$D_{\ell} \le \sum_{j=m+1}^{5\beta} \alpha^{m-j} \le 1$, since $\alpha =\log\frac{1}{\gamma_0}\geq e^5$.
		\end{itemize}
			Combining the above bounds completes the proof of \eqref{eq:MC_bayes-claim}. 
		\end{proof}

		\subsection{$k$ states}
\label{sec:gammak}		
		
	\subsubsection{Proof of \prettyref{thm:gammak} (i)}
	Notice that the prediction problem consists of $k$ sub-problems of estimating the individual rows of $M$, so it suffices show the contribution from each of them is $O\pth{\frac kn}$. In particular, assuming the chain terminates in state 1 we bound the risk of estimating the first row by the add-one estimator $\hat M^{+1}(j|1)={N_{1j}+1\over N_1+k}$.
	Under the absolute spectral gap condition of $\gamma_*\geq \gamma_0$, we show
	\begin{align}
		\EE\qth{\indc{X_n=1}D\pth{{\mymat(\cdot|1)}\|{\hat\mymat^{+1}(\cdot|1)}}}
		\lesssim {k\over n}\pth{1+\sqrt{\log k\over k\gamma_0^4}}.
		\label{eq:risk.row1}
	\end{align}
	By symmetry, 
	we get the desired $\Risk_{k,n}(\gamma_0)\lesssim {k^2\over n}\pth{1+\sqrt{\log k\over k\gamma_0^4}}$.
	The basic steps of our analysis are as follows:
	\begin{itemize}
		\item When $N_1$ is substantially smaller than its mean, we can bound the risk using the worst-case risk bound for add-one estimators and the 
		probability of this rare event.
		\item Otherwise, we decompose the prediction risk as
		\begin{align*}
			D(\mymat(\cdot|1)\|\hat\mymat^{+1}(\cdot|1))
			=\sum_{j=1}^k\qth{\mymat(j|1)\log\pth{\mymat(j|1)(N_1+k)\over N_{1j}+1}-\mymat(j|1)+{N_{1j}+1\over N_1+k}}.
		\end{align*}
		We then analyze each term depending on whether $N_{1j}$ is typical or not.
		Unless $N_{1j}$ is atypically small, the add-one estimator works well whose risk can be bounded quadratically.
	\end{itemize}

		To analyze the concentration of the empirical counts we use the following moment bounds. The proofs are deferred to \prettyref{app:moment.bounds}.
	\begin{lemma}\label{lmm:moment.reversebound}
		Finite reversible and irreducible chains observe the following moment bounds:
		\begin{enumerate}[label=(\roman*)]
			\item \label{lmm:secondmoment.reversebound}
			${\EE\qth{\pth{N_{ij}-N_i{\mymat(j|i)}}^2|X_n=i}}
			\lesssim 
				n\pi_i{\mymat(j|i)}(1-{\mymat(j|i)})+{\sqrt{M(j|i)}\over \gamma_*}
				+{{M(j|i)}\over \gamma_*^2}$
			\item 
			\label{lmm:fourthmoment.reversebound}
			$\EE\qth{\pth{N_{ij}-N_i{\mymat(j|i)}}^4|X_n=i}
			\lesssim  (n\pi_i{\mymat(j|i)}(1-{\mymat(j|i)}))^2
				+{\sqrt{{\mymat(j|i)}}\over \gamma_*}+{{\mymat(j|i)}^2\over \gamma_*^4}$
			\item\label{lmm:fourthmoment.centralbound}
			$\EE\qth{\pth{N_i-(n-1)\pi_i}^4|X_n=i}
			\lesssim {n^2\pi_i^2\over \gamma_*^2}
				+{1\over \gamma_*^4}.$
		\end{enumerate}
	\end{lemma} 
	When $\gamma_*$ is high this shows that the moments behave as if for each $i\in [k]$, $N_1$ is approximately Binomial($n-1,\pi_i$) and $N_{ij}$ is approximately Binomial$(N_i,M(j|i))$, which happens in case of \iid sampling. For \iid models \cite{KOPS15} showed that the add-one estimator achieves $\calO\pth{\frac kn}$ risk bound which we aim here too. In addition, dependency of the above moments on $\gamma_*$ gives rise to sufficient conditions that guarantees parametric rate. The technical details are given below.

%
%

	We decompose the left hand side in \eqref{eq:risk.row1} based on $N_1$ as
		\begin{align*}
		\EE\qth{\indc{X_n=1}D\pth{{\mymat(\cdot|1)}\|{\hat\mymat^{+1}(\cdot|1)}}}
		=\EE\qth{\indc{A^\leq}D\pth{{\mymat(\cdot|1)}\|{\hat\mymat^{+1}(\cdot|1)}}}
		+\EE\qth{\indc{A^>}D\pth{{\mymat(\cdot|1)}\|{\hat\mymat^{+1}(\cdot|1)}}}
		\end{align*} 
		where the typical set $A^>$ and atypical set $A^\leq$ are defined as
		\begin{align*}
		A^\leq\eqdef\sth{X_n=1, N_1\leq {(n-1)\pi_1/ 2}},
		\quad A^>\eqdef\sth{X_n=1, N_1> {(n-1)\pi_1/ 2}}.
		\end{align*}
		For the atypical case, 
note the following deterministic property of the add-one estimator.
Let $\hat Q$ be an add-one estimator with sample size $n$ and alphabet size $k$
of the form  $\hat Q_i = \frac{n_i+1}{n+k}$, where $\sum n_i=n$. Since 
$\hat Q$ is bounded below by $\frac{1}{n+k}$ everywhere, for any distribution $P$, we have
\begin{equation}
D(P \|\hat Q) \leq \log(n+k).
\label{eq:addone-worstcase}
\end{equation}		
Applying this bound on the event $A^\leq$, we have
		\begin{align}
		&\EE\qth{\indc{A^\leq}D\pth{{\mymat(\cdot|1)}\|{\hat\mymat^{+1}(\cdot|1)}}}
		\nonumber\\
		&\leq \log \pth{n\pi_1+k}
		\PP\qth{X_n=1,N_1\leq (n-1)\pi_1/2}
		\nonumber\\
		&\stepa{\lesssim} \indc{n\pi_1\gamma_*\leq 10}
		\pi_1\log\pth{n\pi_1+k}
		+\indc{n\pi_1\gamma_*>10} \pi_1\log \pth{{n\pi_1+k}}
		{\EE\qth{\pth{N_1-(n-1)\pi_1}^4|X_n=1}
			\over n^4\pi_1^4}
		\label{eq:moment.bound.N1}\\
		&\stepb{\leq}\indc{n\pi_1\gamma_*\leq 10}
			{10\over n\gamma_*}\log\pth{{10\over \gamma_*}+k}
			+\indc{n\pi_1\gamma_*>10}\log \pth{n\pi_1+k}\pth{{1\over n^2\pi_1\gamma_*^2}
				+{1\over n^4\pi_1^3\gamma_*^4}}
		\nonumber\\
		&\stepc{\lesssim} \frac 1n
		\sth{\indc{n\pi_1\gamma_*\leq 10}
			{\log(1/\gamma_*)+\log k\over \gamma_*}
			+\indc{n\pi_1\gamma_*>10}\pth{n\pi_1+\log k}\pth{{1\over n\pi_1\gamma_*^2}
			+{1\over n^3\pi_1^3\gamma_*^4}}}
		\nonumber\\
		&{\lesssim} {1\over n}
		\sth{\indc{n\pi_1\gamma_*\leq 10}\pth{\frac 1{\gamma_*^2}+{\log k\over \gamma_*}}
			+\indc{n\pi_1\gamma_*>10}\pth{{1\over\gamma_*^2}
				+{\log k\over \gamma_*}}}
		\lesssim {1\over n\gamma_0^2}+{\log k\over n\gamma_0}.
		\label{eq:bound.A<}
		\end{align} 
	where we got (a) from Markov inequality, (b) from \prettyref{lmm:moment.reversebound}(iii) and (c) using $x+y\leq xy,x,y\geq 2$.

	Next we bound $\EE\qth{\indc{A^>}D\pth{{\mymat(\cdot|1)}\|{\hat\mymat^{+1}(\cdot|1)}}}$. 
	Define \begin{align*}
		\Delta_i
		=\mymat(i|1)\log\pth{\mymat(i|1)\over \hat\mymat^{+1}(i|1)}-\mymat(i|1)+\hat\mymat^{+1}(i|1).
	\end{align*}
	As $D({\mymat (\cdot|1)}\|{\hat\mymat^{+1} (\cdot|1)})= \sum_{i=1}^k\Delta_i$ it suffices to bound $\EE\qth{\indc{A^>}\Delta_i}$ for each $i$. For some $r\geq 1$ to be optimized later consider the following cases separately

	\paragraph{Case (a) $n\pi_1\leq r$ or $n\pi_1{\mymat(i|1)}\leq 10$:}
	
		Using the fact $y\log(y)-y+1\leq (y-1)^2$ with $y={M(i|1)\over \hat M^{+1}(i|1)}={{\mymat(i|1)}(N_1+k)\over N_{1i}+1}$ we get
		\begin{align}\label{eq:chisq.bound}
		\Delta_i\leq {\pth{{\mymat(i|1)}N_1-N_{1i}+{\mymat(i|1)}k-1}^2\over \pth{N_1+k}\pth{N_{1i}+1}}.
		\end{align}
		This implies
		\begin{align}
		\EE\qth{\indc{A^>}\Delta_i}
		&\leq \EE\qth{\indc{A^>}\pth{{\mymat(i|1)}N_1-N_{1i}+{\mymat(i|1)}k-1}^2\over
			\pth{N_1+k}\pth{N_{1i}+1}}
		\nonumber \\
		&\stepa{\lesssim} {{\EE\qth{\indc{A^>}
				\pth{{\mymat(i|1)}N_1-N_{1i}}^2}
			+k^2\pi_1{\mymat(i|1)}^2+\pi_1}\over n\pi_1+k}
		\nonumber \\
		&\stepb{\lesssim} {\pi_1\EE\qth{\left.{\pth{{\mymat(i|1)}N_1-N_{1i}}^2}\right|X_n=1}\over n\pi_1+k}
		+{1+rk{M(i|1)}\over n}
		\end{align}
		where (a) follows from $N_1>{(n-1)\pi_1\over 2}$ in $A^>$ and the fact that $(x+y+z)^2\leq 3(x^2+y^2+z^2)$;
		(b) uses the assumption that either $n\pi_1\leq r$ or $n\pi_1{\mymat(i|1)}\leq 10$. Applying \prettyref{lmm:moment.reversebound}\ref{lmm:secondmoment.reversebound} and the fact that $x+x^2\leq 2(1+x^2)$,
		continuing the last display we get 
		\begin{align*}
		&\EE\qth{\indc{A^>}\Delta_i}
		\lesssim
		{n\pi_1{\mymat(i|1)}+\pth{1+{{M(i|1)}\over \gamma^2_*}}
			\over n}
		+{1+rk{M(i|1)}\over n}
		\lesssim {{1+rk{M(i|1)} \over n}+{M(i|1)\over n\gamma_0^2}}.
		\end{align*}
		Hence
		\begin{align}
		\EE\qth{\indc{A^>}D({\mymat (\cdot|1)}\|{\hat\mymat^{+1} (\cdot|1)})}
		= \sum_{i=1}^k\EE\qth{\indc{A^>}\Delta_i}
		\lesssim {\frac {rk}n +{1\over\gamma_0^2}}.
		\label{eq:bound.A>.case(a)}
		\end{align}

		\paragraph{Case(b) $n\pi_1> r$ and $n\pi_1{\mymat(i|1)}> 10$:}
		
		We decompose $A^>$ based on count of $N_{1i}$ into atypical part $B^{\leq}$ and typical part $B^>$
		\begin{align*}
		B^{\leq}&\eqdef \sth{X_n=1,N_1>{(n-1)\pi_1/2},N_{1i}\leq {(n-1)\pi_1{\mymat(i|1)}/4}}\\
		B^>&\eqdef \sth{X_n=1,N_1>{(n-1)\pi_1/2},N_{1i}> {(n-1)\pi_1{\mymat(i|1)}/4}}
		\end{align*}
		and bound each of $\EE\qth{\indc{B^\leq}\Delta_i}$ and $\EE\qth{\indc{B^>}\Delta_i}$ separately.
		\paragraph{Bound on $\EE\qth{\indc{B^{\leq}}\Delta_i}$}
		Using $\hat M^{+1}(i|1)\geq {1\over N_1+k}$ and $N_{1i}<N_1M(i|1)/2$ in $B^\leq$ we get
		\begin{align}
		\EE\qth{\indc{B^{\leq}}\Delta_i}
		&=\EE\qth{\indc{B^{\leq}}{\mymat(i|1)}\log \pth{{\mymat(i|1)}(N_1+k)\over N_{1i}+1}}
		+\EE\qth{\indc{B^{\leq}}\pth{{N_{1i}+1\over N_1+k}-{\mymat(i|1)}}} 
		\nonumber \\
		&\leq \EE\qth{\indc{B^{\leq}}{\mymat(i|1)}\log \pth{{\mymat(i|1)} (N_1+k)}}
		+\EE\qth{\indc{B^{\leq}}\pth{{N_{1i}\over N_1}-{\mymat(i|1)}}}
		+\EE \qth{\indc{B^{\leq}}\over N_1}
		\nonumber \\
		&\lesssim \EE\qth{\indc{B^{\leq}}{\mymat(i|1)}\log \pth{{\mymat(i|1)} (N_1+k)}}
		+{1\over n}
		\label{eq:B1.risk.decomp}
		\end{align} 
		where the last inequality followed as
		$
		\EE \qth{\indc{B^\leq}/ N_1}
		\lesssim {\PP[X_n=1]/n\pi_1}
		=\frac 1n$.
		Note that for any event $B$ and any function $g$,
		\begin{align*}
		\EE\qth{g(N_1)\indc{N_1\geq t_0,B}}
		=g(t_0)\PP[N_1\geq t_0,B]
		+\sum_{t= t_0+1}^n\pth{g(t)-g(t-1)}
		\PP[N_1\geq t,B].
		\end{align*}
		Applying this identity with $t_0=\ceil{(n-1)\pi_1/2}$, we can bound the expectation term in \eqref{eq:B1.risk.decomp} as
		\begin{align}
		&\EE\qth{\indc{B^{\leq}}{\mymat(i|1)}\log\pth{{\mymat(i|1)}(N_1+k)}}
		\nonumber \\
		&={\mymat(i|1)}
		\log\pth{{\mymat(i|1)}(t_0+k)}\PP\qth{N_1\geq t_0,N_{1i}\leq {n\pi_1{\mymat(i|1)}\over 4},X_n=1}
		\nonumber\\
		&\quad +{\mymat(i|1)}\sum_{t= t_0+1}^{n-1}
		\log\pth{1+\frac 1{t-1+k}}\PP\qth{N_1\geq t+1,N_{1i}\leq {n\pi_1{\mymat(i|1)}\over 4},X_n=1}
		\nonumber \\
		&\leq \pi_1{\mymat(i|1)}
		\log\pth{{\mymat(i|1)}(t_0+k)}\PP\qth{\left.{\mymat(i|1)}N_1-N_{1i}\geq {\mymat(i|1)t_0\over 4}\right|X_n=1}
		\nonumber\\
		&\quad +{{\mymat(i|1)}\over n}\sum_{t=t_0+1}^{n-1}
		\PP\qth{\left.{\mymat(i|1)}N_1-N_{1i}\geq {\mymat(i|1)t\over 4}\right|X_n=1}
		\label{eq:B1.expectation.part}
		\end{align}
		where last inequality uses $\log\pth{1+\frac 1{t-1+k}}\leq \frac 1t\lesssim \frac 1{n\pi_1}$ for all $t\geq t_0$. Using Markov inequality $\PP\qth{Z>c}\leq c^{-4}{\EE\qth{Z^4}}$ for $c>0$, \prettyref{lmm:moment.reversebound}\ref{lmm:fourthmoment.reversebound} and $x+x^4\leq 2(1+x^4)$ with $x=\sqrt{M(i|1)}/\gamma_*$
		\begin{align*}
		&\PP\qth{\left.{\mymat(i|1)}N_1-N_{1i}\geq {{\mymat(i|1)}t\over 4}\right|X_n=1}
		\lesssim {(n\pi_1{\mymat(i|1)})^2
			+{{\mymat(i|1)}^2\over \gamma_*^4} \over \pth{t{\mymat(i|1)}}^4}.
		\end{align*}
		In view of above continuing \eqref{eq:B1.expectation.part} we get
		\begin{align*}
		&\EE\qth{\indc{B^{\leq}}{\mymat(i|1)}\log\pth{{\mymat(i|1)}(N_1+k)}}
		\nonumber\\
		&\lesssim \pth{(n\pi_1{\mymat(i|1)})^2+{{\mymat(i|1)}^2\over \gamma_*^4}}
		\pth{{\pi_1{\mymat(i|1)}\log({\mymat(i|1)}(n\pi_1+k))\over (n\pi_1{\mymat(i|1)})^4}
			+\frac 1{n({\mymat(i|1)})^3}\sum_{t=t_0+1}^{n}{1\over t^4}}
		\nonumber \\
		&{\lesssim} \pth{(n\pi_1{M(i|1)})^2+{{\mymat(i|1)}^2\over \gamma_*^4}\over n}
		\pth{{\log(n\pi_1M(i|1)+kM(i|1))\over (n\pi_1{\mymat(i|1)})^3}
		+\frac 1{(n\pi_1{\mymat(i|1)})^3}}
		\nonumber \\
		&{\lesssim} {1\over n}
		\pth{(n\pi_1{\mymat(i|1)})^2+{{\mymat(i|1)}^2\over \gamma_*^4}}
		{\log(n\pi_1M(i|1)+kM(i|1))\over (n\pi_1{\mymat(i|1)})^3}
		\nonumber\\
		&\lesssim {1\over n}\pth{{\log(n\pi_1M(i|1)+kM(i|1))\over n\pi_1{\mymat(i|1)}}
		+{M(i|1)\log(n\pi_1M(i|1)+k)\over
			n\pi_1\gamma_*^4 (n\pi_1M(i|1))^2}}
		\nonumber\\
		&\stepa{\lesssim} {1\over n}\pth{{n\pi_1M(i|1)+kM(i|1)\over n\pi_1{\mymat(i|1)}}
			+{M(i|1)\log(n\pi_1M(i|1))\over n\pi_1\gamma_*^4 (n\pi_1M(i|1))^2}+{M(i|1)\log k\over
				n\pi_1\gamma_*^4 (n\pi_1M(i|1))^2}}
		\nonumber\\
		&\stepb{\lesssim} {1\over n}
		\pth{1+kM(i|1)+{M(i|1)\log k\over r\gamma_0^4}}
		\end{align*}
		where (a) followed using $x+y\leq xy$ for $x,y\geq 2$ and (b) followed as $n\pi_1\geq r,n\pi_1{\mymat(i|1)}\geq 10$ and $\log(n\pi_1M(i|1))\leq n\pi_1M(i|1)$. In view of \eqref{eq:B1.risk.decomp} this implies
		\begin{align}
			\sum_{i=1}^k\EE\qth{\indc{B^{\leq}}\Delta_i}
			\lesssim \sum_{i=1}^k {1\over n}
			\pth{1+kM(i|1)\pth{1+{\log k\over rk\gamma_0^4}}}
			\lesssim {k\over n}
			\pth{1+{\log k\over rk\gamma_0^4}}.
		\end{align}
		
		\paragraph{Bound on $\EE\qth{\indc{B^>}\Delta_i}$}
		
		\noindent
		
		Using the inequality \eqref{eq:chisq.bound} 
		\begin{align*}
		\EE\qth{\indc{B^>}\Delta_i}
		&\leq \EE\qth{\indc{B^>}\pth{{\mymat(i|1)}N_1-N_{1i}+{\mymat(i|1)}k-1}^2\over
			\pth{N_1+k}\pth{N_{1i}+1}}
		\nonumber \\
		&\lesssim {\EE\qth{\indc{B^>}
				\sth{\pth{{\mymat(i|1)}N_1-N_{1i}}^2}}+k^2\pi_1{\mymat(i|1)}^2+\pi_1
			\over (n\pi_1+k)(n\pi_1{\mymat(i|1)}+1)}
		\nonumber \\
		&\lesssim {\pi_1\EE\qth{\left.{\pth{{\mymat(i|1)}N_1-N_{1i}}^2}\right|X_n=1}\over (n\pi_1+k)(n\pi_1{\mymat(i|1)}+1)}
		+{k{\mymat(i|1)}\over n}
		\end{align*}
		where (a) follows using properties of the set $B^>$ along with $(x+y+z)^2\leq 3(x^2+y^2+z^2)$. Using \prettyref{lmm:moment.reversebound}\ref{lmm:secondmoment.reversebound} we get 
		\begin{align*}
		&\EE\qth{\indc{B^>}\Delta_i}
		\lesssim {n\pi_1{\mymat(i|1)}+\pth{1+{{\mymat(i|1)}\over \gamma^2_*}}
			\over n(n\pi_1{\mymat(i|1)}+1)}
		+{k{\mymat(i|1)}\over n}
		\lesssim {{1+k{\mymat(i|1)} \over n}+{{\mymat(i|1)}\over n\gamma_0^2}}.
		\end{align*}
		Summing up the last bound over $i\in [k]$ and using
		\label{eq:bound.B1} we get for $n\pi_1>r,n\pi_1M(i|1)>10$
		\begin{align*}
		\EE\qth{\indc{A^>}D({\mymat (\cdot|1)}\|{\hat\mymat^{+1} (\cdot|1)})}
		&=\sum_{i=1}^k\qth{\EE\qth{\indc{B^{\leq}}\Delta_i}
			+\EE\qth{\indc{B^>}\Delta_i}}
		\lesssim {k\over n}\pth{1+{1\over k\gamma_0^2}+{\log k\over rk\gamma_0^4}}.
		\end{align*}
		Combining this with \eqref{eq:bound.A>.case(a)} we obtain
		\begin{align*}
		\EE\qth{\indc{A^>}D({\mymat (\cdot|1)}\|{\hat\mymat^{+1} (\cdot|1)})}
		&\lesssim 
		{k\over n}\pth{{1\over k\gamma_0^2}+r+{\log k\over rk\gamma_0^4}}
		\lesssim {k\over n}\pth{1+{\sqrt{\log k\over k\gamma_0^4}}}
		\end{align*}
		where we chose $r=10+\sqrt{\log k\over k\gamma_0^4}$ for the last inequality. In view of \eqref{eq:bound.A<} this implies the required bound.

	\begin{remark}
	\label{rmk:4thmoment}
	We explain the subtlety of the concentration bound in \prettyref{lmm:moment.reversebound} based on fourth moment and why existing Chernoff bound or Chebyshev inequality falls short.
	For example, the risk bound in \eqref{eq:bound.A<} relies on bounding the probability that $N_1$ is atypically small. To this end, one may use the classical Chernoff-type inequality for reversible chains (see \cite[Theorem 1.1]{L98} or \cite[Proposition 3.10 and Theorem 3.3]{P15})
		\begin{align}
			\PP\qth{N_1\leq (n-1)\pi_1/2|X_1=1}
			\lesssim \frac 1{\sqrt{\pi_1}} e^{-\Theta(n\pi_1\gamma_*)};
			\label{eq:pauline}
		\end{align}
		in contrast, the fourth moment bound in \eqref{eq:moment.bound.N1} yields 
		$\PP\qth{N_1\leq (n-1)\pi_1/2|X_1=1} =O(\frac{1}{(n\pi_1\gamma_*)^2})$.
		Although the exponential tail in \prettyref{eq:pauline} is much better, the pre-factor $\frac 1{\sqrt{\pi_1}}$, due to conditioning on the initial state, can lead to a suboptimal result when $\pi_1$ is small. (As a concrete example, consider two states with $M(2|1)=\Theta(\frac {1}n)$ and $M(1|2)=\Theta(1)$. Then $\pi_1=\Theta(\frac{1}{n}),\gamma=\gamma_*\approx \Theta(1)$, and \prettyref{eq:pauline} leads to $\PP\qth{N_1\leq (n-1)\pi_1/2,X_n=1} = O(\frac{1}{\sqrt{n}})$ as opposed to the desired $O(\frac{1}{n})$.)
		
		
		In the same context it is also insufficient to use 2nd moment based bound (Chebyshev), which leads to $\PP\qth{N_1\leq (n-1)\pi_1/2|X_1=1} =O(\frac{1}{n\pi_1\gamma_*})$. This bound is too loose, which, upon substitution into \prettyref{eq:moment.bound.N1}, results in an extra $\log n$ factor in the final risk bound when $\pi_1$ and $\gamma_*$ are large.
	\end{remark}


	\subsubsection{Proof of \prettyref{thm:gammak} (ii)}

		Let $k\geq (\log n)^6$ and $\gamma_0\geq {(\log (n+k))^2\over k}$.
		We prove a stronger result using spectral gap as opposed to the absolute spectral gap. Fix 
		$M$ such that $\gamma \geq \gamma_0$. Denote its stationary distribution by $\pi$. For absolute constants $\tau>0$ to be chosen later and $c_0$ as in \prettyref{lmm:multinomial.bound} below, define
		\begin{gather}
			\epsilon(m)={2k\over m}+{c_0(\log n)^3\sqrt k\over m},
			\quad c_n=100\tau^2{\log n\over n\gamma},
			\nonumber\\
			n_i^{\pm}=n\pi_i\pm \tau\max\sth{{\log n\over n\gamma},{\sqrt{\pi_i\log n\over n\gamma}}}, \quad i=1,\ldots,k.
			\label{eq:nipm}
		\end{gather}
		Let $N_i$ be the number of visits to state $i$ as in \prettyref{eq:transition.count}. 
		We bound the risk by accounting for the contributions from different ranges of $N_i$ and $\pi_i$ separately:
		\begin{align}
			&\EE\qth{\sum_{i=1}^k\indc{X_n=i}
				D\pth{M(\cdot|i)\|\hat M^{+1}(\cdot|i)}}
			\nonumber\\
			&=\sum_{i:\pi_i\geq c_n}\EE\qth{\indc{X_n=i,n_i^-\leq N_i\leq n_i^+}
				D\pth{M(\cdot|i)\|\hat M^{+1}(\cdot|i)}}
			\nonumber\\
			&+\sum_{i:\pi_i\geq c_n}\EE\qth{\indc{X_n=i, N_i> n_i^+ \text{ or } N_i<n_i^-}
				D\pth{M(\cdot|i)\|\hat M^{+1}(\cdot|i)}}
			+\sum_{i:\pi_i< c_n}\EE\qth{\indc{X_n=i}
				D\pth{M(\cdot|i)\|\hat M^{+1}(\cdot|i)}}
			\nonumber\\
			&\leq \log (n+k)\sum_{i:\pi_i\geq c_n} 
			\PP\qth{D(M(\cdot|i)\|\hat M^{+1}(\cdot|i))>\epsilon(N_i),n_i^-\leq N_i\leq n_i^+}
			+\sum_{i:\pi_i\geq c_n}\EE\qth{\indc{X_n=i,n_i^-\leq N_i\leq n_i^+}\epsilon(N_i)}
			\nonumber\\
			&+\log (n+k)\sum_{i:\pi_i\geq c_n}\qth{\PP\qth{N_i\geq n_i^+}
				+\PP\qth{N_i\leq n_i^-}}
			+\sum_{i:\pi_i\leq c_n}\pi_i\log(n+k)
			\nonumber\\
			&\lesssim \log (n+k)\sum_{i:\pi_i\geq c_n}
			\PP\qth{D(M(\cdot|i)\|\hat M^{+1}(\cdot|i))>\epsilon(N_i),n_i^-\leq N_i\leq n_i^+}
			+\sum_{i:\pi_i\geq c_n}\pi_i\max_{n_i^-\leq m\leq n_i^+}\epsilon(m)
			\nonumber\\
			&\quad
			+\log(n+k)\sum_{i:\pi_i\geq c_n}\pth{\PP\qth{N_i> n_i^+}
				+\PP\qth{N_i< n_i^-}}+{k\pth{\log(n+k)}^2\over n\gamma}.
			\label{eq:s5}
		\end{align}
	where the first inequality uses the worst-case bound \prettyref{eq:addone-worstcase} for add-one estimator.
		We analyze the terms separately as follows. 
				
		For the second term, given any $i$ such that $\pi_i\geq c_n$, we have, by definition in \prettyref{eq:nipm},  $n_i^-\geq 9n\pi_i/10$ and $n_i^+-n_i^-\leq n\pi_i/5$, which implies
		\begin{align}
			\sum_{i:\pi_i\geq c_n}\pi_i\max_{n_i^-\leq m\leq n_i^+}\epsilon(m)
			&\leq 
			\sum_{i:\pi_i\geq c_n}\pi_i \pth{{2k\over 0.9n\pi_i}+{10\over 9}{c_0(\log n)^3\sqrt k\over n\pi_i}}
			\lesssim {k^2\over n}
			+{(\log n)^3k^{3/2}\over n}.
			\label{eq:3(ii).N_i.concentration}
		\end{align}		
		For the third term, applying \cite[Lemma 16]{HJLWWY18} (which, in turn, is based on the Bernstein inequality in \cite{P15}), we get
		$\PP\qth{N_i> n_i^+}+\PP\qth{N_i< n_i^-}
			\leq 2n^{-\tau^2\over 4+10\tau}$.

			To bound the first term in \prettyref{eq:s5}, we follow the method in \cite{B61,HJLWWY18}  of representing the sample path of the Markov chain using independent samples generated from $M(\cdot|i)$ which we describe below. Consider a random variable $X_1\sim \pi$ and an array $W=\sth{W_{i\ell}:i=1,\dots,k\text{ and } \ell=1,2,\dots}$ of independent random variables,
			such that $X$ and $W$ are independent and $W_{i\ell} \iiddistr M(\cdot|i)$ for each $i$.
			Starting with generating $X_1$ from $\pi$, at every step $i\geq 2$ we set $X_i$ as the first element in the $X_{i-1}$-th row of $W$ that has not been sampled yet.
			Then one can verify that $\sth{X_1,\dots,X_n}$ is a Markov chain with initial distribution $\pi$ and transition matrix $M$.
			Furthermore, the transition counts satisfy $N_{ij}=\sum_{\ell=1}^{N_i}\indc{W_{i\ell}=j}$, where
$N_i$ be the number of elements sampled from the $i$th row of $W$.
			Note the conditioned on $N_i=m$, the random variables $\{W_{i1},\ldots,W_{im}\}$ are no longer iid. Instead, we apply a union bound. Note that for each fixed $m$, 
			the estimator			
			\[
			\hat M^{+1}(j|i)= {\sum_{\ell=1}^{m}\indc{W_{i\ell}=j}+1\over m+k} \triangleq \hat M^{+1}_m(j|i), \quad j\in [k]
			\]
			is an add-one estimator for $M(j|i)$ based on an \iid sample of size $m$. 
			\prettyref{lmm:multinomial.bound} below provides a high-probability bound for the add-one estimator in this iid setting. Using this result  and the union bound, we have				
		\begin{align*}
			&\sum_{i:\pi_i\geq c_n}\PP\qth{D(M(\cdot|i)\|\hat M^{+1}(\cdot|i))>\epsilon(N_i),n_i^-\leq N_i\leq n_i^+}
			\nonumber\\
			&\leq 
			\sum_{i:\pi_i\geq c_n}\pth{n_i^+-n_i^-} \max_{n_i^-\leq m\leq n_i^+} \PP\qth{D(M(\cdot|i)\|\hat M_m^{+1}(\cdot|i))>\epsilon(m)}
			\leq 
			\sum_{i:\pi_i\geq c_n} {1\over n^2}
			\leq {k\over n^2}
		\end{align*}
		where the second inequality applies \prettyref{lmm:multinomial.bound} with $t=n\geq n_i^+ \geq m$ and uses $n_i^+-n_i^-\leq n\pi_i/5$ for $\pi_i\geq c_n$.

		Combining the above with \eqref{eq:3(ii).N_i.concentration}, we continue \eqref{eq:s5} with $\tau=25$ to get
		\begin{align*}
			&\EE\qth{\sum_{i=1}^k\indc{X_n=i}
				D\pth{M(\cdot|i)\|\hat M^{+1}(\cdot|i)}}
			\lesssim
			{k^2\over n}
			+{(\log n)^3k^{3/2} \over n}
			+{k(\log (n+k))^2\over n\gamma}
		\end{align*}
	which is $\calO\pth{k^2\over n}$ whenever $k\geq (\log n)^6$ and $\gamma\geq {(\log (n+k))^2\over k}$.

	\begin{lemma}[KL risk bound for add-one estimator]\label{lmm:multinomial.bound}
	Let $V_1,\dots, V_m\simiid Q$ for some distribution $Q=\sth{Q_i}_{i=1}^k$ on $[k]$. 
	Consider the add-one estimator $\hat Q^{+1}$ with	
	$\hat Q^{+1}_i=\frac{1}{m+k}(\sum_{j=1}^m\indc{V_j=i}+1)$.
	There exists an absolute constant $c_0$ such that for any $t\ge m$, 
	\begin{align*}
		\PP\qth{D(Q\|\hat Q^{+1})\geq {2k\over m}+{c_0(\log t)^3\sqrt k\over m}}\leq {1\over t^{3}}.
	\end{align*}
\end{lemma}

\begin{proof}
	Let $\hat Q$ be the empirical estimator $\hat Q_i=\frac{1}{m}\sum_{j=1}^m\indc{V_j=i}$. Then $\hat Q^{+1}_i={m\hat Q_i+1\over m+k}$ and hence
	\begin{align}
		D(Q\|\hat Q^{+1})
		&=\sum_{i=1}^k\pth{Q_i\log{Q_i\over \hat Q_i^{+1}}-Q_i+\hat Q^{+1}_i}
		\nonumber\\
		&=\sum_{i=1}^k\pth{Q_i\log{Q_i(m+k)\over m\hat Q_i+1}-Q_i+{m\hat Q_i+1\over m+k}}
		\nonumber\\
		&=\sum_{i=1}^k\pth{Q_i\log{Q_i\over \hat Q_i+\frac 1m}-Q_i+\hat Q_i+\frac 1m}
		+\sum_{i=1}^k\pth{Q_i\log{m+k\over m}-{k\hat Q_i\over m+k}-{k\over m(m+k)}}
		\nonumber\\
		&\leq\sum_{i=1}^k\pth{Q_i\log{Q_i\over \hat Q_i+\frac 1m}-Q_i+\hat Q_i+\frac 1m}
		+\frac km\label{eq:KL-nopoisson}
	\end{align}
	with last equality following by $0\leq\log\pth{m+k\over m}\leq k/m$. 
	
	To control the sum in the above display it suffices to consider its Poissonized version. Specifically, we aim to show
	\begin{align}\label{eq:KL-poisson}
		\PP\qth{\sum_{i=1}^k\pth{ Q_i\log{Q_i\over \hat Q_i^{\mathsf{poi}}+\frac 1m}-Q_i+\hat Q_i^{\mathsf{poi}}+\frac 1m}>{k\over m}+{c_0(\log t)^3\sqrt k\over m}}\leq {1\over t^4}
	\end{align}
	where $m\hat Q^{\mathsf{poi}}_i, i=1,\ldots,k$ are distributed independently as $\Poi(m Q_i)$. (Here and below 
	$\Poi(\lambda)$ denotes the Poisson distribution with mean $\lambda$.) To see why \prettyref{eq:KL-poisson} implies the desired result, letting $w={k\over m}+{c_0(\log t)^3\sqrt k\over m}$ 
	and $Y=\sum_{i=1}^km\hat Q_i^{\mathsf{poi}}\sim \Poi(m)$, we have
	\begin{align}
		&\PP\qth{\sum_{i=1}^k\pth{Q_i\log{Q_i\over \hat Q_i+\frac 1m}-Q_i+\hat Q_i+\frac 1m}>w}
		\nonumber\\
		&\stepa{=} \PP\qth{\left.\sum_{i=1}^k\pth{Q_i\log{Q_i\over \hat Q_i^{\mathsf{poi}}+\frac 1m}-Q_i+\hat Q_i^{\mathsf{poi}}+\frac 1m}>w \right| 
		\sum_{i=1}^kQ^{\mathsf{poi}}_i=1}
		\nonumber\\
		&\stepb{\leq} {1\over t^4\PP[Y=m]}
		= {m!\over t^4e^{-m}m^m}
		\stepc{\lesssim} {\sqrt m\over t^4}
		\leq \frac 1{t^3}.
	\end{align}
	where (a) followed from the fact that conditioned on their sum independent Poisson random variables follow a multinomial distribution;
	(b) applies \eqref{eq:KL-poisson};
	(c) follows from Stirling's approximation. 
	
	To prove \eqref{eq:KL-poisson} we rely on concentration inequalities for sub-exponential distributions. A random variable $X$ is called sub-exponential with parameters $\sigma^2,b>0$, denoted as $\mathsf{SE}(\sigma^2,b)$ if 
	\begin{align}
		\EE\qth{e^{\lambda (X-\EE[X])}}
		\leq e^{\lambda^2\sigma^2\over 2},
		\quad \forall |\lambda|<\frac 1b.
	\end{align} 
	Sub-exponential random variables satisfy the following properties \cite[Sec.~2.1.3]{WainwrightBook19}:
	\begin{itemize}
		\item If $X$ is $\mathsf{SE}(\sigma^2,b)$ for any $t>0$
		\begin{align}\label{eq:subexpo.conc}
			\PP\qth{\abs{X-\EE[X]}\geq v}\leq 
			\begin{cases}
				2e^{-v^2/(2\sigma^2)} , &0<v\leq {\sigma^2\over b}\\
				2e^{-v/(2b)}
				, &v> {\sigma^2\over b}.
			\end{cases}
		\end{align}
	\item {Bernstein condition}: A random variable $X$ is $\mathsf{SE}(\sigma^2,b)$ if it satisfies \begin{align}\label{eq:Bernstein.cond}
		\EE\qth{\abs{X-\EE[X]}^\ell}\leq \frac 12 \ell!\sigma^2b^{\ell-2},\quad \ell =2,3,\dots.
	\end{align} 
	\item If $X_1,\dots,X_k$ are independent $\mathsf{SE}(\sigma^2,b)$, then 
	$\sum_{i=1}^k X_i$ is $\mathsf{SE}(k\sigma^2,b)$.
	
	\end{itemize}
	Define $
		X_i= Q_i\log{Q_i\over \hat Q_i^{\mathsf{poi}}+\frac 1m}-Q_i+\hat Q_i^{\mathsf{poi}}+\frac 1m,i\in [k].
		$
	Then \prettyref{lmm:poisson.SE.conc} below shows that $X_i$'s are independent $\mathsf{SE}(\sigma^2,b)$ with $\sigma^2={c_1(\log m)^4\over m^2},b={c_2(\log m)^2\over n}$ for absolute constants $c_1,c_2$, and hence $\sum_{i=1}^k\pth{X_i-\EE[X_i]}$ is $\mathsf{SE}(k\sigma^2,b)$. In view of \eqref{eq:subexpo.conc} for the choice $c_0=8(c_1+c_2)$ this implies
	\begin{align}\label{eq:SE.conc.app}
		\PP\qth{\sum_{i=1}^k\pth{X_i-\EE[X_i]}\geq c_0{(\log t)^3\sqrt k\over m}}
		&\leq 2e^{-{c_0^2k(\log t)^6\over 2m^2\sigma^2}}
		+2e^{-{c_0\sqrt k(\log t)^3\over 2mb}}
		\leq \frac 1{t^3}.
	\end{align}
	Using $0 \leq y\log y-y+1\leq (y-1)^2,y> 0$ and $\EE\qth{\frac \lambda{\Poi(\lambda)+1}}=\sum_{v=0}^\infty {e^{-\lambda}\lambda^{v+1}\over(v+1)!}=1-e^{-\lambda}$
	\begin{align*}
		\EE\qth{\sum_{i=1}^k X_i}
		&\leq \EE\qth{\sum_{i=1}^k {\pth{Q_i-\pth{\hat Q_i^{\mathsf{poi}}+\frac 1m}}^2\over \hat Q_i^{\mathsf{poi}}+\frac 1m}}
		\nonumber\\
		&= \sum_{i=1}^kmQ_i^2\EE\qth{1\over m\hat Q_i^{\mathsf{poi}}+1}
		-1+\frac km
		=\sum_{i=1}^kQ_i\pth{1-e^{-mQ_i}}
		-1+\frac km
		\leq \frac km.
	\end{align*}
	Combining the above with \eqref{eq:SE.conc.app} we get \eqref{eq:KL-poisson} as required.
\end{proof}
	
	\begin{lemma}\label{lmm:poisson.SE.conc}
	There exist absolute constants $c_1,c_2$ such that the following holds. 
		 For any $p\in (0,1)$ and $nY\sim\Poi(np)$, $X=p\log{p\over Y+\frac 1n}-p+Y+\frac 1n$ is $\mathsf{SE}\pth{{c_1(\log n)^4\over n^2},{c_2(\log n)^2\over n}}$.
	\end{lemma}
	\begin{proof}
	Note that $X$ is a non-negative random variable.
		Since $\EE\qth{\pth{X-\EE[X]}^\ell}\leq 2^{\ell}\EE\qth{X^\ell}$ , by the Bernstein condition \eqref{eq:Bernstein.cond}, 
it suffices to show $\EE[X^\ell]\leq \pth{c_3\ell(\log n)^2\over n}^\ell,\ell=2,3,\dots$ for some absolute constant $c_3$.		
		guarantees the desired sub-exponential behavior. The analysis is divided into following two cases for some absolute constant $c_4\geq 24$.
		\paragraph{Case I $p\geq {c_4\ell\log n\over n}$:}
		
		Using Chernoff bound for Poisson \cite[Theorem 3]{Janson02}
		\begin{align} 
			\PP\qth{|\Poi(\lambda)-\lambda|>x}\leq 2e^{-{x^2\over 2(\lambda+x/3)}},\quad\lambda,x>0,
			\label{eq:poissontail}
		\end{align}
		we get 
		\begin{align}
		\PP\qth{|Y-p|> \sqrt{c_4\ell p\log n\over 4n}}
		&\leq 2\exp\pth{-{c_4n\ell p\log n\over 8np+2\sqrt{c_4n\ell p\log n}}}
		\nonumber\\
		&\leq 
		2\exp\pth{-{c_4\ell\log n\over {8+2\sqrt{c_4\ell\log n/ np}}}}
		\leq {1\over n^{2\ell}}
		\end{align}
		which implies $p/2\leq Y\leq 2p$ with probability at least $1-{n^{-2\ell}}$. Since  $0 \leq X\leq {(Y-p-\frac 1n)^2\over Y+\frac 1n}$, 
		we get
		$
			\EE[X^\ell]
			\lesssim {\pth{\sqrt{{c_4}\ell p\log n/4n}}^{2\ell}\over (p/2)^\ell}
			+{n^\ell\over n^{2\ell}}
			\lesssim \pth{c_4\ell\log n\over n}^\ell.
		$
	\paragraph{Case II $p< {c_4\ell\log n\over n}$:}
	\begin{itemize}
		\item
		On the event $\{Y>p\}$, we have 
		$X \leq Y + \frac{1}{n} \leq 2Y$, where the last inequality follows because $nY$ takes non-negative integer values. Since $X\geq 0$, we have
			$X^\ell\indc{Y> p}\leq (2Y)^\ell\indc{Y> p}$ for any $\ell\geq 2$.
		Using the Chernoff bound \prettyref{eq:poissontail}, we get $Y\leq {2c_4\ell \log n\over n}$ with probability at least $1-n^{-2\ell}$, which implies
		\begin{align*}
				\EE\qth{X^\ell\indc{Y\geq p}}
			&\leq \EE\qth{(2Y)^\ell\indc{Y> p,Y\leq {2c_4\ell \log n\over n}}}
			+\EE\qth{(2Y)^\ell\indc{Y> p,Y> {2c_4\ell \log n\over n}}}
			\nonumber\\
			&\leq\pth{4c_4\ell\log n\over n}^\ell
			+2^\ell\pth{\EE[Y^{2\ell}]\PP\qth{Y> {2c_4\ell \log n\over n}}}^{\frac 12}
			\leq
			\pth{c_5\ell\log n\over n}^\ell
		\end{align*}
		for absolute constant $c_5$.
		Here, the last inequality follows from Cauchy-Schwarz and using the Poisson moment bound \cite[Theorem 2.1]{Ahle21}:\footnote{For a result with less precise constants, 
		see also \cite[Eq.~(1)]{Ahle21} based on \cite[Corollary 1]{latala1997estimation}.}
		$\EE[(nY)^{2\ell}]\leq \pth{2\ell\over \log\pth{1+{2\ell\over np}}}^{2\ell}\leq \pth{c_6\ell\log n}^{2\ell}$ for some absolute constant $c_6$,
		with the second inequality applying the assumption  $p< {c_4\ell\log n\over n}$.

		\item As
		$
			X\indc{Y\leq p}\leq p\log n +\frac 1n
			\lesssim {\ell(\log n)^2\over n},
		$ 
		we get $\EE\qth{X^\ell\indc{Y\leq p}}
		\le \pth{c_7\ell(\log n)^2\over n}^\ell$ for some absolute constant $c_7$. 
	\end{itemize}
	\end{proof}


\subsubsection{Proof of \prettyref{cor:parametricrate}}\label{sec:k.state.lb}
We show the following monotonicity result of the prediction risk.
In view of this result, \prettyref{cor:parametricrate} immediately follows from \prettyref{thm:gamma2} and \prettyref{thm:gammak} (i).
Intuitively, the optimal prediction risk is monotonically increasing with the number of states; this, however, does not follow immediately due to the extra assumptions of irreducibility, reversibility, and prescribed spectral gap.

\begin{lemma}
		$\Risk_{k+1,n}(\gamma_0)\geq \Risk_{k,n}(\gamma_0)$ for all $\gamma_0\in(0,1),k\geq 2$.
\end{lemma}

\begin{proof}
Fix an $M\in\calM_{k}(\gamma_0)$ such that $\gamma_*(M)>\gamma_0$. Denote the stationary distribution $\pi$ such that $\pi M = \pi$.
Fix $\delta\in(0,1)$ and define a transition matrix $\tilde{M}$ with $k+1$ states as follows:
\[
\tilde{M} = \begin{pmatrix} (1-\delta) M & \delta \ones \\ (1-\delta) \pi & \delta \end{pmatrix}
\]
One can verify the following:
\begin{itemize}
	\item $\tilde{M}$ is irreducible and reversible;
	\item The stationary distribution for $\tilde{M}$ is $\tilde \pi=((1-\delta)\pi,\delta)$
	\item The absolute spectral gap of $\tilde{M}$ is $\gamma_*(\tilde{M})=(1-\delta) \gamma_*(M)$, so that $\tilde{M}\in\calM_{k+1}(\gamma_0)$  for all sufficiently small $\delta$.
	
	\item Let $(X_1,\ldots,X_n)$ and $(\tilde X_1,\ldots,\tilde X_n)$ be stationary Markov chains with transition matrices $M$ and $\tilde{M}$, respectively. 
	Then as $\delta\to 0$, $(X_1,\ldots,X_n)$ converges to $(\tilde X_1,\ldots,\tilde X_n)$ in law, i.e., the joint probability mass function converges pointwise.
\end{itemize}

	Next fix any estimator $\hat M$ for state space $[k+1]$. Note that without loss of generality we can assume $\hat \mymat(j|i)>0$ for all $i,j\in [k+1]$ for otherwise the KL risk is infinite. 
	Define 	$\hat{\mymat}^{\text{trunc}}$ as $\hat{\mymat}$ without the $k+1$-th row and column, and 
	denote by $\hat{\mymat}'$ its normalized version, namely, $\hat{\mymat}'(\cdot|i)={\hat{\mymat}^{\text{trunc}}(\cdot|i)\over 1-\hat M^{\text{trunc}}(k+1|i)}$ for $i=1,\ldots,k$.
	Then
	\begin{align*}
	 \EE_{\tilde X^n} \qth{D(\tilde M(\cdot|\tilde X_n)\|{\hat M(\cdot|\tilde X_n)})} 
	\xrightarrow{\delta\to0} & ~ 	\EE_{X^n} \qth{D(M(\cdot|X_n)\|{\hat M(\cdot|X_n)})}\\
	\geq & ~ 	\EE_{X^n} \qth{D(M(\cdot|X_n)\|{\hat M'(\cdot|X_n)})} \\
	\geq & ~ 	\inf_{\hat M} \EE_{X^n} \qth{D(M(\cdot|X_n)\|{\hat M(\cdot|X_n)})} 
	\end{align*}
	where in the first step we applied the convergence in law of $\tilde X^n$ to $X^n$ and the continuity of $P\mapsto D(P\|Q)$ for fixed componentwise positive $Q$;
	in the second step we used the fact that 
	for any sub-probability measure $Q=(q_i)$ and its normalized version $\bar Q = Q/\alpha$ with $\alpha = \sum q_i \leq 1$, we have
	$D(P\|Q) = D(P\|\bar Q) + \log \frac{1}{\alpha} \geq D(P\|\bar Q)$.
	Taking the supremum over $M \in \calM_k(\gamma_0)$ on the LHS and the supremum over $\tilde M\in\calM_{k+1}(\gamma_0)$ on the RHS, and finally the infimum over $\hat M$ on the LHS, we conclude 	$\Risk_{k+1,n}(\gamma_0)\geq \Risk_{k,n}(\gamma_0)$.
\end{proof}

\section{Higher-order Markov chains}
\label{sec:order-m}

\subsection{Basic setups}

In this section we prove \prettyref{thm:m-order-rate}. We start with some basic definitions for higher-order Markov chains.
Let $m\geq 1$. Let $X_1,X_2,\dots$ be an $m^{\text{th}}$-th order Markov chain with state space $\calS$ and transition matrix $M\in\reals^{\calS^m \times \calS}$ so that 
$\prob{X_{t+1}=x_{t+1}|X_{t-m+1}^t=x_{t-m+1}^t} =M(x_{t+1}|x_{t-m+1}^t)$ for all $t\geq m$. 
Clearly, the joint distribution of the process is specified by the transition matrix and the initial distribution, which is a joint distribution for $(X_1,\ldots,X_m)$.


	A distribution $\pi$ on $\calS^m$ is a \emph{stationary} distribution if $\{X_t: t\geq 1\}$ with $(X_1,\ldots,X_m)\sim \pi$ is a stationary process, that is, 
	\begin{equation}
	(X_{i_1+t},\ldots,X_{i_n+t}) \eqlaw (X_{i_1},\ldots,X_{i_n}), \quad \forall n,i_1,\ldots,i_n,t\in\naturals.
	\label{eq:stationarity}
	\end{equation}	
	It is clear that \prettyref{eq:stationarity} is equivalent to $(X_{1},\ldots,X_{m}) \eqlaw (X_{2},\ldots,X_{m+1})$. 
			In other words, $\pi$ is the solution to the linear system:
			\begin{align}
			\pi(x_1,\ldots,x_m) = \sum_{x_0 \in \calS} \pi(x_0,x_1,\ldots,x_{m-1}) M(x_m|x_1,\ldots,x_{m-1}), \quad \forall x_1,\ldots,x_m \in \calS.
			\label{eq:MC_stationarity}
			\end{align}
	Note that this implies, in particular, that $\pi$ as a joint distribution of $m$-tuples itself must satisfy those symmetry properties required by stationarity, such as all marginals being identical, etc. 
	
			%
			%
		%

Next we discuss reversibility. A random process $\{X_t\}$ is \emph{reversible} if 
for any $n$, 
\begin{equation}
X^n ~\eqlaw ~\overline{X^n},
\label{eq:reversibility}
\end{equation}
where $\overline{X^n}\eqdef (X_n,\ldots,X_1)$ denotes the time reversal of $X^n=(X_1,\ldots,X_n)$. 
Note that a reversible $m^{\text{th}}$-order Markov chain must be stationary. Indeed, 
	\begin{equation}
	(X_2,\ldots,X_{m+1}) \eqlaw (X_{m},\ldots,X_1) \eqlaw (X_{1},\ldots,X_m),
	\label{eq:rev2stat}
	\end{equation}	
	where the first equality follows from $(X_1,\ldots,X_{m+1}) \eqlaw (X_{m+1},\ldots,X_1)$.
	The following lemma gives a characterization for reversibility:
	\begin{lemma}\label{lmm:reversibility-general}
	An $m^{\text{th}}$-order stationary Markov chain is reversible if and only if \prettyref{eq:reversibility} holds for $n=m+1$, namely
	\begin{equation}
	\pi(x_1,\ldots,x_m) M(x_{m+1}|x_1,\ldots,x_m)  = \pi(x_{m+1},\ldots,x_2) M(x_1|x_{m+1},\ldots,x_2), \quad \forall x_1,\ldots,x_{m+1}\in\calS.
	\label{eq:rev-mth}
	\end{equation}	

\end{lemma}
\begin{proof}
First, we show that \prettyref{eq:reversibility} for $n=m+1$ implies that for $n\leq m$.
Indeed, 
\begin{align}
	(X_1,\ldots,X_n) \eqlaw (X_{m+1},\ldots,X_{m-n+2}) \eqlaw (X_n,\ldots,X_1)
	\label{eq:MC_reversibility-stationarity}
\end{align}
where the first equality follows from $(X_1,\ldots,X_{m+1}) \eqlaw (X_{m+1},\ldots,X_1)$ and the second applies stationarity.

Next, we show \prettyref{eq:reversibility} for $n=m+2$ and the rest follows from induction on $n$. 
Indeed, 
\begin{align*}
&~ \prob{(X_1,\ldots,X_{m+2}) = (x_1,\ldots,x_{m+2})} \\
= & ~ \pi(x_1,\ldots,x_m) M(x_{m+1}|x_1,\ldots,x_m) M(x_{m+2}|x_2,\ldots,x_{m+1}) \\
\stepa{=} & ~ \pi(x_{m+1},\ldots,x_2) M(x_1|x_{m+1},\ldots,x_2) M(x_{m+2}|x_2,\ldots,x_{m+1}) \\
\stepb{=} & ~ \pi(x_2,\ldots,x_{m+1}) M(x_1|x_{m+1},\ldots,x_2) M(x_{m+2}|x_2,\ldots,x_{m+1}) \\
\stepc{=} & ~ \pi(x_{m+2},\ldots,x_3) M(x_{2}|x_{m+2},\ldots,x_3) M(x_1|x_{m+1},\ldots,x_2) \\
= & ~ \prob{(X_1,\ldots,X_{m+2}) = (x_{m+2},\ldots,x_1)} = \prob{(X_{m+2},\ldots,X_1) = (x_1,\ldots,x_{m+2})}.
\end{align*}
where (a) and (c) apply \prettyref{eq:reversibility} for $n=m+1$, namely, \prettyref{eq:rev-mth};
(b) applies \prettyref{eq:reversibility} for $n=m$.
\end{proof}

In view of the proof of \prettyref{eq:rev2stat}, we note that any distribution $\pi$ on $\calS^m$ and $m^{\text{th}}$-order transition matrix $M$ satisfying $\pi(x^m)=\pi(\overline{x^m})$ and \eqref{eq:rev-mth} also satisfy \eqref{eq:MC_stationarity}. This implies such a $\pi$ is a stationary distribution for $M$. In view of \prettyref{lmm:reversibility-general} the above conditions also guarantee reversibility. This observation can be summarized in the following lemma, which will be used to prove the reversibility of specific Markov chains later. 

\begin{lemma}
	\label{lmm:st-rev}
	Let $M$ be a $k^m\times k$ stochastic matrix describing transitions from $\calS^m$ to $\calS$. Suppose that $\pi$ is a distribution on $\calS^m$ such that $\pi(x^m) = \pi(\overline{x^m})$ and $\pi(x^m)M(x_{m+1}|x^m) = \pi(\overline{x_2^{m+1}})M(x_1|\overline{x_2^{m+1}})$. Then $\pi$ is the stationary distribution of $M$ and the resulting chain is reversible.
\end{lemma}

For $m^{\rm th}$-order stationary Markov chains, the optimal prediction risk is defined as as
\begin{align}
	{\Risk}_{k,n,m} 
	&\eqdef \inf_{\hat M} \sup_{M} \Expect[D(M(\cdot|X_{n-m+1}^n) \| \hat M(\cdot|X_{n-m+1}^n))]
	\nonumber\\
	&=\inf_{\hat M} \sup_{M} \sum_{x^m \in\calS^m} \Expect[D(M(\cdot|x^m) \| \hat M(\cdot|x^m)) \indc{X_{n-m+1}^n=x^m}]
	\label{eq:riskkn2}
\end{align}
where the supremum is taken over all $k^m\times k$ stochastic matrices $M$ and the trajectory is initiated from the stationary distribution. 
In the remainder of this section we will show the following result, completing the proof of \prettyref{thm:m-order-rate} previously announced in \prettyref{sec:intro}.

\begin{theorem}\label{thm:optimal-m}
For all $m\geq 2$, there exist a constant $C_m>0$ such that for all $2\leq k\leq n^{\frac{1}{m+1}}/C_m$,
	$$
	{k^{m+1}\over C_mn}\log\pth{n\over k^{m+1}}
	\leq \Risk_{k,n,m}
	\leq {C_mk^{m+1}\over n}\log\left(\frac n{k^{m+1}}\right).
	$$
	Furthermore, the lower bound holds even when the Markov chains are required to be reversible.
\end{theorem}

\subsection{Upper bound}
We prove the upper bound part of the preceding theorem, using only stationarity (not reversibility). 
We rely on techniques from \cite[Chapter 6, Page 486]{csiszar2004information} for proving redundancy bounds for the $m^{\text{th}}$-order chains. Let $Q$ be the probability assignment given by
\begin{align}
	Q(x^n)
	=\frac 1{k^m}\prod_{a^m\in\calS^m}
	{\prod_{j=1}^kN_{a^mj}!\over k\cdot(k+1)\cdots(N_{a^m}+k-1)},
\end{align}
where $N_{a^mj}$ denotes the number of times the block $a^mj$ occurs in $x^n$, and $N_{a^m}=\sum_{j=1}^k N_{a^mj}$ is the number of times the block $a^m$ occurs in $x^{n-1}$. This probability assignment corresponds to the add-one rule
\begin{align}
	Q(j|x^{n})
	= \hat M_{x^n}^{+1}(j|x_{n-m+1}^n)
	={N_{x_{n-m+1}^nj}+1\over N_{x_{n-m+1}^n}+k}.
\end{align}
Then in view of \prettyref{lmm:riskred}, the following lemma proves the desired upper bound in \prettyref{thm:optimal-m}.

\begin{lemma}
	\label{lmm:red-addone-order-m}	Let $\Red(Q_{X^n})$ be the redundancy of the $m^{\text{th}}$-order Markov chain, as defined in \prettyref{sec:riskred-bound}, and $X^m$ be the corresponding observed trajectory. Then
	$$\Red(Q_{X^n}) \leq \frac 1{n-m}\sth{{k^m(k-1)}\left[\log \left(1+\frac{n-m}{k^m(k-1)}\right)+1\right] + m\log k}.
	$$
\end{lemma}


\begin{proof}
	We show that for every Markov chain with transition matrix $M$ and initial distribution $\pi$ on $\calS^m$, and every trajectory $(x_1,\cdots,x_n)$, it holds that
	\begin{align}\label{eq:MC_pointwise-red-m}
		\log\frac{\pi(x_1^m)\prod_{t=m}^{n-1} M(x_{t+1}|x^t_{t-m+1}) }{Q(x_1,\cdots,x_n)} \le k^m(k-1)\left[\log \left(1+\frac{n-m}{k^m(k-1)}\right)+1\right] + m\log k,
	\end{align}
	where $M(x_{t+1}|x^t_{t-m+1})$ the transition probability of going from $x^t_{t-m+1}$ to $x_{t+1}$.
	Note that 
	\[
	\prod_{t=m}^{n-1} M(x_{t+1}|x^t_{t-m+1}) = \prod_{a^{m+1}\in\calS^{m+1}}M(a_{m+1}|a^m)^{N_{a^{m+1}}} \le \prod_{a^{m+1}\in\calS^{m+1}} (N_{a^{m+1}}/N_{a^m})^{N_{a^{m+1}}},
	\]
	where the last inequality follows from $\sum_{a_{m+1\in\calS}} \frac{N_{a^{m+1}}}{N_{a^m}}\log \frac{N_{a^{m+1}}}{N_{a^m}M(a_{m+1}|a^m)}\geq 0$ for each $a^m$, by the non-negativity of the KL divergence.
	Therefore, we have
	\begin{align}\label{eq:MC_likelihood-ratio-m}
		\frac{\pi(x_1^m)\prod_{t=m}^{n-1} M(x_{t+1}|x_{t-m+1}^t) }{Q(x_1,\cdots,x_n)} \le k^m\cdot \prod_{a^m\in\calS^m} \frac{k\cdot (k+1)\cdot \cdots\cdot (N_{a^m}+k-1)}{N_{a^m}^{N_{a^m}}} \prod_{a_{m+1}\in \calS}\frac{ N_{a^{m+1}}^{N_{a^{m+1}}} }{N_{a^{m+1}}!}.
	\end{align}
	Using \eqref{eq:MC_claim} we continue \eqref{eq:MC_likelihood-ratio-m} to get
	\begin{align*}
		\log \frac{\pi(x_1)\prod_{t=m}^{n-1} M(x_{t+1}|x_t) }{Q(x_1,\cdots,x_n)} &\le m\log k+\sum_{a^m\in\calS^m} \log \frac{k\cdot (k+1)\cdot \cdots\cdot (N_{a^m}+k-1)}{N_{a^m}!} \\
		&= m\log k + \sum_{a^m\in\calS^m} \sum_{\ell = 1}^{N_{a^m}} \log\left(1+\frac{k-1}{\ell}\right)\\
		&\le m\log k + \sum_{a^m\in\calS^m} \int_0^{N_{a^m}} \log\left(1+\frac{k-1}{x}\right)dx \\
		&= m\log k + \sum_{a^m\in\calS^m} \left((k-1)\log\left(1+\frac{N_{a^m}}{k-1}\right) + N_{a^m}\log\left(1+\frac{k-1}{N_{a^m}}\right) \right) \\
		&\stepa{\le} k^m(k-1)\log \left(1+\frac{n-m}{k^m(k-1)}\right) + k^m(k-1) + m\log k,
	\end{align*}
	where (a) follows from the concavity of $x\mapsto \log x$, $\sum_{a^m\in\calS^m} N_{a^m}=n-m+1$, and $\log(1+x)\le x$. 
\end{proof}

\subsection{Lower bound}

\subsubsection{Special case: $m\geq 2,k=2$}

We only analyze the case $m=2$, i.e.~second-order Markov chains with binary states, as the lower bound still applies to the case of $m\ge 3$ case. The transition matrix for second-order chains is given by a $k^2\times k$ stochastic matrices $M$ that gives the transition probability from the ordered pairs $(i,j)\in \calS\times\calS$ to some state $\ell\in\calS$:
\begin{align}
	M(\ell|ij)=\PP\qth{X_3=\ell|X_1=i,X_2=j}.
\end{align}
Our result is the following.
\begin{theorem}
	${\Risk}_{2,n,2}=\Theta\pth{\log n\over n}$.
\end{theorem}
\begin{proof}
	
	The upper bound part has been shown in \prettyref{lmm:red-addone-order-m}.
	For the lower bound, consider the following one-parametric family of transition matrices (we replace $\calS$ by $\sth{1,2}$ for simplicity of the notation)
		\begin{align}
			\tilde\calM=\sth{M_p=
				\bbordermatrix{~ & 1 & 2 \cr
					11 & 1-\frac 1n & \frac 1n \cr
					21 & \frac 1n & 1-\frac 1n \cr
					12 & 1-p & p \cr
					22 & p & 1-p}
				: 0\leq p\leq 1}
				\label{eq:Mp-2nd}
		\end{align}
		and place a uniform prior on $p\in[0,1]$. One can verify that each $M_p$ has the uniform stationary distribution over the set $\sth{1,2}\times\sth{1,2}$ and the chains are reversible. 
		
	Next we introduce the set of trajectories based on which we will lower bound the prediction risk. Analogous to the set $\calX=\cup_{t=1}^n\calX_t$ defined in \eqref{eq:MC_Xt-sets} for analyzing the first-order chains, we define
		\begin{align}
			\calV=\sth{1^{n-t} z^{t}:z_1=z_2=z_t=2,z_{i}^{i+1}\neq 11, i\in[t-1], t=4,\dots,n-2} \subset \{1,2\}^n.
		\end{align}
	In other words, the sequences in $\calV$ start with a string of 1's before transitioning into two consecutive 2's, are forbidden to have no consecutive 1's thereafter, and finally end with 2.  
%
	
	To compute the probability of sequences in $\calV$, we need the following preparations.
	Denote by $\oplus$ the the operation that combines any two blocks from $\sth{22,212}$ via merging the last symbol of the first block and the first symbol of the second block, for example, 
	$	22\oplus 212=2212,
		22\oplus 22\oplus 22=2222$.
		Then for any $x^n\in \calV$ we can write it in terms of the initial all-1 string, followed by alternating run of blocks from $\{22,212\}$ with the first run being of the block 22 (all the runs have positive lengths), combined with the merging operation $\oplus$:
		\begin{align}
			\label{eq:MC_calV-decomposition}
		x^n=\underbrace{1\dots1}_{\text{all ones}} \underbrace{22\oplus 22\dots\oplus 22}_{p_1 \text{ many } 22}\oplus \underbrace{212\oplus 212\dots\oplus 212}_{p_2 \text{ many } 212} \oplus\underbrace{ 22\oplus 22\dots\oplus 22}_{p_3 \text{ many } 22}\oplus \underbrace{212\oplus 212\dots\oplus 212}_{p_4 \text{ many }212}\oplus 22\oplus \dots.
	\end{align}
	Let the vector $(q_{22\to 22},q_{22\to 212},q_{212\to 22},q_{212\to 212})$ denotes the transition probabilities between blocks in $\sth{22,212}$ (recall the convention that the two blocks overlap in the symbol 2). Namely, according to \prettyref{eq:Mp-2nd},
		\begin{align*}
			q_{22\to 22}
			&=\PP\qth{X_3=2,X_2=2|X_2=2,X_1=2}
			=M(2|22)=1-p
			\\
			q_{22\to 212}
			&=\PP\qth{X_4=2,X_3=1,X_2=2|X_2=2,X_1=2}
			=M(2|21)M(1|22)=\pth{1-\frac 1n}p
			\\
			q_{212\to 22}
			&=\PP\qth{X_4=2,X_3=2|X_3=2,X_2=1,X_1=2}
			=M(2|12)=p
			\\
			q_{212\to 212}
			&=\PP\qth{X_5=2,X_4=1,X_3=2|X_3=2,X_2=1,X_1=2}
			=M(2|21)M(1|12)=\pth{1-\frac 1n}(1-p).
			\\
		\end{align*}
		Given any $x^n\in \calV$ we can calculate its probability under the law of $M_p$ using frequency counts $\bm{F}(x^n)=\pth{F_{111},F_{22\to 22},F_{22\to 212},F_{212\to 22},F_{212\to 212}}$, defined as
		\begin{gather*}
			F_{111}
			=\sum_{i}\indc{x_{i}=1,x_{i+1}=1,x_{i+2}=1},\quad
			F_{22\to 22}
			=\sum_{i}\indc{x_{i}=2,x_{i+1}=2,x_{i+2}=2},\\
			F_{22\to 212}
			=\sum_{i}\indc{x_{i}=2,x_{i+1}=2,x_{i+2}=1,x_{i+3}=2},
			\quad
			F_{212\to 22}
			=\sum_{i}\indc{x_{i}=2,x_{i+1}=1,x_{i+2}=2,x_{i+3}=2},\\
			F_{212\to 212}
			=\sum_{i}\indc{x_{i}=2,x_{i+1}=1,x_{i+2}=2,x_{i+3}=1,x_{i+4}=2}.
		\end{gather*}
		Denote $\mu(x^n|p)=\PP\qth{X^n=x^n|p}$. Then for each $x^n\in\calV$ with $\bm F(x^n)=\bm F$ we have
		\begin{align}\label{eq:MC_prob-F}
			&\mu(x^n|p)
			\nonumber\\
			&=\PP(X^{F_{111}+2}=1^{F_{111}+2})M(2|11)M(2|12)\prod_{a,b\in\sth{22,212}}q_{a\to b}^{F_{a\to b}}
			\nonumber\\
			&=\frac 14\pth{1-\frac 1n}^{F_{111}}\frac 1n \cdot p \cdot 
			p^{F_{212\to 22}}\sth{p\pth{1-\frac 1n}}^{F_{22\to 212}}(1-p)^{F_{22\to 22}}\sth{(1-p)\pth{1-\frac 1n}}^{F_{212\to 212}}
			\nonumber\\
			&=\frac 14\pth{1-\frac 1n}^{F_{111}+F_{22\to 212}+F_{212\to 212}}
			\frac 1n p^{y+1}(1-p)^{f-y}
		\end{align}
		where $y=F_{212\to 22}+F_{22\to 212}$ denotes the number of times the chain alternates between runs of 22 and runs of 212, and $f=F_{212\to 22}+F_{22\to 212}+F_{212\to 212}+F_{22\to 22}$ denotes the number of times the chain jumps between blocks in $\{22,212\}$.
		
		 Note that the range of $f$ includes all the integers in between 1 and $(n-6)/2$. This follows from the definition of $\calV$ and the fact that if we merge either 22 or 212 using the operation $\oplus$ at the end of any string $z^{t}$ with $z_t=2$, it increases the length of the string by at most 2. Also, given any value of $f$ the value of $y$ ranges from 0 to $f$. 
		\begin{lemma}
		\label{lmm:countyf}	
		The number of sequences in $\calV$ corresponding to a fixed pair $(y,f)$ is $\binom fy$. 	
		\end{lemma}
		\begin{proof}
			Fix $x^n\in \calV$ and let that $p_{2i-1}$ is the length of the $i$-th run of 22 blocks and $p_{2i}$ is the length of the $i$-th run of 212 blocks in $x^n$ as depicted in \eqref{eq:MC_calV-decomposition}.
		The $p_i$'s are all non-negative integers. There are total $y+1$ such runs and the $p_i$'s satisfy $\sum_{i=1}^{y+1}p_i=f+1$, as the total number of blocks is one more than the total number of transitions. Each positive integer solution to this equation $\sth{p_i}_{i=1}^{y+1}$ corresponds to a sequence $x^n\in\cal V$ and vice versa. The total number of such sequences is $\binom fy$. 
		\end{proof}
		
We are now ready to compute the Bayes estimator and risk.
		For any $x^n\in\calV$ with a given $(y,f)$, the Bayes estimator of $p$ with prior $p\sim\Unif[0,1]$ is 
		\[
		\hat p(x^n) = \Expect[p|x^n] = \frac{\Expect[p \cdot \mu(x^n|p)]}{\Expect[\mu(x^n|p)]}
		\overset{\eqref{eq:MC_prob-F}}{=}{y+2\over f+3}.
		\]
		Note that the probabilities $\mu(x^n|p)$ in \eqref{eq:MC_prob-F} can be bounded from below by $\frac 1{4en} p^{y+1}(1-p)^{f-y}$. Using this, for each $x^n\in \calV$ with given $y,f$ we get the following bound on the integrated squared error for a particular sequence $x^n$
		\begin{align}
			&\int_0^1 \mu(x^n|p)(p-\hat p(x^n))^2 dp
			\nonumber\\
			&\geq \frac 1{4en} \int_0^1 p^{y+1}(1-p)^{f-y}\pth{p-{y+2\over f+3}}^2dp
			=\frac 1{4en} {(y+1)!(f-y)!\over (f+2)!}{(y+2)(f-y+1)\over (f+3)^2(f+4)}
			\label{eq:bayes-2nd}
		\end{align}
		where the last equality followed by noting that the integral is the variance of a $\text{Beta}(y+2,f-y+1)$ random variable without its normalizing constant. 
		
		Next we bound the risk of any predictor by the Bayes error. Consider any predictor $\hat \M(\cdot|ij)$ (as a function of the sample path $X$) for transition from $ij$, $i,j\in\sth{1,2}$. By the Pinsker's inequality, we conclude that
		\begin{align}
			D(\M(\cdot|12) \| \hat \M(\cdot|12))
			\geq  \frac{1}{2} \|\M(\cdot|12)-\hat \M(\cdot|12)\|_{\ell_1}^2 \geq \frac{1}{2} (p-\hat \M(2|12))^2
		\end{align}
		and similarly, $D(\M(\cdot|22) \| \hat \M(\cdot|22)) \geq \frac{1}{2}(p-\hat \M(1|22))^2$.
		Abbreviate $\hat \M(2|12) \equiv \hat p_{12}$ and $\hat \M(1|22) \equiv \hat p_{22}$, both functions of $X$. 
		Using 
	\prettyref{eq:bayes-2nd} and \prettyref{lmm:countyf}, we have
		\begin{align}
			& \sum_{i,j=1}^3 \Expect[D( \M(\cdot|ij)\|\hat \M(\cdot|ij))) \indc{X_{n-1}^n=ij} ]	\nonumber \\
			&\geq  ~ \frac 12 \Expect\qth{(p-\hat p_{12})^2 \indc{X_{n-1}^n=12,X^n\in\calV} + (p-\hat p_{22})^2 \indc{X_{n-1}^n=22,X^n\in\calV}}	\nonumber \\
			&\geq  ~ \frac 12\int_0^1\qth{\sum_{\bm{F}}\sum_{x^n\in\calV:\bm{F}(x^n)=\bm{F}} \mu(x^n|p)\pth{(p-\hat p_{12})^2  \indc{x_{n-1}^n=12} + (p-\hat p_{22})^2 \indc{x_{n-1}^n=22} }}dp\nonumber \\
			&\geq ~ \frac 12\int_0^1\qth{\sum_{\bm{F}}\sum_{x^n\in\calV:\bm{F}(x^n)=\bm{F}} \mu(x^n|p) (p-\hat p(x^n))^2  }dp \nonumber \\
			&\geq ~ \frac 12\sum_{f=1}^{\frac {n-6}2}\sum_{y=0}^{f} 
			\binom fy  \frac 1{4en} {(y+1)!(f-y)!\over (f+2)!}{(y+2)(f-y+1)\over (f+3)^2(f+4)}
			\nonumber\\
			& \geq ~ \frac 1{8en} \sum_{f=1}^{\frac {n-6}2}\sum_{y=0}^{f} 
			{y+1\over (f+2)(f+1)}{(y+2)(f-y+1)\over (f+3)^2(f+4)} 
			\geq ~ \Theta\pth{\frac 1n} \sum_{f=1}^{\frac {n-6}2}
			\sum_{y=\frac f4}^{\frac f3} \frac 1{f^2}
			= ~ \Theta\pth{\log n\over n}.
		\end{align}	
	\end{proof}

\subsubsection{General case: $m\geq 2, k\geq 3$}

We will prove the following.

\begin{theorem}\label{thm:optimal-lower-order-m}
	For absolute constant $C$, we have 
	$$\Risk_{k,n,m}
	\geq \frac 1{2^{m+4}}\pth{\frac 12-{2^m-2\over n}}
	\pth{1-\frac 1n}^{n-2m+1}{(k-1)^{m+1}\over n}\log\pth{{1\over  2^{2m+8}\cdot 3\pi e(m+1)}\cdot {n-m\over (k-1)^{m+1} }}.$$
\end{theorem}
For ease of notation let $\calS=\sth{1,\dots,k}$.
Denote ${\tilde \calS}=\sth{2,\dots,k}$. 
Consider an 
$m^{\text{th}}$-order transition matrix $M$ of the following form:
 

\begin{align}
	M(s|x^m)=\text{\begin{tabular}{|c|c|c|}
			\hline 
			 \multirow{2}{*}{Starting string $x^m$} & 
			\multicolumn{2}{c|}{Next state} \\\cline{2-3} 
			& $s=1$ & $s\in\sth{2,\dots,k}$ \\
			\hline & & \\
			$1^m$ & $1- \frac 1n$ & $\frac 1{n(k-1)}$ \\ 
			& & \\
			\hline
			& & \\
			\multirow{2}{*}{$1x^{m-1}, x^{m-1}\in{\tilde \calS}^{m-1}$} & $1-b$ & $\frac b{(k-1)}$\\
			& & \\
			\hline & & \\
			\multirow{2}{*}{$x^{m}\in{\tilde \calS}^{m}$} &$\frac 1n$ & {\mbox{$\pth{1-\frac 1n}T(s|x^m)$}}
			\\
			& & \\
			\hline & &\\
			\multirow{2}{*}{$x^m\notin \sth{1^m,1{\tilde \calS}^{m-1},{\tilde \calS}^{m}}$} & $\frac 12$ & $\frac 1{2(k-1)}$
			\\
			& & \\
			\hline
		\end{tabular}
	},
	\quad b=\frac 12-{2^m-2\over n}.
	\label{eq:MC_M_construction2}
\end{align}\\

\noindent 
Here $T$ is a $(k-1)^m\times (k-1)$ transition matrix for an $m^{\rm th}$-order Markov chain with state space $\tilde \calS$, satisfying the following property:
\begin{enumerate}[label=(P)]
	\item \label{pt:MC_P} $T(x_{m+1}|x^{m}) = T(x_1|\overline{x_2^{m+1}}),\quad \forall x^{m+1}\in\tilde \calS^{m+1}$.
\end{enumerate}

\begin{lemma}\label{lmm:MC_stationary-T}
	Under the condition \ref{pt:MC_P}, the transition matrix $T$ has a stationary distribution that is uniform on ${\tilde \calS}^m$.
	Furthermore, the resulting $m^{\rm th}$-order Markov chain is reversible (and hence stationary).
\end{lemma}
\begin{proof}
	We prove this result using \prettyref{lmm:st-rev}. Let $\pi$ denote the uniform distribution on $\tilde \calS^m$, i.e., $\pi(x^m)=\frac 1{(k-1)^m}$ for all $x^m\in\tilde \calS^m$. Then for any $x^m\in\tilde \calS^m$ the condition $\pi(x^m)=\pi(\overline{x^m})$ follows directly and $\pi(x^m)T(x_{m+1}|x^m)
	=\pi(\overline{x_2^{m+1}})T(x_{1}|\overline{x_2^{m+1}})$ follows from the assumption \ref{pt:MC_P}.
\end{proof}

Next we address the stationarity and reversibility of the chain with the bigger transition matrix $M$ in \prettyref{eq:MC_M_construction2}:

\begin{lemma}\label{lmm:MC_stationary-M}
Let $M$ be defined in \prettyref{eq:MC_M_construction2}, wherein the transition matrix $T$ 
satisfies  the condition \ref{pt:MC_P}. Then $M$ has a stationary distribution given by
	\begin{align}
		\pi(x^m)=
		\begin{cases}
		\frac{1}{2}	&  x^m  = 1^m\\
		\frac{b}{(k-1)^m} & x^{m}\in{\tilde \calS}^{m} \\
		\frac{1}{n(k-1)^{d(x^{m})}} & \text{otherwise}
		\end{cases}
		\label{eq:pi-mth}
	\end{align}
		where  $d(x^m)\triangleq\sum_{i=1}^m\indc{x_i\in{\tilde \calS}}$ and $b=\frac 12-{2^m-2\over n}$ as in 
	\prettyref{eq:MC_M_construction2}. 
	Furthermore, the $m^{\rm th}$-order Markov chain with initial distribution $\pi$ and transition matrix $M$ is reversible.
\end{lemma}
\begin{proof}
	Note that the choice of $b$ guarantees that $\sum_{x^m\in\calS^m}\pi(x^m)=1$. 
	Next we again apply \prettyref{lmm:st-rev} to verify stationarity and reversibility.
	First of all, since 
	$d(x^m)=d(\overline{x^m})$, we have $\pi(x^m)=\pi(\overline{x^m})$ for all $x^m\in\calS^m$ . 
	Next we check the condition $\pi(x^m)M(x_{m+1}|x^m) = \pi(\overline{x_2^{m+1}})M(x_1|\overline{x_2^{m+1}})$. 
	For the sequence $1^{m+1}$ the claim is easily verified. For the rest of the sequences we have the following.
	\begin{itemize}
		\item \textbf{Case 1 ($x^{m+1}\in {\tilde \calS}^{m+1}$):}
		Note that $x^{m+1}\in\tilde \calS^{m+1}$ if and only if $x^m,{\overline{x_2^{m+1}}}\in\tilde \calS^m$. This implies
		\begin{align*}
			\pi(x^m)M(x_{m+1}|x^m)
			&=
			\frac b{(k-1)^m}\pth{1-\frac 1n}T(x_{m+1}|x^m)
			\nonumber\\
			&=\frac b{(k-1)^m}\pth{1-\frac 1n}T(x_1|\overline{x_2^{m+1}})
			=\pi(\overline{x_2^{m+1}})M(x_1|\overline{x_2^{m+1}}).
		\end{align*}
		\item \textbf{Case 2 ($x^{m+1}\in 1{\tilde \calS}^m$ or $x^{m+1}\in {\tilde \calS}^m1$):} By symmetry it is sufficient to analyze the case $x^{m+1}\in 1{\tilde \calS}^m$. Note that in the sub-case $x^{m+1}\in 1{\tilde \calS}^m$, $x^m\in 1\tilde \calS^{m-1}$ and $\overline{x_2^{m+1}}\in\tilde \calS^{m}$. This implies \begin{align}
			\pi(x^m)=\frac 1{n(k-1)^{m-1}},
			&\quad 
			M(x_{m+1}|x^m)=\frac b{k-1},
			\nonumber\\
			\pi(\overline{x_2^{m+1}})=\frac b{(k-1)^{m}},
			& \quad M(x_{1}|\overline{x_2^{m+1}})=\frac 1n.
		\end{align}
		In view of this we get
		$\pi(x^m)M(x_{m+1}|x^m)
		=\pi(\overline{x_2^{m+1}})M(x_1|\overline{x_2^{m+1}}).$
		\item \textbf{Case 3 ($x^{m+1}\notin 1^{m+1}\cup {\tilde \calS}^{m+1}\cup1{\tilde \calS}^{m}\cup{\tilde \calS}^{m}1$):} 
		
		\noindent Suppose that $x^{m+1}$ has $d$ many elements from ${\tilde \calS}$. Then $x^m,x_2^{m+1}\notin \sth{1^m,{\tilde \calS}^m}$. We have the following sub-cases.
		\begin{itemize}
			\item If $x_1=x_{m+1}=1$, then both $x^m,x_2^{m+1}$ have exactly $d$ elements from ${\tilde \calS}$. This implies $\pi(x^m)=\pi(\overline{x_2^{m+1}})=\frac 1{n(k-1)^d}$ and $M(x_{m+1}|x^m)=M(x_1|\overline{x_2^{m+1}})=\frac 12$. 
			\item If $x_1,x_{m+1}\in{\tilde \calS}$, then both $x^m,x_2^{m+1}$ have exactly $d-1$ elements from ${\tilde \calS}$. This implies $\pi(x^m)=\pi(\overline{x_2^{m+1}})=\frac 1{n(k-1)^{d-1}}$ and $M(x_{m+1}|x^m )=M(x_1|\overline{x_2^{m+1}})=\frac 1{2(k-1)}$. 
			\item If $x_1=1,x_{m+1}\in {\tilde \calS}$, then $x^m$ has $d-1$ elements from $\tilde S$ and $x_2^{m+1}$ has $d$ elements from $\calS$. This implies $\pi(x^m)=\frac 1{n(k-1)^{d-1}},\pi(\overline{x_2^{m+1}})=\frac 1{n(k-1)^{d}}$ and $M(x_{m+1}|x^m)=\frac 1{2(k-1)},M(x_1|\overline{x_2^{m+1}})=\frac 12$. 
			\item If $x_1\in {\tilde \calS},x_{m+1}=1$, then $x^m$ has $d$ elements from $\tilde S$ and $x_2^{m+1}$ has $d-1$ elements from $\calS$ then $\pi(x^m)=\frac 1{n(k-1)^{d}},\pi(\overline{x_2^{m+1}})=\frac 1{n(k-1)^{d-1}}$ and $M(x_{m+1}|x^m)=\frac 12,M(x_1|\overline{x_2^{m+1}})=\frac 1{2(k-1)}$. 
		\end{itemize} For all these sub-cases we have $\pi(x^m)M(x_{m+1}|x^m) = \pi(\overline{x_2^{m+1}})M(x_1|\overline{x_2^{m+1}})$ as required.
	\end{itemize}
	
	This finishes the proof.
\end{proof}

Let $(X_1,\ldots,X_n)$ be the trajectory of a stationary Markov chain with transition matrix $M$ as in \eqref{eq:MC_M_construction2}. 
We observe the following properties:
\begin{enumerate}[label=(R\arabic*)]
	\item \label{pt:MC_21} This Markov chain is irreducible and reversible. Furthermore, the stationary distribution $\pi$ assigns probability $\frac 12$ to the initial state $1^m$.
	\item \label{pt:MC_22} For $m\leq t\leq n-1$, let $\calX_t$ denote the collections of trajectories $x^n$ such that $x_1,x_2,\cdots,x_t=1$ and $x_{t+1},\cdots,x_n\in {\tilde \calS}$. Then using \prettyref{lmm:MC_stationary-M}
	\begin{align}\label{eq:MC_Xt_prob2}
		\PP(X^n\in\calX_t)&= \PP(X_1=\cdots=X_t=1) 
		\cdot
		\PP(X_{t+1}\neq 1|X_{t-m+1}^{t}=1^{m})
		\nonumber\\
		&\quad \cdot
		\prod_{i=2}^{m-1}
		\PP(X_{t+i}\neq 1|X_{t-m+i}^{t}=1^{m-i+1},X_{t+1}^{t+i-1}\in{\tilde \calS}^{i-1})
		\nonumber\\
		&\quad \cdot\PP(X_{t+m}\neq 1|X_t=1,X_{t+1}^{t+m-1}\in{\tilde \calS}^{m-1})
		\cdot \prod_{s=t+m}^{n-1} \PP(X_{s+1}\neq 1|X_{s-m+1}^s\in{\tilde \calS}^m) \nonumber \\
		&= \frac{1}{2}\cdot \left(1-\frac{1}{n}\right)^{t-m}\cdot \frac{b}{n2^{m-2}}\cdot \left(1-\frac{1}{n}\right)^{n-m-t} =\frac b{n2^{m-1}}\pth{1-\frac 1n}^{n-2m}.
	\end{align}
	Moreover, this probability does not depend of the choice of $T$; 
	
	\item \label{pt:MC_23} Conditioned on the event that $X^n\in\calX_t$, the trajectory $(X_{t+1},\cdots,X_n)$ has the same distribution as a length-$(n-t)$ trajectory of a stationary $m^{\text{th}}$-order Markov chain with state space ${\tilde \calS}$ and transition probability $T$, and the uniform initial distribution.
	Indeed,
	\begin{align*}
		& \prob{X_{t+1}=x_{t+1},\ldots,X_n=x_n|X^n\in\calX_t} \\
		&= \frac{\frac{1}{2}\cdot \left(1-\frac{1}{n}\right)^{t-m} \cdot \frac b{n2^{m-2}(k-1)^{m}} \prod_{s=t+m}^{n-1} \pth{1-\frac1n}T(x_{s+1}|x_{s-m+1}^s) }{\frac{b}{n2^{m-1}}\left(1-\frac{1}{n}\right)^{n-2m}} \\
		&= \frac{1}{(k-1)^m} 
		\prod_{s=t+m}^{n-1} T(x_{s+1}|x_{s-m+1}^s).
	\end{align*}
\end{enumerate}

\paragraph{Reducing the Bayes prediction risk to mutual information}

Consider the following Bayesian setting, we first draw $T$ from some prior satisfying property \ref{pt:MC_P}, then generate the stationary $m^{\rm th}$-order Markov chain $X^n=(X_1,\ldots,X_n)$ with state space $[k]$ and transition matrix $M$ in \eqref{eq:MC_M_construction2} and stationary distribution $\pi$ in \prettyref{eq:pi-mth}. The following lemma lower bounds the Bayes prediction risk.
\begin{lemma}\label{lmm:riskred_markov2}
	Conditioned on $T$, let $Y^n=(Y_1,\ldots,Y_n)$ denote an $m^{\rm th}$-order stationary Markov chain on state space 
	${\tilde \calS}=\{2,\ldots,k\}$ with transition matrix $T$ and uniform initial distribution. Then
	\begin{align*}
		&\inf_{\widehat{M}}\EE_{T}\left[\EE[D(M(\cdot|X_{n-m+1}^n)\| \widehat{M}(\cdot|X_{n-m+1}^n))) ]\right] \nonumber\\
		&\ge \frac{b(n-1)}{n^22^{m-1}}\pth{1-\frac 1n}^{n-2m}\left(I(T;Y^{n-m}) - m\log (k-1)\right). 
	\end{align*}
\end{lemma}

\begin{proof}
	We first relate the Bayes estimator of $M$ and $T$ (given the $X$ and $Y$ chain respectively).
	For each $m\leq t\leq n$, denote by $\hat M_t=\hat M_t(\cdot|x^t)$ 
	the Bayes estimator of $M(\cdot|x_{t-m+1}^t)$ given $X^t=x^t$, and 
	$\hat T_t(\cdot|y^t)$ 
	the Bayes estimator of $T(\cdot|y_{t-m+1}^t)$ given $Y^t=y^t$. 
	For each $t=1,\ldots,n-1$ and for each trajectory $x^n=(1,\ldots,1,x_{t+1},\ldots,x_n) \in \calX_t$, recalling the form \eqref{eq:MC_bayes} of the Bayes estimator, 
	we have, for each $j\in{\tilde \calS}$, 
	\begin{align*}
		&\hat M_n(j|x^n)\\ 
		&= ~ \frac{\prob{X^{n+1}=(x^n,j)}}{\prob{X^n=x^n}} \\
		&= ~ \frac{\Expect[\frac{1}{2}\cdot \left(1-\frac{1}{n}\right)^{t-m} \cdot \frac b{n2^{m-2}(k-1)^{m}} \prod_{s=t+m}^{n-1} M(x_{s+1}|x_{s-m+1}^s) M(j|x_{n-m+1}^n)]}{\Expect[\frac{1}{2}\cdot \left(1-\frac{1}{n}\right)^{t-m} \cdot \frac b{n2^{m-2}(k-1)^{m}} \prod_{s=t+m}^{n-1} M(x_{s+1}|x_{s-m+1}^s)]} \\
		&= ~ \pth{1-\frac{1}{n}} \frac{\Expect[\frac 1{(k-1)^m}\prod_{s=t+m}^{n-1} T(x_{s+1}|x_{s-m+1}^s) T(j|x_{n-m+1}^n)]}{\Expect[\frac 1{(k-1)^m}\prod_{s=t+m}^{n-1} T(x_{s+1}|x_{s-m+1}^s)]} \\
		&= ~ \pth{1-\frac{1}{n}} \frac{\prob{Y^{n-t+1} = (x_{t+1}^n, j)} }{\prob{Y^{n-t} = x_{t+1}^n}} \\
		&= ~ \pth{1-\frac{1}{n}} \hat T_{n-t}(j|x_{t+1}^n) . 
	\end{align*}
	Furthermore, since $M(1|x^m)=1/n$ for all $x^m\in {\tilde \calS}$ in the construction \eqref{eq:MC_M_construction2}, the Bayes estimator also satisfies $\hat M_n(1|x^n) = 1/n$ for $x^n\in \calX_t$ and $t\le n-m$.
	In all, we have
	\begin{equation}
		\hat M_n(\cdot|x^n) = \frac{1}{n} \delta_1 + \pth{1-\frac{1}{n}}\hat T_{n-t}(\cdot|x_{t+1}^n) , \quad x^n\in\calX_t, t\le n-m. 
		\label{eq:MC_bayesMT2}
	\end{equation}
	with $\delta_1$ denoting the point mass at state 1, which parallels the fact that 
	\begin{equation}
		M(\cdot|y^m) = \frac{1}{n} \delta_1 + \pth{1-\frac{1}{n}} T(\cdot|y^m), \quad y^m\in {\tilde \calS}^m.
		\label{eq:MC_MT2}
	\end{equation}

	By \ref{pt:MC_22}, each event $\{X^n\in\calX_t\}$ occurs with probability at least ${b\over n2^{m-1}}\pth{1-\frac 1n}^{n-2m}$, and is independent of $T$. Therefore,
	\begin{align}\label{eq:MC_decomposition2}
		&\EE_{T}\left[\EE[D(M(\cdot|X_{n-1}X_n)\| \widehat{M}(\cdot|X^n)) ]\right] \nonumber \\
		&\ge \frac b{n2^{m-1}}\pth{1-\frac 1n}^{n-2m}\sum_{t=m}^{n-m}\EE_{T}\left[\EE[D(M(\cdot|X_{n-m+1}^n)\| \widehat{M}(\cdot|X^n)) | X^n\in\calX_t ]\right].  
	\end{align}
	By \ref{pt:MC_23},  the conditional joint law of $(T,X_{t+1},\ldots,X_n)$ on the event $\{X^n\in\calX_t\}$ is the same as the joint law of $(T,Y_{1},\ldots,Y_{n-t})$. 
	Thus, we may express the Bayes prediction risk in the $X$ chain as 
	\begin{align}\label{eq:MC_reduction_M_N2}
		\EE_{T}\left[\EE[D(M(\cdot|X_{n-m+1}^n)\| \widehat{M}(\cdot|X^n)) | X^n\in\calX_t ]\right] 
		&	\stepa{=} \left(1-\frac{1}{n}\right)\cdot \EE_{T}\left[\EE[D(T(\cdot|Y_{n-t-m+1}^{n-t}) \| \widehat{T}(\cdot|Y^{n-t}))]\right] \nonumber \\
		& \stepb{=} \left(1-\frac{1}{n}\right)\cdot I(T;Y_{n-t+1}|Y^{n-t}),
	\end{align}
	where (a) follows from \prettyref{eq:MC_bayesMT2}, \prettyref{eq:MC_MT2}, and the fact that for distributions $P,Q$ supported on ${\tilde \calS}$,
	$D(\epsilon \delta_1 + (1-\epsilon) P \|\epsilon \delta_1 + (1-\epsilon) Q)=(1-\epsilon) D(P\|Q)$;
	(b) is the mutual information representation \prettyref{eq:MC_Bayes-MI} of the Bayes prediction risk.
	Finally, the lemma follows from \eqref{eq:MC_decomposition2}, \eqref{eq:MC_reduction_M_N2}, and the chain rule
	\begin{align*}
		\sum_{t=m}^{n-m} I(T;Y_{n-t+1}|Y^{n-t}) = I(T;Y^{n-m}) - I(T;Y^m) \ge  I(T;Y^{n-m}) - m\log(k-1),
	\end{align*} 
	as $I(T;Y^m)\le H(Y^m)\le m \log(k-1)$. 
\end{proof}

\paragraph{Prior construction and lower bounding the mutual information}

We assume that $k=2k_0+1$ for some integer $k_0$. For simplicity of notation we replace $\tilde S$ by $\calY={1,\dots,k-1}$. This does not affect the lower bound. Define an equivalent relation on $|\calY|^{m-1}$ given by the following rule: $x^{m-1}$ and $y^{m-1}$ are related if and only if $x^{m-1}=y^{m-1}$ or $x^{m-1}=\overline{y^{m-1}}$. Let $R_{m-1}$ be a subset of $\calY^{m-1}$ that consists of exactly one representative from each of the equivalent classes. As each of the equivalent classes under this relation will have at most two elements the total number of equivalent classes is at least $|\calY|^{m-1}\over 2$, i.e., $|R_{m-1}|\geq {(k-1)^{m-1}\over 2}$.
We consider the following prior: let $u=\sth{u_{ix^{m-1}j}}_{i\leq j\in{[k_0]},x^{m-1}\in R_{m-1}}$ be iid and uniformly distributed in $[1/(4k_0),3/(4k_0)]$ and for each $i\leq j,x^{m-1}\in R_{m-1}$ define  $u_{jx^{m-1}i},u_{i\overline{x^{m-1}}j},u_{j\overline{x^{m-1}}i}$ to be same as $u_{i{x^{m-1}}j}$. Let the transition matrix $T$ be given by
\begin{align}\label{eq:MC_u-T-relation-order-m}
	&T(2j-1|2i-1,x^{m-1}) = T(2j|2i,x^{m-1}) = u_{ix^{m-1}j},
	\nonumber\\
	&T(2j|2i-1,x^{m-1}) = T(2j-1|2i,x^{m-1}) = \frac 1{k_0}-u_{ix^{m-1}j}, \quad  i,j\in \calY,x^{m-1}\in \calY^{m-1}.
\end{align}
One can check that the constructed $T$ is a stochastic matrix and satisfies the property \ref{pt:MC_P}, which enforces uniform stationary distribution. Also each entry of $T$ belongs to the interval $[{1\over 2(k-1)},{3\over 2(k-1)}]$. 

Next we use the following lemma to derive estimation guarantees on $T$.

\begin{lemma}\label{lmm:L2_upper-order-m}
	Suppose that $T$ is an $\ell^{m}\times \ell$ transition matrix, on state space $\calY^m$ with $|\calY|=\ell$, satisfying $T(x_{m+1}|x^{m}) = T(x_1|\overline{x_2^{m+1}}),\quad \forall x^{m+1}\in [\ell]^{m+1}$ and $T(y_{m+1}|y^m)\in [\frac{c_1}{\ell},\frac{c_2}{\ell}]$ with $0<c_1<c_2<1<c_1$ for all $y^{m+1}\in [\ell]^{m+1}$. Then there is an estimator $\widehat{T}$ based on stationary trajectory $Y^n$ simulated from $T$ such that
	\begin{align*}
		\EE[\|\widehat{T} - T \|_{\mathsf{F}}^2] \le\frac{4c_1^{2m+3}(m+1)\ell^{2m}}{c_2(n-m)} ,
	\end{align*}
	where $\|\widehat{T} - T \|_{\mathsf{F}} = \sqrt{\sum_{y^{m+1}} (\widehat{T}(y_{m+1}|y^m) - T(y_{m+1}|y^m))^2}$ denotes the Frobenius norm.
\end{lemma}

For our purpose we will use the above lemma on $T$ with $\ell=k-1,c_1=\frac 12,c_2=\frac 32$. Therefore it follows that there exist estimators $\widehat{T}(Y^n)$ and $\widehat{u}(Y^n)$ such that
\begin{align}
	\EE[\|\widehat{u}(Y^n) - u\|_2^2] \le \EE[\|\widehat{T}(Y^n) - T \|_{\mathsf{F}}^2] \le \frac{4c_2(m+1)(k-1)^{2m}}{c_1^{2m+3}(n-m)}. 
	\label{eq:MSEu-second}
\end{align}
Here and below, we identify $u=\sth{u_{ix^{m-1}j}}_{i\leq j,x^{m-1}\in R_{m-1}}$ and $\hat u=\sth{\hat u_{ix^{m-1}j}}_{i\leq j,x^{m-1}\in R_{m-1}}$ as ${|R_{m-1}|k_0(k_0+1)\over 2}={|R_{m-1}|(k^2-1)\over 8}$-dimensional vectors.

Let $h(X) = \int -f_X(x)\log f_X(x)dx$ denote the differential entropy of a continuous random vector $X$ with density $f_X$ w.r.t the Lebesgue measure 
and $h(X|Y)=\int -f_{XY}(xy)\log f_{X|Y}(x|y)dxdy$ the conditional differential entropy (cf.~e.g.~\cite{cover}). Then
\begin{align}
	h(u) = \sum_{i\leq j\in{[k_0]},x^{m-1}\in R_{m-1}}h(u_{ix^{m-1}j}) = -\frac{|R_{m-1}|(k^2-1)}{8}\log(k-1). 
\end{align}
Then
\begin{align*}
	I(T;Y^n)
	& \stepa{=} I(u;Y^n)\\
	& \stepb{\geq} I(u;\hat u(Y^n)) = h(u)-h(u|\hat u(Y^n))\\
	& \stepc{\geq} h(u)-h(u-\hat u(Y^n))\\
	& \stepd{\geq} \frac{|R_{m-1}|(k^2-1)}{16}\log\left(\frac{c_1^{2m+3}|R_{m-1}|(k^2-1)(n-m)}{64\pi ec_2(m+1)(k-1)^{2m+2}}\right) 
	\ge \frac{(k-1)^{m+1}}{32}\log\left(\frac{n-m}{c_m(k-1)^{m+1}}\right).
\end{align*}
for constant $c_m={128\pi ec_2(m+1)\over c_1^{2m+3}}$, where 
(a) is because $u$ and $T$ are in one-to-one correspondence by \prettyref{eq:MC_u-T-relation-order-m};
(b) follows from the data processing inequality;
(c) is because $h(\cdot)$ is translation invariant and concave;
(d) follows from the maximum entropy principle \cite{cover}: 
$h(u-\hat u(Y^n)) \leq \frac{|R_{m-1}|(k^2-1)}{16}\log\left(\frac{2\pi e}{|R_{m-1}|(k^2-1)/8}\cdot \EE[\|\widehat{u}(Y^n) - u\|_2^2] \right)$, which in turn is bounded by \prettyref{eq:MSEu-second}.
Plugging this lower bound into Lemma \ref{lmm:riskred_markov2} completes the lower bound proof of \prettyref{thm:optimal-m}. 

\subsubsection{Proof of \prettyref{lmm:L2_upper-order-m} via pseudo spectral gap}
	In view of \prettyref{lmm:MC_stationary-T} we get that the stationary distribution of $T$ is uniform over $\calY^m$, and there is a one-to-one correspondence between the joint distribution of $Y^{m+1}$ and the transition probabilities
	\begin{align}
		\PP\qth{Y^{m+1}=y^{m+1}}
		=\frac 1{\ell^{m}}T(y_{m+1}|y^m).
	\end{align}
 	Consider the following estimator $\widehat{T}$: for $y_{m+1}\in [\ell]^{m+1}$, let 
	\begin{align*}
		\widehat{T}(y_{m+1}|y^m) = \ell^m\cdot \frac{\sum_{t=1}^{n-m} \indc{Y_t^{t+m}=y^{m+1}} }{n-m}. 
	\end{align*}
	Clearly $\EE[\widehat{T}(y_{m+1}|y^m)]=\ell^m\PP\qth{y_{m+1}|y^m}= T(y_{m+1}|y^m)$. Next we observe that the sequence of random variables $\sth{Y_{t}^{t+m}}_{t=1}^{n-m}$ is a first-order Markov chain on $[\ell]^{m+1}$. Let us denote its transition matrix by $T_{m+1}$ and note that its stationary distribution is given by $\pi(a^{m+1})=\ell^{-m}T(a_{m+1}|a^m), a^{m+1}\in[\ell]^{m+1}$. For the transition matrix $T_{m+1}$, which must be non-reversible, the \emph{pseudo spectral gap} $\gamma_{\text{ps}}(T_{m+1})$ is defined as
	\begin{align*}
		\gamma_{\text{ps}}(T_{m+1}) = \max_{r\ge 1} \frac{\gamma( (T_{m+1}^*)^r T_{m+1}^r )}{r}, 
	\end{align*}
	where $T_{m+1}^*$ is the adjoint of $T_{m+1}$ defined as $T_{m+1}^*(b^{m+1}|a^{m+1}) = \pi(b^{m+1})T(a^{m+1}|b^{m+1})/\pi(a^{m+1})$. 
	With these notations, the concentration inequality of \cite[Theorem 3.2]{P15} gives the following variance bound:
	\begin{align*}
		\var(\widehat{T}(y_{m+1}|y^m)) \le \ell^{2m}\cdot \frac{4\PP\qth{Y^{m+1}=y^{m+1}}}{\gamma_{\text{ps}}(T_{m+1})(n-m)} \le \ell^{2m}\cdot \frac{4T(y_{m+1}|y^m)\ell^{-m}}{\gamma_{\text{ps}}(T_{m+1})(n-m)}.
	\end{align*}
	
	The following lemma bounds the pseudo spectral gap from below.
	\begin{lemma}\label{lmm:MC_bound-pseudo-spectral-gap}
		Let $T\in \mathbb{R}^{\ell^m\times \ell}$ be the transition matrix of an $m$-th order Markov chain $(Y_t)_{t\geq 1}$ 
		over a discrete state space $\mathcal{Y}$ with $|\mathcal{Y}|=\ell$, and assume that 
		\begin{itemize}
			\item all the entries of $T$ lie in the interval $[\frac {c_1}\ell,\frac {c_2}\ell]$ for some absolute constants $0<c_1<c_2$;
			\item $T$ has the uniform stationary distribution on $[\ell]^m$.
		\end{itemize}
		Let $T_{m+1}\in \reals^{\ell^{m+1} \times \ell^{m+1}}$ be the transition matrix of the first-order Markov chain $((Y_t,Y_{t+1},\cdots,Y_{t+m}))_{t\geq 1}$. Then we have
		\begin{align*}
			\gamma_{\text{\rm ps}}(T_{m+1}) \ge {c_1^{2m+3}\over c_2(m+1)}. 
		\end{align*}
	\end{lemma}
	Consequently, we have
	\begin{align*}
		\EE[\|\widehat{T} - T \|_{\mathsf{ F}}^2] = \sum_{y^{m+1}\in [\ell]^{m+1}} 	\var(\widehat{T}(y_{m+1}|y^m)) \le  \sum_{y^{m+1}\in [\ell]^{m+1}} \frac {4c_2(m+1)\ell^m}{c_1^{2m+3}}\cdot \frac{T(y_{m+1}|y^m)}{n-m} = \frac{4c_2(m+1)\ell^{2m}}{c_1^{2m+3}(n-m)}, 
	\end{align*}
	completing the proof.

\begin{proof}[Proof of \prettyref{lmm:MC_bound-pseudo-spectral-gap}]
	As $T_{m+1}$ is a first-order Markov chain, the stochastic matrix $T_{m+1}^{m+1}$ defines the probabilities of transition from $(Y_t,Y_{t+1},\cdots,Y_{t+m})$ to $(Y_{t+m+1},Y_{t+m+2},\cdots,Y_{t+2m+1})$. By our assumption on $T$
	\begin{align}\label{eq:MC_min-T_{m+1}}
		\min_{a^{2m+2}\in \calY^{2m+2}} T^{m+1}_{m+1}(a^{2m+2}_{m+2}|a^{m+1})
		\geq \prod_{t=0}^{m}T(a_{2m+2-t}|a^{2m+1-t}_{m+2-t})
		\geq {c_1^{m+1} \over \ell^{m+1}}.
	\end{align}
	Given any $a^{m+1},b^{m+1}\in \calY^{m+1}$, using the above inequality we have
	\begin{align}\label{eq:MC_min-T*_{m+1}}
		&(T_{m+1}^*)^{m+1}(b^{m+1}|a^{m+1})
		\nonumber\\
		&=\sum_{\bm y_1\in\calY^{m+1},\dots,\bm y_{m}\in\calY^{m+1}}T_{m+1}^*(b^{m+1}|\bm y_{m})\sth{\prod_{t=1}^{m-1} T_{m+1}^*(\bm y_{m-t+1}|\bm y_{m-t})}T_{m+1}^*(\bm y_1|a^{m+1})
		\nonumber\\
		&=\sum_{\bm y_1\in\calY^{m+1},\dots,\bm y_{m}\in\calY^{m+1}}{\pi( b^{m+1})T_{m+1}(\bm y_{m}|b^{m+1})\over \pi(\bm y_m)}\sth{\prod_{t=1}^{m-1}{\pi(\bm y_{m-t+1})T_{m+1}(\bm y_{m-t}|\bm y_{m-t+1})\over \pi(\bm y_{m-t})}}{\pi(\bm y_1)T_{m+1}(a^{m+1}|\bm y_1)\over \pi(a^{m+1})}
		\nonumber\\
		&={\pi(b^{m+1})\over \pi(a^{m+1})}
		\sum_{\bm y_1\in\calY^{m+1},\dots,\bm y_{m}\in\calY^{m+1}}T_{m+1}(\bm y_{m}|b^{m+1})\sth{\prod_{t=1}^{m-1}T_{m+1}(\bm y_{m-t}|\bm y_{m-t+1})}T_{m+1}(a^{m+1}|\bm y_1)
		\nonumber\\
		&={\pi(b^{m+1})\over \pi(a^{m+1})}T^{m+1}_{m+1}(a^{m+1}|b^{m+1})
		\nonumber\\
		&={\pi(b^m)T(b_{m+1}|b^m)\over \pi(b^m)T(a_{m+1}|a^m)}T^{m+1}_{m+1}(a^{m+1}|b^{m+1})
		\geq {c_1\over c_2}\cdot {c_1^{m+1}\over \ell^{m+1}}.
	\end{align}
	Using \eqref{eq:MC_min-T_{m+1}},\eqref{eq:MC_min-T*_{m+1}} we get
	\begin{align}
		&\min_{a^{m+1},b^{m+1}\in\calY^{m+1}}\sth{(T_{m+1}^*)^{m+1}T^{m+1}_{m+1}}(b^{m+1}|a^{m+1})
		\nonumber\\
		&\geq 
		\sum_{d^{m+1}\in\calY^{m+1}}\pth{\min_{a^{m+1},d^{m+1}\in\calY^{m+1}}(T_{m+1}^*)^{m+1}(d^{m+1}|a^{m+1})}
		\pth{\min_{b^{m+1},d^{m+1}\in\calY^{m+1}}T^{m+1}_{m+1}(b^{m+1}|d^{m+1})}
		\nonumber\\
		&\geq \sum_{d^{m+1}\in\calY^{m+1}}{c_1^{2m+3}\over c_2\ell^{2m+2}}
		\geq {c_1^{2m+3}\over c_2\ell^{m+1}}.
	\end{align}
	As $(T_{m+1}^*)^{m+1}T^{m+1}_{m+1}$ is an $\ell^{m+1}\times \ell^{m+1}$ stochastic matrix, we can use \prettyref{lmm:special-hoffman} to get the lower bound on its spectral gap $\gamma((T_{m+1}^*)^{m+1}T^{m+1}_{m+1})\geq {c_1^{2m+3}\over c_2}$. Hence we get
	\begin{align}
		\gamma_{\text{\rm ps}}(T_{m+1})\geq {\gamma((T_{m+1}^*)^{m+1}T^{m+1}_{m+1})\over m+1}
		\geq {c_1^{2m+3}\over c_2(m+1)}
	\end{align}
	as required.
	A more generalized version of \prettyref{lmm:special-hoffman} can be found in from \cite{H67}.
	
	\begin{lemma}\label{lmm:special-hoffman}
		Suppose that $A$ is a $d\times d$ stochastic matrix with $\min_{i,j}A_{ij}\geq \epsilon$. Then for any eigenvalue $\lambda$ of $A$ other than 1 we have $|\lambda|\leq 1-d\epsilon$.
	\end{lemma}
	
	\begin{proof}
		Suppose that $\lambda$ is an eigenvalue of $A$ other than 1 with non-zero left eigenvector $\bm v$, i.e. $\lambda v_j=\sum_{i=1}^dv_i A_{ij},j=1,\dots,d$. As $A$ is a stochastic matrix we know that $\sum_{j}A_{ij}=1$ for all $i$ and hence $\sum_{i=1}^d v_i=0$. This implies 
		\begin{align}
			|\lambda v_j|
			=\left|\sum_{i=1}^dv_iA_{ij}\right|
			=\left|\sum_{i=1}^dv_i(A_{ij}-\epsilon)\right|
			\leq \sum_{i=1}^d|v_i(A_{ij}-\epsilon)|
			= \sum_{i=1}^d|v_i|(A_{ij}-\epsilon)
		\end{align}
		with the last equality following from $A_{ij}\geq \epsilon$. Summing over $j=1,\dots d$ in the above equation and dividing by $\sum_{i=1}^d|v_i|$ we get $|\lambda|\leq 1-d\epsilon$ as required.
	\end{proof}
	
\end{proof}

\section{Discussions and open problems}
	\label{sec:discussion}
	
	We discuss the assumptions and implications of our results as well as related open problems.
	
	\paragraph{Very large state space.}
	\prettyref{thm:optimal} determines  the optimal prediction risk under the assumption of $k \lesssim \sqrt{n}$. 
	When $k \gtrsim \sqrt{n}$, \prettyref{thm:optimal} shows that the KL risk is bounded away from zero. However, as the KL risk can be as large as $\log k$, it is a meaningful question to determine the optimal rate in this case, which, thanks to the general reduction in \prettyref{eq:riskred-approx}, reduces to determining the redundancy for symmetric and general Markov chains. 
	For iid data, the minimax \emph{pointwise} redundancy is known to be $n \log \frac{k}{n} + O(\frac{n^2}{k})$ \cite[Theorem 1]{szpankowski2012minimax} when $k\gg n$.
	Since the average and pointwise redundancy usually behave similarly, for Markov chains it is reasonable to conjecture that the redundancy is $\Theta(n \log \frac{k^2}{n})$ in the large alphabet regime of $k \gtrsim \sqrt{n}$, which, in view of \prettyref{eq:riskred-approx}, would imply the optimal prediction risk is 
$\Theta(\log \frac{k^2}{n})$ for $k \gg \sqrt{n}$.
In comparison, we note that the prediction risk is at most $\log k$, achieved by the uniform distribution.


	\paragraph{Other loss functions}
	As mentioned in \prettyref{sec:technique}, standard arguments based on concentration inequalities inevitably rely on mixing conditions such as the spectral gap. In contrast, the risk bound in \prettyref{thm:optimal}, which is free of any mixing condition, is enabled by powerful techniques from universal compression which bound the redundancy by the pointwise maximum over all trajectories combined with information-theoretic or combinatorial argument. This program only relies on the Markovity of the process rather than stationarity or spectral gap assumptions. 
	The limitation of this approach, however, is that the reduction between prediction and redundancy crucially depends on the form of the KL loss function\footnote{In fact, this connection breaks down if one swap $M$ and $\hat M$ in the KL divergence in \prettyref{eq:riskkn}.} in \prettyref{eq:riskkn}, which allows one to use the mutual information representation and the chain rule to relate individual risks to the cumulative risk. 
More general loss in terms of $f$-divergence have been considered in \cite{HOP18}. Obtaining spectral gap-independent risk bound for these loss functions, this time without the aid of universal compression, is an open question.


	\paragraph{Stationarity}
	As mentioned above, the redundancy result in \prettyref{lmm:red-addone} (see also \cite{D83,tatwawadi2018minimax}) holds for nonstationary Markov chains as well.
	However, our redundancy-based risk upper bound in \prettyref{lmm:riskred} crucially relies on stationarity.
	It is unclear whether the result of \prettyref{thm:optimal} carries over to nonstationary chains.

\appendix

\section{Mutual information representation of prediction risk}
\label{app:caprisk}

The following lemma justifies the representation \prettyref{eq:caprisk} for the prediction risk as maximal conditional mutual information.
Unlike \prettyref{eq:capred} for redundancy which holds essentially without any condition \cite{kemperman1974shannon}, here we impose certain compactness assumptions which hold finite alphabets such as finite-state Markov chains studied in this paper.

\begin{lemma}
\label{lmm:caprisk}	
Let $\calX$ be finite and let $\Theta$ be a compact subset of $\reals^d$. 
Given $\{P_{X^{n+1}|\theta}: \theta \in \Theta\}$, define the prediction risk 
\begin{equation}
	\Risk_{n} \triangleq \inf_{Q_{X_{n+1}|X^{n}}} \sup_{\theta\in\Theta} D(P_{X_{n+1}|X^{n},\theta} \| Q_{X_{n+1}|X^{n}} | P_{X^{n}|\theta}),
	\label{eq:risk-n}
	\end{equation}	
Then
	\begin{equation}
	\Risk_n = \sup_{P_\theta \in\calM(\Theta)} I(\theta;X_{n+1}|X^n).
	\label{eq:caprisk-lemma}
	\end{equation}		
	where $\calM(\Theta)$ denotes the collection of all (Borel) probability measures on $\Theta$.
\end{lemma}

Note that for stationary Markov chains, \prettyref{eq:caprisk} follows from \prettyref{lmm:caprisk}	since one can take $\theta$ to be the joint distribution of $(X_1,\ldots,X_{n+1})$ itself which forms a compact subset of the probability simplex on $\calX^{n+1}$.

\begin{proof}
	It is clear that \prettyref{eq:risk-n} is equivalent to
	\[
	\Risk_{n} = \inf_{Q_{X_{n+1}|X^{n}}} \sup_{P_\theta\in\calM(\Theta)} D(P_{X_{n+1}|X^{n},\theta} \| Q_{X_{n+1}|X^{n}} | P_{X^{n},\theta}).
	\]
	By the variational representation \prettyref{eq:MI-var} of conditional mutual information, we have
	\begin{equation}
	I(\theta;X_{n+1}|X^n)=\inf_{Q_{X_{n+1}|X^{n}}} D(P_{X_{n+1}|X^{n},\theta} \| Q_{X_{n+1}|X^{n}} | P_{X^{n},\theta}).
	\label{eq:MI-var1}
	\end{equation}
	Thus \prettyref{eq:caprisk-lemma} amounts to justifying the interchange of infimum and supremum in \prettyref{eq:risk-n}. 
	It suffices to prove the upper bound.	
	
	Let $|\calX|=K$. For $\epsilon\in(0,1)$, define an auxiliary quantity:
		\begin{equation}
	\Risk_{n,\epsilon} \triangleq \inf_{Q_{X_{n+1}|X^{n}} \geq \frac{\epsilon}{K}}
 \sup_{P_\theta \in\calM(\Theta)} D(P_{X_{n+1}|X^{n},\theta} \| Q_{X_{n+1}|X^{n}} | P_{X^{n},\theta}),
	\label{eq:risk-npes}
	\end{equation}	
	where the constraint in the infimum is pointwise, namely, $Q_{X_{n+1}=x_{n+1}|X^{n}=x^n} \geq \frac{\epsilon}{K}$ for all $x_1,\ldots,x_{n+1} \in \calX$.
		By definition, we have $\Risk_n \leq \Risk_{n,\epsilon}$.	Furthermore, $\Risk_{n,\epsilon} $ can be 
	equivalently written as
		\begin{equation}
	\Risk_{n,\epsilon} = \inf_{Q_{X_{n+1}|X^{n}}} 
	\sup_{P_\theta\in\calM(\Theta)} D(P_{X_{n+1}|X^{n},\theta} \| (1-\epsilon) Q_{X_{n+1}|X^{n}} + \epsilon U | P_{X^{n},\theta}),
	\label{eq:risk-npes2}
	\end{equation}
	where $U$ denotes the uniform distribution on $\calX$. 
	
	We first show that the infimum and supremum in \prettyref{eq:risk-npes2} can be interchanged. This follows from the standard minimax theorem. Indeed, note that 	$D(P_{X_{n+1}|X^{n},\theta} \| (1-\epsilon) Q_{X_{n+1}|X^{n}} + \epsilon U | P_{X^{n},\theta})$ is convex in $Q_{X_{n+1}|X^{n}} $, affine in $P_\theta$, continuous in each argument, and takes values in $[0,\log \frac{K}{\epsilon}]$.	
	Since $\calM(\Theta)$ is convex and weakly compact (by Prokhorov's theorem) and the collection of conditional distributions $Q_{X_{n+1}|X^{n}}$  is convex, the minimax theorem (see, e.g., \cite[Theorem 2]{Fan53}) yields
			\begin{equation}
	\Risk_{n,\epsilon} =  
	\sup_{\pi\in\calM(\Theta)}\inf_{Q_{X_{n+1}|X^{n}}}D(P_{X_{n+1}|X^{n},\theta} \| (1-\epsilon) Q_{X_{n+1}|X^{n}} + \epsilon U | P_{X^{n},\theta}).
	\label{eq:risk-nepsilon-minimax}
	\end{equation}
	Finally, by the convexity of the KL divergence, for any $P$ on $\calX$, we have
	\[
	D(P\|(1-\epsilon) Q + \epsilon U) \leq (1-\epsilon) D(P\|Q)+\epsilon D(P\|U) \leq (1-\epsilon) D(P\|Q)+\epsilon \log K,
	\label{eq:KL-mix}
	\]
	which, in view of \prettyref{eq:MI-var1} and \prettyref{eq:risk-nepsilon-minimax}, implies
	\[
	\Risk_{n} \leq \Risk_{n,\epsilon} \leq \sup_{P_\theta \in\calM(\Theta)} I(\theta;X_{n+1}|X^n) + \epsilon \log K. 
	\]
	By the arbitrariness of $\epsilon$, \prettyref{eq:caprisk-lemma} follows.	
\end{proof}

	\section{Proof of \prettyref{lmm:moment.reversebound}}\label{app:moment.bounds}
	Recall that for any irreducible and reversible finite states transition matrix $M$ with stationary distribution $\pi$ the followings are satisfied:
	\begin{enumerate}
		\item $\pi_i>0$ for all $i$.
		\item $M(j|i)\pi_i=M(i|j)\pi_j$ for all $i,j$. 
	\end{enumerate}

	The following is a direct consequence of the Markov property.
	\begin{lemma}\label{lmm:product.markov.expectation}
		For any $1\leq t_1<\dots<t_m<\dots<t_k$ and any $Z_2=f\pth{X_{t_k},\dots,X_{t_{m}}},Z_1=g\pth{X_{t_{m-1}},\dots,X_{t_1}}$ we have
		\begin{align}
		&\EE\qth{Z_2\indc{X_{t_m}=j}Z_1|X_1=i}
		=\EE\qth{Z_2|X_{t_m}=j}
		\EE\qth{\indc{X_{t_m}=j}Z_1|X_1=i}
		\end{align}
	\end{lemma} 

	For $t\geq 0$, denote the $t$-step transition probability by $\prob{X_{t+1}=j|X_1=i}=M^t(j|i)$, which is the $ij$th entry of $M^t$.
	The following result is standard (see, e.g., \cite[Chap.~12]{LevinPeres17}). We include the proof mainly for the purpose of introducing the spectral decomposition.	
	\begin{lemma}\label{lmm:prob.bounds}
		Define $\lambda_*\eqdef 1-\gamma_*=\max\sth{\abs{\lambda_i}:i\neq 1}$. For any $t\geq 0$, 
		$
		\abs{\mymat^t(j|i)-\pi_j}
		\leq \lambda_*^{t}{\sqrt{\pi_j\over \pi_i}}.
		$
	\end{lemma}
	\begin{proof}
	Throughout the proof all vectors are column vectors except for $\pi$.
		Let $D_{\pi}$ denote the diagonal matrix with entries $D_\pi(i,i)=\pi_i$. By reversibility, 
		$D_\pi^{\frac{1}{2}}MD_\pi^{-\frac{1}{2}}$, which shares the same spectrum with $M$, is a symmetric matrix and admits the spectral decomposition 
		$D_\pi^{\frac{1}{2}}MD_\pi^{-\frac{1}{2}} = \sum_{a=1}^k\lambda_a u_au_a^\top$ for some orthonormal basis $\{u_1,\ldots,u_k\}$; in particular, $\lambda_1=1$ and $u_{1i}=\sqrt{\pi_i}$. Then for each $t\geq 1$,
		\begin{align}\label{eq:spectral.decomposition}
			\mymat^t=\sum_{a=1}^k\lambda_a^tD_\pi^{-\frac 12}u_au_a^\top D_\pi^{\frac 12} = \ones \pi + \sum_{a=2}^k\lambda_a^tD_\pi^{-\frac 12}u_au_a^\top D_\pi^{\frac 12}.
			\quad 
		\end{align}
		where $\ones$ is the all-ones vector.
		As $u_a$'s satisfy $\sum_{a=1}^k u_au_a^\top=I$ we get $\sum_{a=2}^k u_{ab}^2=1-u_{a1}^2\leq 1$ for any $b=1,\dots,k$. Using this along with Cauchy-Schwarz inequality we get
		\begin{align*}
		\abs{\mymat^t(j|i)-\pi_j}
		\leq \sqrt{\pi_j\over \pi_i}\sum_{a=2}^k\abs{\lambda_a}^t |u_{ai}u_{aj}|
		\leq \lambda_*^t\sqrt{\pi_j\over \pi_i} 
		\pth{\sum_{a=2}^k u_{ai}^2}^{\frac 12}
			\pth{\sum_{a=2}^k u_{aj}^2}^{\frac 12}
		\leq \lambda_*^t\sqrt{\pi_j\over \pi_i}
		\end{align*} 
		as required.
	\end{proof}
	
	\begin{lemma}\label{lmm:h.function.bounds}
	Fix states $i,j$. For any integers $a\geq b\geq 1$, define $$h_{s}(a,b)=
		\abs{\EE\qth{\indc{X_{a+1}=i}
			\pth{\indc{X_a=j}-\mymat(j|i)}^s|X_b=i}},\quad s=1,2,3,4.$$
		Then
		\begin{enumerate}[label=(\roman*)]
			\item 
			${h_1(a,b)}
			\leq 2{\sqrt{{\mymat(j|i)}}}{\lambda^{a-b}_*}$
			\item $\abs{h_2(a,b)-\pi_i\mymat(j|i)(1-\mymat(j|i))}\leq 4{\sqrt{{\mymat(j|i)}}}{\lambda^{a-b}_*}.$
			\item 
			${h_3(a,b)},{h_4(a,b)}\leq \pi_i\mymat(j|i)(1-\mymat(j|i))+4{\sqrt{{\mymat(j|i)}}}{\lambda^{a-b}_*}.$
		\end{enumerate}
	\end{lemma}
	\begin{proof}
		We apply \prettyref{lmm:prob.bounds} and time reversibility:
		\begin{enumerate}[label=(\roman*)]
			\item 
			\begin{align*}
			h_1(a,b)
			&=\abs{\PP\qth{X_{a+1}=i,X_a=j|X_b=i}-\mymat(j|i)\PP\qth{X_{a+1}=i|X_b=i}}
			\nonumber\\
			&=\abs{\mymat(i|j){\mymat^{a-b}}(j|i)-\mymat(j|i)\mymat^{a-b+1}(i|i)}
			\nonumber\\
			&\leq \mymat(i|j)\abs{\mymat^{a-b}(j|i)-\pi_j}+\mymat(j|i)\abs{\mymat^{a-b+1}(i|i)-\pi_i}
			\nonumber\\
			&\leq\lambda_*^{a-b}\mymat(i|j)\sqrt{\pi_j\over \pi_i}+\mymat(j|i)\lambda^{a-b+1}_*
			\nonumber\\
			&=\lambda^{a-b}_*\sqrt{\mymat(j|i)\mymat(i|j)}
			+{\mymat(j|i)}\lambda_*^{a-b+1}
			\leq 2\sqrt{\mymat(j|i)}{\lambda^{a-b}_*}.
			\end{align*}

			\item 
			\begin{align*}
			&|h_2(a,b)-\pi_i\mymat(j|i)(1-\mymat(j|i))|
			\nonumber\\
			=&\Big|\EE\qth{\indc{X_{a+1}=i,X_a=j}|X_b=i} -\pi_i\mymat(j|i)+\pth{\mymat(j|i)}^2 (\EE\qth{\indc{X_{a+1}=i}|X_b=i}-\pi_i) \\
			&-2M(j|i) (\EE\qth{\indc{X_{a+1}=i,X_a=j}|X_b=i}-\pi_iM(j|i))\Big| \\
			\leq &\abs{\Prob\qth{X_{a+1}=i,X_a=j|X_b=i}-\pi_j\mymat(i|j)}+({\mymat(j|i)})^2\abs{\PP\qth{X_{a+1}=i|X_b=i}-\pi_i}
			\nonumber\\
			&\quad+2{\mymat(j|i)}\abs{\PP\qth{X_{a+1}=i,X_a=j|X_b=i}-\pi_j\mymat(i|j)}
			\nonumber\\
			=&\mymat(i|j)\abs{\mymat^{a-b}(j|i)-\pi_j}+({\mymat(j|i)})^2\abs{M^{a-b+1}(i|i)-\pi_i}
			+2\mymat(j|i)\mymat(i|j)\abs{\mymat^{a-b}(j|i)-\pi_j}
			\nonumber\\
			\leq & \mymat(i|j)\sqrt{\pi_j\over \pi_i}\lambda_*^{a-b}
			+({\mymat(j|i)})^2\lambda^{a-b+1}_*
			+2{\mymat(j|i)}\mymat(i|j)\sqrt{\pi_j\over \pi_i}\lambda^{a-b}_*
			\nonumber\\
			\leq &\lambda_*^{a-b}\pth{\sqrt{\mymat(i|j)}\sqrt{M(i|j)\pi_j\over \pi_i}
			+({\mymat(j|i)})^2
			+2{\mymat(j|i)}\sqrt{\mymat(i|j)}\sqrt{M(i|j)\pi_j\over \pi_i}}
			\nonumber\\
			\leq &
			4\sqrt{\mymat(j|i)}\lambda^{a-b}_*.
			\end{align*}
			\item 
			$h_3(a,b),h_4(a,b)\leq h_2(a,b)$.			\qedhere
		\end{enumerate}
	\end{proof}

	\begin{proof}[Proof of \prettyref{lmm:moment.reversebound}(i)]
		\label{sec:secondmoment.reversebound}

		For ease of notation we use $c_0$ to denote an absolute constant whose value may vary at each occurrence. 
		Fix $i,j \in[k]$. 
		Note that the empirical count defined in \prettyref{eq:transition.count} can be written as 
		$N_i = \sum_{a=1}^{n-1}\indc{X_{n-a}=i}$ and $N_{ij}=\sum_{a=1}^{n-1}\indc{X_{n-a}=i,X_{n-a+1}=j}$.
		Then
		\begin{align*}
		&\EE\qth{\pth{M(j|i)N_i-N_{ij}}^2|X_n=i}
		\nonumber\\
		=&\EE\qth{\left.\pth{\sum_{a=1}^{n-1}\indc{X_{n-a}=i}
				\pth{\indc{X_{n-a+1}=j}-{\mymat(j|i)}}}^2\right|X_n=i}
		\nonumber\\
		\stepa{=}&\EE\qth{\left.\pth{\sum_{a=1}^{n-1}\indc{X_{a+1}=i}
				\pth{\indc{X_a=j}-{\mymat(j|i)}}}^2\right|X_1=i}\\
		\stepb{=}& \abs{\sum_{a,b}\EE\qth{\eta_a\eta_b|X_1=i}}
		\le 2\sum_{a\geq b}\abs{\EE\qth{\eta_a\eta_b|X_1=i}},
		\end{align*}
		where 
		(a) is due to time reversibility;
		in (b) we defined $\eta_a \triangleq \indc{X_{a+1}=i}\pth{\indc{X_a=j}-{\mymat(j|i)}}$. 
		We divide the summands into different cases and apply \prettyref{lmm:h.function.bounds}.

		\paragraph{Case I: Two distinct indices.}
		For any $a>b$, using \prettyref{lmm:product.markov.expectation} we get
		\begin{align}
			\abs{\EE\qth{\eta_a\eta_b|X_1=i}}
			=\abs{\EE\qth{\eta_a|X_{b+1}=i}}\abs{\EE\qth{\eta_b|X_1=1}}
			=h_1(a,b+1)h_1(b,1)
		\end{align}
		which implies
		\begin{align*}
		&\mathop{\sum\sum}_{n-1\geq a>b\geq 1} \abs{\EE\qth{\eta_a\eta_b|X_1=i}}
		=\mathop{\sum\sum}_{n-1\geq a>b\geq 1} h_1(a,b+1)h_1(b,1)
		\lesssim {\mymat(j|i)}
		\mathop{\sum\sum}_{n-1\geq a>b\geq 1}\lambda^{a-2}_*
		\lesssim {M(j|i)\over {\gamma^2_*}}.
		\end{align*}
		Here the last inequality (and similar sums in later deductions) can be explained as follows. Note that for $\gamma_*\geq\frac 12$ (i.e. $\lambda_*\leq \frac 12$), the sum is clearly bounded by an absolute constant; for $\gamma_*<\frac 12$ (i.e. $\lambda_*> \frac 12$), we compare the sum with the mean (or higher moments in other calculations) of a geometric random variable.

		\paragraph{Case II: Single index.}
		
		\begin{align}
		&\sum_{a=1}^{n-1}\EE\qth{\eta_a^2|X_1=i}
		=\sum_{a=1}^{n-1}h_2(a,1)
		\lesssim n\pi_i{\mymat(j|i)}(1-{\mymat(j|i)})
		+{\sqrt{{\mymat(j|i)}}\over \gamma_*}.
		\end{align}

		Combining the above we get 
		\begin{align*}
		\EE\qth{\pth{N_{ij}-{\mymat(j|i)}N_i}^2|X_n=i}
		\lesssim n\pi_i{\mymat(j|i)}(1-{\mymat(j|i)})
		+{\sqrt{{\mymat(j|i)}}\over \gamma_*}+{M(j|i)\over \gamma_*^2}
		\end{align*}
		as required.
	\end{proof}

	\begin{proof}[Proof of \prettyref{lmm:moment.reversebound}(ii)] 
		
		We first note that due to reversibility we can write (similar as in proof of \prettyref{lmm:moment.reversebound}(i)) with $\eta_a=\indc{X_{a+1}=i}\pth{\indc{X_a=j}-{\mymat(j|i)}}$
		\begin{align}
		&\EE\qth{\pth{M(j|i)N_i-N_{ij}}^4|X_n=i}
		\nonumber\\
		&=\EE\qth{\left.\pth{\sum_{a=1}^{n-1}\indc{X_{a+1}=i}
				\pth{\indc{X_a=j}-{\mymat(j|i)}}}^4\right|X_1=i}
		\nonumber\\
		&=\abs{\sum_{a,b,d,e}\EE\qth{\eta_a\eta_b\eta_d\eta_e|X_1=i}}
		\leq
		 \sum_{a,b,d,e}\abs{\EE\qth{\eta_a\eta_b\eta_d\eta_e|X_1=i}}
		\lesssim \sum_{a\geq b\geq d\geq e}\abs{\EE\qth{\eta_a\eta_b\eta_d\eta_e|X_1=i}}.
		\end{align}
		We bound the sum over different combinations of $a\geq b\geq d\geq e$ to come up with a bound on the required fourth moment. We first divide the $\eta$'s into groups depending on how many distinct indices of $\eta$ there are. We use the following identities which follow from \prettyref{lmm:product.markov.expectation}: for indices $a>b>d>e$
		\begin{itemize}
			\item $\abs{\EE\qth{\eta_a\eta_b\eta_d\eta_e|X_1=i}}
			=h_1(a,b+1)h_1(b,d+1)h_1(d,e+1)h_1(e,1)$
			\item For $s_1,s_2,s_3\in\sth{1,2}$, $\abs{\EE\qth{\eta_a^{s_1}\eta_b^{s_2}\eta_d^{s_3}|X_1=i}}
			=h_{s_1}(a,b+1)h_{s_2}(b,d+1)h_{s_3}(d,1)$
			\item For $s_1,s_2\in\sth{1,2,3}$, $\abs{\EE\qth{\eta_a^{s_1}\eta_b^{s_2}|X_1=i}}
			=h_{s_1}(a,b+1)h_{s_2}(b,1)$
			\item $\EE\qth{\eta_a^4|X_1=1}=h_4(a,1)$
		\end{itemize}
		and then use \prettyref{lmm:h.function.bounds} to bound the $h$ functions.
		\paragraph{Case I: Four distinct indices.}
		Using \prettyref{lmm:h.function.bounds} we have
		\begin{align*} 
			\mathop{\sum\sum\sum\sum}_{n-1\geq a>b>d>e\geq 1} \abs{\EE\qth{\eta_a\eta_b\eta_d\eta_e|X_1=i}}
			&=\mathop{\sum\sum\sum\sum}_{n-1\geq a>b>d>e\geq 1}
			h_1(a,b+1)h_1(b,d+1)h_1(d,e+1)h_1(e,1)
			\nonumber\\
			&{\leq} M(j|i)^2
			\mathop{\sum\sum\sum\sum}_{n-1\geq a>b>d>e\geq 1}\lambda^{a-4}_*
			\lesssim{M(j|i)^2\over \gamma_*^4}.
			\end{align*}

			\paragraph{Case II: Three distinct  indices.} 
			There are three cases, namely $\eta_a^2\eta_b\eta_d,\eta_a\eta_b^2\eta_d$ and $\eta_a\eta_b\eta_d^2$.
			\begin{enumerate}
			\item Bounding $\mathop{\sum\sum\sum}_{n-1\geq a>b>d\geq 1}
			\abs{\EE\qth{\eta_a^2\eta_b\eta_d|X_1=i}}$:
				\begin{align*}
			\mathop{\sum\sum\sum}_{n-1\geq a>b>d\geq 1}
			\abs{\EE\qth{\eta_a^2\eta_b\eta_d|X_1=i}}
			&=\mathop{\sum\sum\sum}_{n-1\geq a>b>d\geq 1}
			h_2(a,b+1)h_1(b,d+1)h_1(d,1)
			\nonumber\\
			&\lesssim \mathop{\sum\sum\sum}_{n-1\geq a>b>d\geq 1}
			\pth{\pi_i{\mymat(j|i)}(1-{\mymat(j|i)})
				+{\sqrt{{\mymat(j|i)}}}{\lambda^{a-b-1}_*}}
			{\mymat(j|i)}{\lambda^{b-2}_*}
			\nonumber\\
			&\lesssim {{ {{\mymat(j|i)}\over \gamma_*^2}n\pi_i{\mymat(j|i)}(1-{\mymat(j|i)})}
				+{{\mymat(j|i)}^{\frac 32}\over \gamma_*^3}}
			\nonumber\\
			&\lesssim {\pth{n\pi_i{\mymat(j|i)}(1-{\mymat(j|i)})}^2
				+{{\mymat(j|i)}^{\frac 32}\over \gamma_*^3}+{{\mymat(j|i)}^2\over \gamma_*^4}}
			\end{align*}
			where the last inequality followed by using $xy\leq x^2+y^2$.

		\item Bounding $\mathop{\sum\sum\sum}_{n-2\geq a>b>d\geq 1}
		\abs{\EE\qth{\eta_a\eta_b^2\eta_d|X_1=i}}$: 
			
			\begin{align*}
			&\mathop{\sum\sum\sum}_{n-2\geq a>b>d\geq 1}
			\abs{\EE\qth{\eta_a\eta_b^2\eta_d|X_1=i}}
			\nonumber\\
			&=\mathop{\sum\sum\sum}_{n-2\geq a>b>d\geq 1}
			h_1(a,b+1)h_2(b,d+1)h_1(d,1)
			\nonumber\\
			&\lesssim \mathop{\sum\sum\sum}_{n-2\geq a>b>d\geq 1}
			\pth{\pi_i{\mymat(j|i)}(1-{\mymat(j|i)})
				+{\sqrt{{\mymat(j|i)}}}{\lambda^{b-d-1}_*}}
			M(j|i){\lambda^{a-b+d-2}_*}
			\nonumber\\
			&\lesssim {M(j|i)\over \gamma_*^2}n\pi_iM(j|i)(1-M(j|i))
				+{M(j|i)^{\frac 32}\over \gamma_*^3}
			\nonumber\\
			&\lesssim {n\pi_i{\mymat(j|i)}(1-{\mymat(j|i)})}^2
				+{M(j|i)^{\frac 32}\over \gamma_*^3}+{M(j|i)^2\over \gamma_*^4}.
			\end{align*}

			\item Bounding $\mathop{\sum\sum\sum}_{n-2\geq a>b>d\geq 1}
			\abs{\EE\qth{\eta_a\eta_b\eta_d^2|X_1=i}}$:
			\begin{align*}
			&\mathop{\sum\sum\sum}_{n-2\geq a>b>d\geq 1}
			\abs{\EE\qth{\eta_a\eta_b\eta_d^2|X_1=i}}
			\nonumber\\
			&= \mathop{\sum\sum\sum}_{n-2\geq a>b>d\geq 1}
			h_1(a,b+1)h_1(b,d+1)h_2(d,1)
			\nonumber\\
			&\lesssim \mathop{\sum\sum\sum}_{n-2\geq a>b>d\geq 1}
			\pth{\pi_i{\mymat(j|i)}(1-{\mymat(j|i)})+\sqrt{M(j|i)}{\lambda^{d-1}_*}}
			M(j|i)\lambda^{a-d-2}_*
			\nonumber\\
			&\lesssim{{{{\mymat(j|i)}\over \gamma_*^2}n\pi_i{\mymat(j|i)}(1-{\mymat(j|i)})}
				+{{\mymat(j|i)}^{\frac 32}\over \gamma_*^3}}\\
			&\lesssim
			\pth{n\pi_i{\mymat(j|i)}(1-{\mymat(j|i)})}^2
				+{{\mymat(j|i)}^{\frac 32}\over \gamma_*^3}
				+{{\mymat(j|i)}^2\over \gamma_*^4}.
			\end{align*}
			\end{enumerate}

			\paragraph{Case III: Two distinct indices.}
			There are three different cases, namely $\eta_a^2\eta_b^2,\eta_a^3\eta_b$ and $\eta_a\eta_b^3$.
			
			\begin{enumerate}
				
			\item Bounding $\mathop{\sum\sum}_{n-2\geq a>b\geq 1}
			\abs{\EE\qth{\eta_a^2\eta_b^2|X_1=i}}$:
			 \begin{align*}
			&\mathop{\sum\sum}_{n-2\geq a>b\geq 1}
			\EE\qth{\eta_a^2\eta_b^2|X_1=i}
			\nonumber\\
			&=\mathop{\sum\sum}_{n-2\geq a>b\geq 1}
			h_2(a,b+1)h_2(b,1)
			\nonumber\\
			& \lesssim \mathop{\sum\sum}_{n-2\geq a>b\geq 1}
			{\pth{\pi_i{\mymat(j|i)}(1-{\mymat(j|i)})+{\sqrt{{\mymat(j|i)}}}\lambda^{a-b-1}_*}
			\pth{\pi_i{\mymat(j|i)}(1-{\mymat(j|i)})+{\sqrt{{\mymat(j|i)}}}
					\lambda^{b-1}_*}}
			\nonumber\\
			&\lesssim \mathop{\sum\sum}_{n-2\geq a>b\geq 1}
			\Big\{\pi_i{\mymat(j|i)}(1-{\mymat(j|i)})
				{\sqrt{{\mymat(j|i)}}}
				({\lambda^{a-b-1}_*}+{\lambda^{b-1}_*})\\
				&	\quad	+\pth{\pi_i{\mymat(j|i)}(1-{\mymat(j|i)})}^2
				+{{\mymat(j|i)}}\lambda_*^{a-2}
				\Big\}
			\nonumber\\
			&\lesssim {\pth{n\pi_i{\mymat(j|i)}(1-{\mymat(j|i)})}^2
				+{{{\sqrt{{\mymat(j|i)}}\over \gamma_*}}n\pi_i{\mymat(j|i)}(1-{\mymat(j|i)})}
				+{{{\mymat(j|i)}\over \gamma_*^2}}}
			\nonumber\\
			&\lesssim \pth{n\pi_i{\mymat(j|i)}(1-{\mymat(j|i)})}^2+{{{\mymat(j|i)}\over \gamma_*^2}}.
			\end{align*}

			\item Bounding $\mathop{\sum\sum}_{n-2\geq a>b\geq 1}
			\abs{\EE\qth{\eta_a^3\eta_b|X_1=i}}$:
			
			\begin{align*}
			&\mathop{\sum\sum}_{n-2\geq a>b\geq 1}
			\abs{\EE\qth{\eta_a^3\eta_b|X_1=i}}
			\nonumber\\
			&= \mathop{\sum\sum}_{n-2\geq a>b\geq 1}
			h_3(a,b+1)h_1(b,1)
			\nonumber\\
			&\lesssim \mathop{\sum\sum}_{n-2\geq a>b\geq 1}
			\pth{\pi_i{\mymat(j|i)}(1-{\mymat(j|i)})+{\sqrt{{\mymat(j|i)}}}{\lambda^{a-b-1}_*}}
			{\sqrt{{\mymat(j|i)}}}{\lambda^{b-1}_*}
			\nonumber\\
			&\lesssim {\sqrt{{\mymat(j|i)}}\over \gamma_*}n\pi_i{\mymat(j|i)}(1-{\mymat(j|i)})
				+{{\mymat(j|i)}\over \gamma_*^2}
			\lesssim \pth{n\pi_i{\mymat(j|i)}(1-{\mymat(j|i)})}^2+{{{\mymat(j|i)}\over \gamma_*^2}}.
			\end{align*}

			\item Bounding $\mathop{\sum\sum}_{n-2\geq a>b\geq 1}
			\abs{\EE\qth{\eta_a\eta_b^3|X_1=i}}$:
			
			\begin{align*}
			&\mathop{\sum\sum}_{n-2\geq a>b\geq 1}
			\abs{\EE\qth{\eta_a\eta_b^3|X_1=i}}
			\nonumber\\
			&= \mathop{\sum\sum}_{n-2\geq a>b\geq 1}
			h_1(a,b+1)h_3(b,1)
			\nonumber\\
			&\lesssim \mathop{\sum\sum}_{n-2\geq a>b\geq 1}
			\pth{\pi_i{\mymat(j|i)}(1-{\mymat(j|i)})
				+{\sqrt{{\mymat(j|i)}}}\lambda^{b-1}_*}
			{\sqrt{{\mymat(j|i)}}}{\lambda^{a-b-1}_*}
			\nonumber\\
			&\lesssim{{\sqrt{{\mymat(j|i)}}\over \gamma_*}n\pi_i{\mymat(j|i)}(1-{\mymat(j|i)})
				+{ {{\mymat(j|i)}\over \gamma_*^2}}}
			\lesssim \pth{n\pi_i{\mymat(j|i)}(1-{\mymat(j|i)})}^2
				+{ {{\mymat(j|i)}\over \gamma_*^2}}.
			\end{align*}
			\end{enumerate}

			\paragraph{Case IV: Single index.}
			
			Bound on $\sum_{a=1}^{n-1}\EE\qth{\eta_a^4|X_1=i}$:
			\begin{align*}
			&\sum_{a=1}^{n-1}\EE\qth{\eta_a^4|X_1=i}
			=\sum_{a=1}^{n-1}h_4(a,1)
			{\leq} {n\pi_i{\mymat(j|i)}(1-{\mymat(j|i)})}
			+{{\sqrt{{\mymat(j|i)}}\over \gamma_*}}.
			\end{align*}
			
			Combining all cases we get 
			\begin{align*}
			\EE\qth{\pth{{\mymat(j|i)}N_i-N_{ij}}^4|X_n=i}
			&\lesssim \pth{n\pi_i{\mymat(j|i)}(1-{\mymat(j|i)})}^2
				+{\sqrt{{\mymat(j|i)}}\over \gamma_*}
				+{M(j|i)\over \gamma_*^2}
				+{M(j|i)^{\frac 32}\over \gamma_*^3}
				+{M(j|i)^2\over \gamma_*^4}
			\nonumber\\
			&\lesssim \pth{n\pi_i{\mymat(j|i)}(1-{\mymat(j|i)})}^2
			+{\sqrt{{\mymat(j|i)}}\over \gamma_*}
			+{M(j|i)^2\over \gamma_*^4}
			\end{align*}
			as required.
		\end{proof}
	
	\begin{proof}[Proof of \prettyref{lmm:moment.reversebound}(iii)]
	Throughout our proof we repeatedly use the spectral decomposition \eqref{eq:spectral.decomposition} applied to the diagonal elements:
	$$M^t(i|i)=\pi_i+\sum_{v\geq 2}\lambda_v^tu_{vi}^2,
	\quad \sum_{v\geq 2}u_{vi}^2\leq 1.$$
	Write $N_i-(n-1)\pi_i=\sum_{a=1}^{n-1}\xi_a$ where $\xi_a=\indc{X_a=i}-\pi_i$. 
	For $a\geq b\geq d\geq e$,
	\begin{align}
	&\EE\qth{\xi_a\xi_b\xi_d\xi_e|X_1=i}
	\nonumber \\
	&=\EE\qth{\xi_a\xi_b\pth{\indc{X_d=i,X_e=i}
			-\pi_i\indc{X_d=i}-\pi_i\indc{X_e=i}+\pi_i^2}|X_1=i}
	\nonumber \\
	&=\EE\qth{\xi_a\xi_b\indc{X_d=i,X_e=i}|X_1=i}
	-\pi_i\EE\qth{\xi_a\xi_b\indc{X_d=i}|X_1=i}
	\nonumber \\
	&\quad -\pi_i\EE\qth{\xi_a\xi_b\indc{X_e=i}|X_1=i}
	+\pi_i^2\EE\qth{\xi_a\xi_b|X_1=i}
	\nonumber \\
	&=\EE\qth{\xi_a\xi_b|X_d=i}
	\PP\qth{X_d=i|X_e=i}\PP[X_e=i|X_1=i]
	-\pi_i\EE\qth{\xi_a\xi_b|X_d=i}\PP[X_d=i|X_1=i]
	\nonumber \\
	&\quad -\pi_i \EE\qth{\xi_a\xi_b|X_e=i}\PP[X_e=i|X_1=i]
	+\pi_i^2\EE\qth{\xi_a\xi_b|X_1=i}
	\nonumber \\
	&=\EE\qth{\xi_a\xi_b|X_d=i}
	\sth{\mymat^{d-e}(i|i)\mymat^{e-1}(i|i)
		-\pi_i\mymat^{d-1}(i|i)}
	\nonumber\\
	&\quad -\sth{\pi_i \EE\qth{\xi_a\xi_b|X_e=i}\mymat^{e-1}(i|i)
		-\pi_i^2\EE\qth{\xi_a\xi_b|X_1=i}}
	\label{eq:gub30}
	\end{align}
	Using the Markov property for any $d\leq b\leq a$, we get
	\begin{align}
	&\abs{\EE[\xi_a\xi_b|X_d=i]-\pi_i\sum_{v\geq 2}u_{vi}^2\lambda_v^{a-b}}
	\nonumber\\
	&=\abs{\EE\qth{\indc{X_a=i,X_b=i}-\pi_i\indc{X_a=i}
		-\pi_i\indc{X_b=i}+\pi_i^2|X_d=i}-\pi_i\sum_{v\geq 2}u_{vi}^2\lambda_v^{a-b}}
	\nonumber\\
	&=\abs{\mymat^{a-b}(i|i)\mymat^{b-d}(i|i)-\pi_i\mymat^{a-d}(i|i)
	-\pi_i\mymat^{b-d}(i|i)+\pi_i^2-\pi_i\sum_{v\geq 2}u_{vi}^2\lambda_v^{a-b}}
	\nonumber\\
	&=\left|\pth{\pi_i+\sum_{v\geq 2} u_{vi}^2\lambda_v^{a-b}}
	\pth{\pi_i+\sum_{v\geq 2}u_{vi}^2\lambda_v^{b-d}}
	\quad-\pi_i\pth{\pi_i+\sum_{v\geq 2}u_{vi}^2\lambda_v^{a-d}}\right.
	\nonumber\\
	&\left.\quad -\pi_i\pth{\pi_i+\sum_{v\geq 2}u_{vi}^2\lambda_v^{b-d}}
	+\pi_i^2-\pi_i\sum_{v\geq 2}u_{vi}^2\lambda_v^{a-b}\right|
	\nonumber\\
	&= \abs{
	\pth{\sum_{v\geq 2}u_{vi}^2\lambda_v^{a-b}}
	\pth{\sum_{v\geq 2}u_{vi}^2\lambda_v^{b-d}}
	-\pi_i{\sum_{v\geq 2}u_{vi}^2\lambda_v^{a-d}}}
	\nonumber\\
	&\leq 
	\lambda_*^{a-d}\pth{\sum_{v\geq 2}u_{vi}^2}
	\pth{\sum_{v\geq 2}u_{vi}^2}
	+\lambda_*^{a-d}\pi_i{\sum_{v\geq 2}u_{vi}^2}
	\leq 2\lambda_*^{a-d}.
		\label{eq:gub29}
	\end{align}
	We also get for $d\geq e$
	\begin{align}
	&\abs{M^{d-e}(i|i)\mymat^{e-1}(i|i)
	-\pi_i\mymat^{d-1}(i|i)}
	\nonumber\\
	&=\abs{\pth{\pi_i+\sum_{v\geq 2}u_{vi}^2\lambda_v^{d-e}}
	\pth{\pi_i+\sum_{v\geq 2}u_{vi}^2\lambda_v^{e-1}}
	-\pi_i\pth{\pi_i+\sum_{v\geq 2}u_{vi}^2\lambda_v^{d-1}}}
	\nonumber\\
	&=\abs{\pi_i{\sum_{v\geq 2}u_{vi}^2\lambda_v^{e-1}}
	+\pi_i{\sum_{v\geq 2}u_{vi}^2\lambda_v^{d-e}}
	+\pth{\sum_{v\geq 2}u_{vi}^2\lambda_v^{e-1}}
	\pth{\sum_{v\geq 2}u_{vi}^2\lambda_v^{d-e}}
	 -\pi_i{\sum_{v\geq 2}u_{vi}^2\lambda_v^{d-1}}}
	\nonumber\\
	&\leq 2\lambda_*^{d-1}
	+\pi_i\lambda_*^{e-1}
	+\pi_i\lambda_*^{d-e}.
	\end{align}
	This implies 
	\begin{align}
	&\abs{\EE\qth{\xi_a\xi_b|X_d=i}}
		\abs{M^{d-e}(i|i)\mymat^{e-1}(i|i)
			-\pi_i\mymat^{d-1}(i|i)}
	\nonumber\\
	&\leq \pth{\pi_i\sum_{v\geq 2}u_{vi}^2\lambda_v^{a-b}+2\lambda_*^{a-d}}
	\pth{2\lambda_*^{d-1}
		+\pi_i\lambda_*^{e-1}
		+\pi_i\lambda_*^{d-e}}
	\nonumber\\
	&\leq \pth{\pi_i\lambda_*^{a-b}+2\lambda_*^{a-d}}
	\pth{2\lambda_*^{d-1}
		+\pi_i\lambda_*^{e-1}
		+\pi_i\lambda_*^{d-e}}
	\nonumber\\
	&\leq 4\qth{\pi_i^2\lambda_*^{a-b+d-e}
	+\pi_i^2\lambda_*^{a-b+e-1}
	+\pi_i\pth{\lambda_*^{a-b+d-1}
		+\lambda_*^{a-d+e-1}
		+\lambda_*^{a-e}}
	+\lambda_*^{a-1}}
	\label{eq:gub32}
	\end{align}
	Using \eqref{eq:gub29} along with \prettyref{lmm:prob.bounds} for any $e\leq b\leq a$ we get
	\begin{align}
	&\abs{\pi_i\EE\qth{\xi_a\xi_b|X_e=i}\mymat^{e-1}(i|i)
	-\pi_i^2\EE\qth{\xi_a\xi_b|X_1=i}}
	\nonumber\\
	&\leq \pi_i\abs{\EE\qth{\xi_a\xi_b|X_e=i}}\abs{\mymat^{e-1}(i|i)-\pi_i}
	+\pi_i^2\abs{\EE\qth{\xi_a\xi_b|X_e=i}-\pi_i\sum_{v\geq 2} u_{vi}^2\lambda_v^{a-b}}\\
	\linebreak
	&+\pi_i^2\abs{\EE\qth{\xi_a\xi_b|X_1=i}-\pi_i\sum_{v\geq 2} u_{vi}^2\lambda_v^{a-b}}
	\nonumber\\
	&\leq \pi_i\qth{\pi_i\sum_{v\geq 2} u_{vi}^2\lambda_v^{a-b}
		+2{\lambda_*^{a-e}}}2\lambda_*^{e-1}
	+2\pi_i^2\lambda_*^{a-e}
	+2\pi_i^2{\lambda_*^{a-1}}
	\nonumber\\
	&\leq 2\pi_i^2{\lambda_*^{a-b+e-1}}
	+4\pi_i^2{\lambda_*^{a-e}}
	+4\pi_i^2{\lambda_*^{a-1}}.
	\end{align}
	This together with \eqref{eq:gub32} and \eqref{eq:gub30} implies 
	\begin{align}
	\begin{split}\label{eq:crossterms}
		\abs{\EE\qth{\xi_a\xi_b\xi_d\xi_e|X_1=i}}
	&\lesssim \pi_i^2
		\pth{\lambda_*^{a-b+d-e}
			+\lambda_*^{a-b+e-1}}
		+\lambda_*^{a-1} \\
	&\quad +\pi_i
		\pth{\lambda_*^{a-b+d-1} 
			+\lambda_*^{a-d+e-1}
			+\lambda_*^{a-e}}
	\end{split}
	\end{align}
	To bound the sum over $n-1\geq a\geq b\geq d\geq e\geq 1$, we divide the analysis according to the number of distinct ordered indices related variations in terms.
	
	\paragraph{Case I: four distinct indices.}
We 	sum \prettyref{eq:crossterms} over all possible $a>b>d>e$.
	\begin{itemize}
		\item For the first term,
		\begin{align*}
		&\pi_i^2
		\mathop{\sum\sum\sum\sum}_{n-1\geq a>b>d>e\geq 1}
			\lambda_*^{a-b+d-e}
		\lesssim {n\pi_i^2\over \gamma_*}
		\mathop{\sum\sum}_{n-1\geq a>b\geq 3}\lambda_*^{a-b}
		\lesssim {n^2\pi_i^2\over\gamma_*^2}.
		\end{align*}
		\item For the second term,
		\begin{align*}
		&\pi_i^2
		\mathop{\sum\sum\sum\sum}_{n-1\geq a>b>d>e\geq 1}
			\lambda_*^{a-b+e-1}
		\lesssim{n\pi_i^2\over \gamma_*}
		\mathop{\sum\sum}_{n-1\geq a>b\geq 3}\lambda_*^{a-b}
		\lesssim {n^2\pi_i^2\over\gamma_*^2}
		\end{align*}
	
		\item For the third term,
		\begin{align*}
			&\mathop{\sum\sum\sum\sum}_{n-1\geq a>b>d>e\geq 1}
			\lambda_*^{a-1}
			\lesssim \sum_{n-1\geq a\geq4}a^3\lambda_*^{a-1}
			\lesssim {1\over \gamma_*^4}.
		\end{align*}
	
		\item For the fourth term,
		\begin{align*}
		\pi_i\mathop{\sum\sum\sum\sum}_{n-1\geq a>b>d>e\geq 1}
		\lambda_*^{a-b+d-1}
		\leq {\pi_i\over \gamma_*^2} \mathop{\sum\sum}_{n-1\geq a>b\geq 3}
		\lambda_*^{a-b}
		\lesssim\frac {n\pi_i}{\gamma_*^3}
		\end{align*}
		
		\item For the fifth term,
		\begin{align*}
		&\pi_i
		\mathop{\sum\sum\sum\sum}_{n-1\geq a>b>d>e\geq 1}
		\lambda_*^{a-d+e-1}
		\lesssim\frac{\pi_i}{\gamma_*}
		\pth{\mathop{\sum\sum}_{n-1\geq a>b\geq 3}\lambda_*^{a-b}}
		\pth{\sum_{d\geq 2}^{b-1}\lambda_*^{b-d}}
		\lesssim{n\pi_i\over \gamma_*^3}.
		\end{align*}
		
		\item For the sixth term,
		\begin{align*}
		\pi_i
		\mathop{\sum\sum\sum\sum}_{n-1\geq a>b>d>e\geq 1}
		\lambda_*^{a-e}
		\lesssim
		\pi_i\pth{\mathop{\sum\sum}_{n-1\geq a>b\geq 3}\lambda_*^{a-b}}
		\pth{\sum_{d\geq 2}^{b-1}\lambda_*^{b-d}}
		\pth{\sum_{e\geq 1}^{d-1}\lambda_*^{d-e}}
		\lesssim {n\pi_i\over \gamma_*^3}.
		\end{align*}
		
		Combining the above bounds and using the fact that $ab\leq a^2+b^2$, we obtain
		\begin{align}
		&\mathop{\sum\sum\sum\sum}_{n-1\geq a>b>d>e\geq 1}
			\abs{\EE\qth{\xi_a\xi_b\xi_d\xi_e|X_1=i}}
		\lesssim   {{n^2\pi_i^2\over \gamma_*^2}+{n\pi_i\over \gamma_*^3}
				+{1\over \gamma_*^4}}
			\lesssim {{n^2\pi_i^2\over \gamma_*^2}+{1\over \gamma_*^4}}. 
			\label{eq:crossbound-4th}
		\end{align}
	\end{itemize}

	\paragraph{Case II: three distinct indices.}
	There are three cases, namely, $\xi_a\xi_b^2\xi_e$, $\xi_a\xi_b\xi_e^2$, and $\xi_a^2\xi_b\xi_e$.
	
	\begin{enumerate}
		\item Bounding $\mathop{\sum\sum\sum}_{n-1\geq a>b>e\geq 1}	\abs{\EE\qth{\xi_a\xi_b^2\xi_e|X_1=i}}$: We specialize \eqref{eq:crossterms} with $b=d$ to get
	\begin{align*}
	\abs{\EE\qth{\xi_a\xi_b^2\xi_e|X_1=i}}
	&\lesssim   \pi_i\pth{\lambda_*^{a-b+e-1}
		+\lambda_*^{a-e}}
		+\lambda_*^{a-1}.
	\end{align*}
	 Summing over $a,b,e$ we have
		\begin{align}
		&\mathop{\sum\sum\sum}_{n-1\geq a>b>e\geq 1}
		\abs{\EE\qth{\xi_a\xi_b^2\xi_e|X_1=i}}
		\nonumber\\
		&\lesssim\mathop{\sum\sum\sum}_{n-1\geq a>b>e\geq 1}
		 \sth{\pi_i\pth{\lambda_*^{a-b+e-1}+\lambda_*^{a-e}}
		 +\lambda_*^{a-1}}
		\nonumber\\
		&\lesssim {\pi_i\over \gamma_*}\mathop{\sum\sum}_{n-1\geq a>b\geq 2}
		\lambda_*^{a-b}
		+\pi_i\pth{\mathop{\sum\sum}_{n-1\geq a>b\geq 2}\lambda_*^{a-b}}
		\pth{\sum_{e\geq 1}^{b-1}\lambda_*^{b-e}}
		+\sum_{n-1\geq a\geq 3}a^3\lambda_*^{a-1}
		\nonumber\\
		&\lesssim {n\pi_i\over \gamma_*^2}+{1\over \gamma_*^3}
		\lesssim {n^2\pi_i^2\over \gamma_*^2}+\frac 1{\gamma_*^3}
		\label{eq:crossbound-3rd1}
		\end{align}
		with last inequality following from $xy\leq x^2+y^2$.
	\item Bounding $\mathop{\sum\sum\sum}_{n-1\geq a>b>e\geq 1}
		\abs{\EE\qth{\xi_a\xi_b\xi_e^2|X_1=i}}$: We specialize \eqref{eq:crossterms} with $e=d$ to get
	\begin{align*}
	\abs{\EE\qth{\xi_a\xi_b\xi_e^2|X_1=i}}
	& \lesssim \pi_i^2\lambda_*^{a-b}
		+\pi_i
		\pth{\lambda_*^{a-b+e-1}
			+\lambda_*^{a-e}}
		+\lambda_*^{a-1}.
	\end{align*}
	Summing over $a,b,e$ and applying \eqref{eq:crossbound-3rd1}, we get
	\begin{align}
	&\mathop{\sum\sum\sum}_{n-1\geq a>b>e\geq 1}
	\abs{\EE\qth{\xi_a\xi_b\xi_e^2|X_1=i}}
	\nonumber\\
	&\lesssim\mathop{\sum\sum\sum}_{n-1\geq a>b>e\geq 1}
		\sth{\pi_i^2\lambda_*^{a-b}
			+\pi_i\pth{\lambda_*^{a-b+e-1}
				+\lambda_*^{a-e}}+\lambda_*^{a-1}}
	\nonumber\\
	&\lesssim {n\pi_i^2}\mathop{\sum\sum}_{n-1\geq a>b\geq 2} \lambda_*^{a-b}
	+{n\pi_i\over \gamma_*^2}
	+{1\over \gamma_*^3}	
	\lesssim  {n^2\pi_i^2\over \gamma_*}
		+{n\pi_i\over \gamma_*^2}
		+{1\over \gamma_*^3}
	\lesssim {n^2\pi_i^2\over \gamma_*^2}
	+{1\over \gamma_*^3}.
	\label{eq:crossbound-3rd2}
	\end{align}
	
	\item Bounding $\mathop{\sum\sum\sum}_{n-1\geq a>b>e\geq 1}
		\abs{\EE\qth{\xi_a^2\xi_b\xi_e|X_1=i}}$:
		Specializing \eqref{eq:crossterms} with $a=b$ we get
		\begin{align*}
			&\abs{\EE\qth{\xi_b^2\xi_d\xi_e|X_1=i}}
			\lesssim  
			\pi_i^2\pth{\lambda_*^{d-e}
				+\lambda_*^{e-1}}
			+\lambda_*^{b-1}
			+\pi_i\pth{\lambda_*^{d-1}
				+\lambda_*^{b-d+e-1}+\lambda_*^{b-e}},
		\end{align*}
	which is equivalent to
	\begin{align*}
	&\abs{\EE\qth{\xi_a^2\xi_b\xi_e|X_1=i}}
	\lesssim  
		\pi_i^2\pth{\lambda_*^{b-e}
			+\lambda_*^{e-1}}
		+\lambda_*^{a-1}
			+\pi_i\pth{\lambda_*^{b-1}+\lambda_*^{a-b+e-1}
			+\lambda_*^{a-e}}.
	\end{align*}
	For the first, second and fourth terms 
		\begin{align*}
		\mathop{\sum\sum\sum}_{n-1\geq a>b>e\geq 1}
			\sth{\pi_i^2\pth{\lambda_*^{b-e}
				+\lambda_*^{e-1}}
			+\pi_i\lambda_*^{b-1}}
		\lesssim {\pi_i^2\over \gamma_*}\mathop{\sum\sum}_{n-1\geq a>b\geq 2}
		1
		+{n\pi_i\over \gamma_*^2}
		{\lesssim} {n^2\pi_i^2\over \gamma_*}
		+ {n\pi_i\over \gamma_*^2},
		\end{align*}
	and for summing the remaining terms we use \eqref{eq:crossbound-3rd1}, which implies
	\begin{align}
	&\mathop{\sum\sum\sum}_{n-1\geq a>b>e\geq 1}
		\abs{\EE\qth{\xi_a^2\xi_b\xi_e|X_1=i}}
	\lesssim   {{n^2\pi_i^2\over \gamma_*}
		+{n\pi_i\over \gamma_*^2}
		+{1\over \gamma_*^3}}
	\lesssim {{n^2\pi_i^2\over \gamma_*^2}
		+{1\over \gamma_*^3}}.
	\label{eq:crossbound-3rd3}
	\end{align}
	\end{enumerate}
	
	\paragraph{Case III: two distinct indices.}
	There are three cases, namely, $\eta_a^2\eta_e^2,\eta_a\eta_e^3$ and $\eta_a^3\eta_e$.
	
	\begin{enumerate}
		\item Bounding $\mathop{\sum\sum}_{n-1\geq a>e\geq 1}
	\EE\qth{\xi_a^2\xi_e^2|X_1=i}$:
	Specializing \eqref{eq:crossterms} for $a=b$ and $e=d$ we get
	\begin{align*}
	&{\EE\qth{\xi_a^2\xi_e^2|X_1=i}}
	\lesssim   \pi_i^2
		+\pi_i\pth{\lambda_*^{e-1}
			+\lambda_*^{a-e}}
		+\lambda_*^{a-1}.
	\end{align*}
	
	Summing up over $a,e$ we have
	\begin{align}
	&\mathop{\sum\sum}_{n-1\geq a>e\geq 1}{\EE\qth{\xi_a^2\xi_e^2|X_1=i}}
	\lesssim \mathop{\sum\sum}_{n-1\geq a>e\geq 1}
	\sth{\pi_i^2+\pi_i\pth{\lambda_*^{e-1}+\lambda_*^{a-e}}+\lambda_*^{a-1}}
	\lesssim   n^2\pi_i^2
		+{n\pi_i\over \gamma_*}
		+{1\over \gamma_*^2}.
	\label{eq:crossbound-2ndi}
	\end{align}

	\item Bounding $\mathop{\sum\sum}_{n-1\geq a>e\geq 1}
	\abs{\EE\qth{\xi_a\xi_e^3|X_1=i}}$:
	Specializing \eqref{eq:crossterms} for $e=b=d$ we get	
	\begin{align*}
	\abs{\EE\qth{\xi_a\xi_e^3|X_1=i}}
	\lesssim   \pi_i\lambda_*^{a-e}
		+\lambda_*^{a-1}
	\end{align*}	
	which sums up to
	\begin{align}
	&\mathop{\sum\sum}_{n-1\geq a>e\geq 1}
	\abs{\EE\qth{\xi_a\xi_e^3|X_1=i}}
	\lesssim   \pi_i\mathop{\sum\sum}_{n-1\geq a>e\geq 1}\lambda_*^{a-e}
		+\mathop{\sum\sum}_{n-1\geq a>e\geq 1}\lambda_*^{a-1}
	\lesssim  {n\pi_i\over \gamma_*}+{1\over \gamma_*^2}.
	\label{eq:crossbound-2nd2}
	\end{align}

	\item Bounding $\mathop{\sum\sum}_{n-1\geq a>e\geq 1}
	\abs{\EE\qth{\xi_a^3\xi_e|X_1=i}}$:
	Specializing \eqref{eq:crossterms} for $a=b=d$ we get
	\begin{align*}
	\abs{\EE\qth{\xi_a^3\xi_e|X_1=i}}
	&\lesssim   \pi_i\pth{\lambda_*^{a-e}
			+\lambda_*^{e-1}}
		+\lambda_*^{a-1}
	\end{align*}
	which sums up to
	\begin{align}
	&\mathop{\sum\sum}_{n-1\geq a>e\geq 1}
		\abs{\EE\qth{\xi_a^3\xi_e|X_1=i}}
	\lesssim \mathop{\sum\sum}_{n-1\geq a>e\geq 1}
	\sth{\pi_i\pth{\lambda_*^{a-e}
		+\lambda_*^{e-1}}+\lambda_*^{a-1}}	
	\lesssim   {n\pi_i\over \gamma_*}+{1\over \gamma_*^2}.
	\label{eq:crossbound-2nd3}
	\end{align}
	\end{enumerate}

	\paragraph{Case IV: single distinct index.}
	We specialize \eqref{eq:crossterms} to $a=b=d=e$ to get
	\begin{align*}
	\EE\qth{\xi_a^4|X_1=i}
	&\lesssim \pi_i+\lambda_*^{a-1}. 
	\end{align*}
	Summing the above over $a$
	\begin{align}
	\sum_{a=1}^{n-1}\EE\qth{\xi_a^4|X_1=i}
	\lesssim   {n\pi_i}
		+{1\over \gamma_*}.
	\label{eq:crossbound-ist}
	\end{align}
	Combining \eqref{eq:crossbound-4th}--\eqref{eq:crossbound-ist} and 
	using ${n\pi_i\over \gamma_*}\lesssim {n^2\pi_i^2\over \gamma_*^2}
	+{1\over\gamma_*^4}$, we get
	\begin{align*}
	&\EE\qth{\pth{N_i-(n-1)\pi_i}^4|X_1=i}
	\lesssim  {n^2\pi_i^2\over \gamma_*^2}
		+{1\over\gamma_*^4}.
	\end{align*}

\end{proof}

\section*{Acknowledgment}
The authors are grateful to Alon Orlitsky for helpful and encouraging comments and to 
Dheeraj Pichapati for providing the full version of \cite{FOPS2016}. The authors also thank David Pollard for insightful discussions on Markov chains at the initial stages of the project.

\bibliographystyle{alpha}
\bibliography{Markov_refs}

\end{document}